\documentclass[letterpaper]{article}

\usepackage{arxiv}
\AtBeginDocument{%
  \newgeometry{textheight=9in,textwidth=6.5in,top=1in,headheight=22.5pt,headsep=25pt,footskip=30pt,marginparwidth=62pt,marginparsep=7pt}%
}

\usepackage[english]{babel}
\usepackage{csquotes}
\usepackage{booktabs}
\usepackage{graphicx}
\usepackage{sidecap}
\usepackage[hypertexnames=false]{hyperref}
\hypersetup{
	colorlinks,
	linkcolor={red!50!black},
	citecolor={blue!50!black},
	urlcolor={blue!80!black}
}
\usepackage{amssymb,amsmath,amsthm}
\usepackage{mathrsfs}
\usepackage{mathtools}

\usepackage{accents}
\usepackage[normalem]{ulem}
\usepackage{dsfont}
\usepackage{diagbox,multirow}
\usepackage{hhline}
\usepackage[dvipsnames,x11names]{xcolor}
\usepackage{array}
\usepackage{colortbl}
\usepackage{subfig}
\usepackage{enumerate}
\usepackage[shortlabels]{enumitem}
\usepackage{makecell}
\usepackage{float}
\usepackage{framed}
\usepackage{varwidth}
\usepackage[export]{adjustbox}
\usepackage{xparse}
\usepackage {tikz}
\usetikzlibrary{cd}
\usetikzlibrary{arrows,decorations.pathmorphing,backgrounds,positioning,fit,matrix}
\usetikzlibrary{decorations.markings}

\newlength\mylen
\tikzset{
	bicolor/.style 2 args={
		dashed,dash pattern=on 20pt off 20pt,-,#1,
		postaction={draw,dashed,dash pattern=on 20pt off 20pt,-,#2,dash phase=20pt}
	},
}

\usepackage{etoolbox}
\apptocmd{\sloppy}{\hbadness 10000\relax}{}{}
\usepackage{amsbsy}

\newtheorem{lemma}{Lemma}[section]
\newtheorem{theorem}[lemma]{Theorem}
\newtheorem{proposition}[lemma]{Proposition}
\newtheorem{corollary}[lemma]{Corollary}
\newtheorem*{proposition*}{Proposition}
\theoremstyle{definition}
\newtheorem{definition}[lemma]{Definition}
\newtheorem{example}{Example}

\newtheorem{remark}[lemma]{Remark}

\newtheoremstyle{named}{}{}{\itshape}{}{\bfseries}{.}{.5em}{\thmnote{#3}}
\theoremstyle{named}
\newtheorem*{namedtheorem}{}

\numberwithin{equation}{section}

\newcommand{\IP}{\hyperlink{ass:IP}{\textup{\textbf{(IP)}}}}

\newcommand{\UM}{\hyperlink{ass:UM}{\textup{\textbf{(UM)}}}}
\newcommand{\B}{\hyperlink{ass:B}{\textup{\textbf{(B)}}}}
\newcommand{\BD}{\hyperlink{ass:BD}{\textup{\textbf{(BD)}}}}
\newcommand{\C}{\hyperlink{ass:C}{\textup{\textbf{(C)}}}}
\newcommand{\Cp}{\hyperlink{ass:C}{\textup{\textbf{(C\textsubscript{p})}}}}

\newcommand{\FB}{\hyperlink{ass:FB}{\textup{\textbf{(FB)}}}}
\newcommand{\EMp}{\hyperlink{ass:EM}{\textup{\textbf{(EM\textsubscript{p})}}}}

\newcommand{\Cone}{\hyperlink{ass:C}{\textup{\textbf{(C\textsubscript{1})}}}}

\newcommand{\EM}[1][]{\hyperlink{ass:EM}{\textup{\textbf{(EM\ifx&#1&\else\textsubscript{#1}\fi)}}}}
\newcommand{\SP}[1][]{\hyperlink{ass:SP}{\textup{\textbf{(S\ifx&#1&\else\textsubscript{#1}\fi)}}}}

\newcommand{\N}{\mathbb{N}}
\newcommand{\R}{\mathbb{R}}
\newcommand{\Q}{\mathcal{Q}}
\newcommand{\QD}{Q_{D}}
\newcommand{\QN}{Q_{N}}
\newcommand{\D}{\mathcal{D}}

\newcommand{\lm}{\lambda}
\newcommand{\sgn}{\operatorname{sgn}}

\makeatletter 
\newcommand{\sbullet}{%
	\hbox{\fontfamily{lmr}\fontsize{.6\dimexpr(\f@size pt)}{0}\selectfont\textbullet}}

\makeatother

\DeclareMathSymbol{\shortminus}{\mathbin}{AMSa}{"39}

\newcommand{\dom}{\textnormal{dom}} 
\newcommand{\Deg}{\operatorname{Deg}}
\newcommand{\id}{\operatorname{id}}
\newcommand{\Deltaphi}{\mathcal{L}}
\newcommand{\Deltaphip}{\mathcal{L}^{(p)}}
\newcommand{\Dep}{\Delta^{(p)}}
\newcommand{\Sep}{S_\lambda^{(p)}}
\newcommand{\ellone}{\ell^1(X,\mu)}
\newcommand{\ellp}{\ell^p(X,\mu)}
\newcommand{\deltaphim}{\mathcal{L}_{\min}}
\newcommand{\deltam}{\Delta_{\min}}

\newcommand{\Deltaphiom}{\Deltaphi_{|\Omega}}

\newcommand{\Lmin}{\mathcal{L}_{n}}

\newcommand{\Deltadir}[1][]{\Delta_{D\ifx&#1&\else,#1\fi}}

\colorlet{red2}{red!65!black}
\definecolor{azure}{rgb}{0.0, 0.5, 1.0}
\definecolor{darkpastelgreen}{rgb}{0.01, 0.75, 0.24}
\definecolor{lightgreen}{rgb}{0.56, 0.93, 0.56}
\definecolor{lightgray}{rgb}{0.83, 0.83, 0.83}
\definecolor{gray}{rgb}{0.5, 0.5, 0.5}
\definecolor{darkspringgreen}{rgb}{0.09, 0.45, 0.27}

\usepackage[style=numeric,backend=biber,giveninits=true,maxbibnames=99,sortcites=true]{biblatex}
\begin{document}

\title{On the accretivity and m-accretivity of Laplacians\\ and porous medium-type operators on graphs}
\date{}
\author{Davide Bianchi\\
	School of Mathematics (Zhuhai)\\
	Sun Yat-sen University\\
	Zhuhai 518055 (China)\\
    \textit{bianchid@mail.sysu.edu.cn}\\
\And
Matthias Keller \\
Institut für Mathematik\\
Universität Potsdam \\
    Potsdam (Germany)\\
\textit{matthias.keller@uni-potsdam.de}\\
\And
Alberto G. Setti\\
Dipartimento di Scienza e Alta Tecnologia\\
	Universit\`a dell'Insubria\\
	Como 22100 (Italy)\\
\textit{alberto.setti@uninsubria.it}\\
\And
Rados{\l}aw K. Wojciechowski\\
Department of Mathematics
	\\
	Graduate Center -- CUNY \\
    New York, NY (USA)\\
    and \\
Department of Mathematics and Computer Science
	\\
	York College -- CUNY\\
    Jamaica, NY (USA)\\
\textit{rwojciechowski@gc.cuny.edu}
}

\setcounter{tocdepth}{2}


\maketitle

\begin{abstract}
We study the accretivity and m-accretivity of Laplacian and porous medium-type operators on weighted graphs.
In particular, we give several conditions that imply these properties for
maximal operators and investigate when these operators agree with various restrictions. 
For porous medium-type operators
on $\ell^1$ and for Laplacians on $\ell^p$ for $p \in [1,\infty)$,
we show that there always exists a dense subset of the domain on which 
the maximal operator is m-accretive. 
As a consequence, we establish that
accretivity, m-accretivity and
injectivity of the shifted operator are all equivalent for these maximal operators.
Under additional conditions on the graph, we then prove that
the maximal operators are m-accretive on the entire domain,
not just a dense subset.

We also investigate minimal operators and show that they are m-accretive
if and only if the minimal and maximal operators agree and the maximal operator
is accretive. We then give some conditions
that imply this agreement.
Furthermore, for the minimal Laplacian on $\ell^p$, we show that accretivity
and m-accretivity are not equivalent. For the $\ell^2$ case, we give connections
to Markov uniqueness and essential self-adjointness.
For the $\ell^\infty$ case, we establish the equivalence of stochastic completeness at infinity, 
m-accretivity for the maximal Laplacian on $\ell^\infty$,
and m-accretivity of the minimal Laplacian on $\ell^1$.
\end{abstract}

\section{Introduction}\label{sec:intro}
We study the accretivity and m-accretivity of nonlinear operators of porous medium-type
on graphs. More specifically, we consider operators of the form 
$\Deltaphip \colon \dom(\Deltaphip)\subseteq \ellp\to \ellp$ with maximal domain
$$\dom(\Deltaphip)= \left\{u \in \ellp \mid \Phi u \in \dom (\Delta),\; \Delta\Phi u \in \ellp\right\}$$
and acting as $\Deltaphip= \Delta\Phi$
where $p\in [1,\infty]$, $\Delta$ is the formal Laplacian associated to a  graph $G=(X,w,\kappa,\mu)$, and 
 $\Phi$ is the canonical extension to a function space of a continuous, monotone
increasing map $\phi \colon \R \to \R$ with $\phi(0)=0$. Here and throughout,
monotone increasing is meant in the weak sense, that is,
$\phi(s_{1}) \leq \phi(s_{2})$ whenever $s_{1} < s_{2}$.  We refer to $\phi$
as the \emph{nonlinearity} and to $\Delta\Phi$ as the associated \emph{nonlinear
operator}.
Explicitly, $\Phi u(x) = \phi(u(x))$ and
$$\Deltaphip u (x) = \frac{1}{\mu(x)}\sum_{y \in X}w(x,y)(\phi(u(x))-\phi(u(y))) + \frac{\kappa(x)}{\mu(x)}\phi(u(x))$$
for $u \in \dom(\Deltaphip)$ and $x \in X$.
In the special case when $\phi(s)=s^m\coloneqq s|s|^{m-1}$, the associated nonlinear 
operator is called the \emph{porous medium operator} for $m>1$ and the \emph{fast diffusion operator} for $m\in (0,1)$, see \cite{vazquez2007porous,bianchi2022generalized}. When $m=1$, $\Phi$ is the identity operator $\id$ and $\Deltaphip$ acts as the formal graph
Laplacian $\Delta$. For a complete introduction to the graph setting and notations, 
we direct the reader to Section \ref{sec:notation}.

We note that the present assumptions allow $\phi$ to be constant on nontrivial intervals
 and not necessarily  surjective, in the same spirit as~\cite{benilan1981continuous}.

If $T \colon \dom(T)\subseteq E \to E$ is a linear or nonlinear operator 
on a Banach space $\mathbb{E}=(E,\|\cdot\|)$ and
$S_\lambda = \id+\lambda T$
for $\lambda \in \R$ is the shifted operator associated to $T$, 
then $T$ is said to be \emph{accretive} if
$$\| S_\lambda u - S_\lambda v \| \geq \|u-v\|$$
for all $u,v\in\dom(T)$ and all $\lambda>0$.
In particular, if $T$ is accretive, then $S_\lambda$ is injective for all $\lm>0$.
Furthermore, $T$ is said to be \emph{m-accretive} if $T$ is accretive and $S_\lambda$ is surjective for every  $\lambda >0$. 
In particular, if $T$ is m-accretive, then the operators $S_\lambda$ are bijective and the resolvents $S_\lambda^{-1}$ are non-expansive for all $\lambda>0$. In the case of Hilbert spaces, accretivity is equivalent to monotonicity, see, for example, \cite[Chapter 13.2]{deimling1985nonlinear}.

Accretivity is a potent property that ensures the existence and uniqueness of solutions to general abstract differential equations on Banach spaces, see \cite{benilan1988evolution, barbu2010nonlinear}. For symmetric linear operators on Hilbert spaces,  m-accretivity 
is equivalent
to positivity and self-adjointness \cite[Problem V.3.32]{kato2013perturbation}. Furthermore, accretivity has significant implications in analytic and asymptotic perturbation theory, see \cite[Chapter 13]{goebel1990topics} and \cite{neidhardt2018operator}, and plays a crucial role in approximation and interpolation theory. In particular, accretivity ensures that numerical approximations for the Cauchy problem via Euler discretization converge to a mild solution within certain prescribed bounds \cite{nochetto2006nonlinear,chill2024real,beurich2024interpolation}.

In the Euclidean setting, the theory of mild solutions for the generalized porous medium equation $u_t = \Delta\Phi u + f$, where $\Phi$ is a maximal monotone graph in the sense of multivalued
function theory, is well-established within the framework of nonlinear semigroup theory and 
accretive operators on $L^1(D)$
where $D \subseteq \R^n$ is an open and bounded set. 
A foundational result, due to Brezis and Strauss~\cite{BrezisStraussJFA1973}, shows that the operator $\mathcal{A}u = -\Delta v$ with domain
\[
\operatorname{dom}(\mathcal{A}) = \bigl\{u \in L^1(D) \mid \exists v \in W_0^{1,1}(D),\; \Delta v \in L^1(D),\; v(x) \in \Phi(u(x)) \text{ a.e.}\bigr\}
\]
is m-accretive  on $L^1(D)$ with dense domain. This yields existence, uniqueness, and contractivity of mild solutions via the Crandall--Liggett theorem~\cite{CrandallLiggett1971} along with a maximum principle and the existence of
an order-preserving semigroup of contractions, see V\'{a}zquez~\cite[Chapter 10]{vazquez2007porous} for a comprehensive treatment. To the best of our knowledge, there are no results
for accretivity on $L^p$ for $p > 1$. 
In contrast, the graph setting, which is the focus of the present work, requires a substantially different approach. We also note that, while the multivalued function context
is standard for the general theory on Euclidean spaces, we do not pursue
this direction here as it is not typical for the graph setting. However, all essential technical
details are already present in the case of single-valued functions.

In \cite{bianchi2022generalized}, the existence and uniqueness of mild solutions to the 
Cauchy problem for the generalized porous medium equation
on a graph with $\ellone$ initial data were established by using a tailored sequence of approximating solutions. 
A key step was to prove that $\Deltaphi=\Deltaphi^{(1)}$, 
which can
be considered as an operator with a maximal domain on $\ell^1$,
is m-accretive when restricted to a suitable dense subset 
of its domain under additional assumptions on the graph. In this paper, 
we remove all restrictions on the graph and show that there always exists
such a dense subset, see Theorem~\ref{thm:new}.
We then give sufficient conditions 
to guarantee that $\Deltaphi$ is m-accretive on 
the entire domain and not just 
a dense subset. Our main results show that this is the case
when we assume infinite measure 
of infinite paths which holds, in particular, 
if there is a uniform lower bound on the vertex measure, see Theorem~\ref{t:lf};
or when
 the edge degree is uniformly bounded and $\Phi$ preserves $\ell^1$,
see Theorem~\ref{t:accretive_bd}.
A key observation is that injectivity
of the shifted operator is already sufficient for the m-accretivity
of $\Deltaphi$, see Theorem~\ref{t:fix}.

As mentioned above, accretivity has connections to uniqueness of 
operators. In this paper, we also consider the operator obtained by
restricting the maximal $\Deltaphi$ to the finitely supported 
functions and taking the closure when this restriction is closable. 
We denote the resulting operator by $\deltaphim$ and think of $\deltaphim$
as a minimal restriction of $\Deltaphi$ when $\Deltaphi$ is closed.
It turns out, when $\Deltaphi$ is accretive, that the m-accretivity of
$\deltaphim$ is equivalent to the fact that these two operators
agree, see Lemma~\ref{l:min_equal_accretive}.
In this context, our ultimate results give conditions for $\Deltaphi = \deltaphim$ in the case
of the porous medium operator, see Theorem~\ref{t:min=max} and Corollary~\ref{c:graph_metric}.
One of the assumptions introduces a notion for a metric to be
intrinsic that is adapted to functions in $\ell^1$. 
Here, completeness properties with respect to such a metric imply 
uniqueness of operators in the spirit of the use of such metrics
to conclude essential self-adjointness and Markov uniqueness
on $\ell^2$, see \cite{HK14, HKMW13, keller2021graphs, Sch20}.

We also consider the m-accretivity of the maximal Laplacian on $\ell^p$.
We first study the closability of the Laplacian
on $\ell^p$, see Lemma~\ref{lem:EMp}, and give conditions for the maximal
Laplacian on $\ell^1$
to be equal to the minimal restriction that 
are weaker than those required for the corresponding result for $\Deltaphi$,
see Theorem~\ref{t:Laplacian_min_max}.
In parallel to $\Deltaphi$ on $\ell^1$, we then show that there always exists 
a dense subset of $\ell^p$ so that the restriction of the maximal Laplacian to this subset
is m-accretive, see Theorem~\ref{thm:new2}. As a consequence, as for $\Deltaphi$, we show that
accretivity, m-accretivity and injectivity of the shifted operator are all equivalent
for the maximal Laplacian on $\ell^p$, see Theorem~\ref{t:acc_inj}.
We also show that metric completeness with respect to an $\ell^2$-intrinsic metric
implies m-accretivity of the maximal Laplacian on all $\ell^p$ spaces for $p \in (1,\infty)$
and extend the criteria for the m-accretivity of $\Deltaphi$ on $\ell^1$ 
to various ranges of $\ell^p$ spaces, 
see Theorem~\ref{t:Lap_accretive}.

For the minimal Laplacian on $\ell^p$, we show
that this operator is accretive whenever it is defined,
see Lemma~\ref{l:min_accretivity}.
We then explore when the maximal Laplacian is equal to the minimal restriction, 
see Propositions~\ref{p:bounded_Lap}~and~\ref{p:IP_Lap} as well as 
Corollaries~\ref{c:balls_Lap}~and~\ref{c:sc_um}.
This equivalence implies the m-accretivity of the minimal restriction.
On the other hand, we give conditions under which the minimal Laplacian on $\ell^p$
is not m-accretive for all $p \in [1,\infty]$, see Theorem~\ref{t:non_minimal_m-accretivity}. 
This shows that the phenomenon of the equivalence
of accretivity and m-accretivity which was observed for the maximal operators, does not hold
for the minimal operators.

By using general theory, 
m-accretivity of the maximal Laplacian on $\ell^2$ implies both form (equivalently, Markov) uniqueness
and essential self-adjointness, which is furthermore equivalent to the fact
that the minimal and maximal Laplacians on $\ell^2$ are equal,
see Propositions~\ref{p:MU}~and~\ref{p:ESA}.
On the other hand, failure of form uniqueness
implies that the maximal Laplacian is not accretive on all $\ell^p$
spaces and thus the maximal Laplacian is not equal to the minimal restriction, 
see Theorem~\ref{t:counter}.
For the case of $p=\infty$, we show that m-accretivity of
the maximal Laplacian on $\ell^\infty$ is equivalent to stochastic completeness
at infinity which is also equivalent to the m-accretivity
of the minimal Laplacian on $\ell^1$, see Theorem~\ref{thm:stoch_comp_maccr}. As a consequence,
when a graph is stochastically complete at infinity and the measure
is uniformly bounded from below, the maximal Laplacians on $\ell^p$
are m-accretive for all $p \in [1,\infty]$ and are equal to the minimal
Laplacians for $p \in [1, \infty)$, see Corollary~\ref{c:sc_um}.

To the best of our knowledge, m-accretivity of (Schr\"odinger operators driven by) the graph Laplacian was studied for the first time in \cite{milatovic2015maximal} for the case of vector  bundles over graphs.  
Furthermore, \cite{milatovic2015maximal} gives some criteria for the equality of the minimal and maximal operators.
For the case of non-symmetric graph Laplacians, see \cite{anne2020m}.
On the other hand, form uniqueness, essential self-adjointness, and stochastic
completeness are all well-studied properties
for graph Laplacians; see, e.g., 
\cite{BG15, CTT11, Gol14, HKLW12, HK14, HKMW13, HKMRW, HMW21, IKMW25, Jor08,
JP11, KL12, keller2021graphs, KMW25, Mas09, Sch20, Tor10, Web10, Woj08, Woj09, Woj21} among other works. 
Furthermore, there is a rapidly growing literature on nonlinear equations on graphs; we
mention \cite{ma2022porous,punzo2025semilinear, GLY16a, GLY16b, GMP, HX23, LSZ23, LY21, LW17, KS18,HM15,berchio2026semilinear,berchio2026fractional} as a representative sample.

We structure the paper as follows.
In Section~\ref{sec:notation}, we 
introduce basic definitions and give some consequences of additional assumptions
on the graph. We also prove a comparison principle which will be useful for subsequent considerations
and discuss the accretivity of formal operators on finitely supported functions. 
Section~\ref{sec:main} is the heart of the paper.
There, we discuss the accretivity and m-accretivity of the maximal operator
$\Deltaphi$ as well as when it agrees with some of its restrictions.
In Section~\ref{s:Laplacian_accretivity}, we focus on the maximal Laplacian on $\ell^p$
as well as restrictions, extending many results for $\Deltaphi$ on $\ell^1$
to Laplacians on $\ell^p$ and show how the minimal operator
may not be m-accretive.
In Section~\ref{sec:Laplacian}, we discuss connections between form 
uniqueness, essential self-adjointness and accretivity for the Laplacian on $\ell^2$.
In Section~\ref{s:sc_and_acc}, we give connections between stochastic completeness
and accretivity on $\ell^\infty$ and $\ell^1$.
Along the way, we  
strengthen one of the main results
in~\cite{bianchi2022generalized}. 
We give additional details on the connection to this previous result
in Appendix~\ref{sec:proof_Theo1}.

\section{Notation and preliminaries}\label{sec:notation}
In this section we introduce basic notations and the main
players. In particular, we first introduce graphs as well as
Laplacian and porous medium-type operators. We then prove
a comparison principle which gives uniqueness of solutions
for equations involving porous medium-type operators. 
This uniqueness will be crucial for future considerations.
We then introduce some additional assumptions on graphs
which will be used when we establish our main results 
as well as introduce the notions of accretivity
and m-accretivity for general operators on Banach
spaces. Finally, we discuss the accretivity of formal
operators on finitely supported functions.

We start with notations for the identity operator  as well as the Nemytskii operator which arises from
composition of functions.
Given a set $X$, we let $C(X)=\{ f\colon X \to \R\}$ denote the set
of all real-valued functions on $X$.
If $E\subseteq C(X)$, 
we denote by $\operatorname{id}\colon E \to E$ the identity operator. 
If $\phi \colon \dom\left(\phi\right)\subseteq \R \to \R$ is a function, then we denote by the capital letter $\Phi$ the canonical extension of $\phi$ to $E$, 
i.e., the operator $\Phi \colon \dom\left(\Phi\right)\subseteq E \to C(X)$ 
given by
\begin{align*}
	\dom\left(\Phi\right)&= \left\{ u \in E \mid u(x) \in \dom\left(\phi\right) \textup{ for all } x \in X  \right\}\\
	\Phi u(x)&= \phi(u(x)).
\end{align*}

If $u, v \in C(X)$, we write $u\geq v$ if $u(x)\geq v(x)$ for every $x \in X$. All other ordering symbols are defined accordingly. We say that a function $u$ is \emph{positive} if $u \geq 0$
and \emph{strictly positive} if $u>0$, while $u$ is \emph{negative}, 
resp.~\emph{strictly negative}, if $-u$ is positive, resp.~strictly positive. Otherwise, we say that a function
is \emph{sign-changing}.

\subsection{Graphs and function spaces}
We now introduce graphs and some associated function spaces.
For a detailed introduction to the graph setting as presented here, see \cite{keller2021graphs}.

\begin{definition}[Graph]
	A \textit{graph} is a quadruple $G=(X,w,\kappa,\mu)$ given by
	\begin{itemize}
		\item  a countable set of \emph{vertices} $X$,
		\item a positive \emph{edge-weight} function $w\colon X\times X \to [0,\infty)$,
		\item  a positive \emph{killing term} $\kappa \colon X \to [0,\infty)$,
		\item a strictly positive \emph{vertex measure} $\mu \colon X \to (0,\infty)$
	\end{itemize}
	where the edge-weight function $w$ satisfies:
	\begin{enumerate}
		\item[(1)]\label{assumption:symmetry} Symmetry: $w(x,y)=w(y,x)$ for every $x,y \in X$.
		\item[(2)]\label{assumption:loops} No loops: $w(x,x)=0$ for every $x \in X$.
		\item[(3)]\label{assumption:degree} Local summability: $\sum_{y\in X} w(x,y) < \infty$ for every $x \in X$.
	\end{enumerate}
\end{definition}
Whenever the vertex set $X$ is finite, we call $G$ a \emph{finite graph}.
We think of $x, y \in X$ with $w(x,y)>0$ as being \emph{connected}
by an edge with weight $w(x,y)$, write $x \sim y$, and call $x$ and $y$ 
\emph{neighbors} in this case.
We say that a graph is \emph{locally finite} if every vertex
has finitely many neighbors, i.e.,
$| \{ y \mid y \sim x \}| < \infty$
for every $x \in X$. We stress that we do not assume local finiteness in general. 
We let 
$$\Deg(x)= \frac{1}{\mu(x)}\left( \sum_{y \in X}w(x,y) + \kappa(x)\right)$$
denote the \emph{degree} of $x \in X$ and note that the local summability assumption implies this sum is always finite.

A \emph{path} in a graph is a sequence (either finite or infinite) of distinct vertices $(x_n)$ such that
$x_n \sim x_{n+1}$ for all $n$ such that $n$ and $n+1$ are in the index set, that is, all consecutive vertices
are neighbors. 
We write $\gamma=(x_n)$ for a path with $\gamma(n)=x_n$.
We assume that all graphs are \emph{connected}
in the sense that for any two distinct vertices $x, y \in X$, there exists a path
that starts at $x$ and ends at $y$.
A subset $U \subseteq X$ \emph{induces a subgraph} by considering the restriction of 
$w$ to $U \times U$ and the restrictions of $\kappa$ and $\mu$ to $U$.
We then call $U \subseteq X$ \emph{connected} if the subgraph
$U$ induces is connected, i.e.,
if for any two distinct vertices $x, y \in U$,
there exists a path of vertices in $U$ that starts at $x$ and ends at $y$.
A sequence of finite connected subsets $(X_n)$ is called an \emph{exhaustion} of $G$
if $X_n \subseteq X_{n+1}$ for all $n$ and $X=\bigcup_n X_n$.

When a path consists of finitely many vertices, we say that the \emph{length}
of the path is the number of edges in the path.
The \emph{combinatorial graph distance} is the length of the
shortest path between two distinct vertices $x, y \in X$. This is denoted by $d(x,y)$
where we additionally let $d(x,x)=0$.

We let $C_c(X)$ denote the finitely
supported functions in $C(X)$. We introduce the usual $\ellp$ spaces by
\begin{align*}
\ell^p(X,\mu)&= \{ f \in C(X) \mid \sum_{x \in X} |f(x)|^p \mu(x)< \infty \} \qquad	\mbox{for } p \in [1,\infty)\\
\ell^\infty(X) &= \{ f\in C(X) \mid \sup_{x \in X}
|f(x)| < \infty\} \qquad \qquad \mbox{for } p=\infty
\end{align*}
	with  norms
$$
\|f\|_p= \begin{cases}
	\left(\sum_{x\in X} |f(x)|^p\mu(x)\right)^{{1}/{p}} & \mbox{for } p \in [1,\infty)\\
	\sup_{x\in X}|f(x)| & \mbox{for } p =\infty
\end{cases}
$$	
and remark that the $\ell^2$ norm is induced by the inner product
$$
\langle f, g \rangle = \sum_{x\in X} f(x)g(x)\mu(x)
$$
giving a Hilbert space structure to $\ell^2(X,\mu)$.

\subsection{The formal Laplacian and porous medium-type operators}\label{SSect: The formal Laplacian}
We now introduce the family of operators that we will consider. More specifically,
we first let 
$$\dom(\Delta) = \left\{ u\in C(X) \mid \sum_{y \in X} w(x,y)|u(y)| < \infty
\textup{ for all } x\in X \right\} $$
denote the domain of the \emph{formal Laplacian}.
For $u \in \dom(\Delta)$, we let the formal Laplacian act as
$$\Delta u(x) = \frac{1}{\mu(x)} \sum_{y \in X} w(x,y) (u(x)-u(y)) +
\frac{\kappa(x)}{\mu(x)} u(x)$$
for all $x \in X$.

Next, we define the formal porous medium-type operator. 

We let $\phi \colon \R \to \R$ be continuous and monotone increasing, where
monotone increasing is always understood in the weak sense, i.e.,
$\phi(s_{1}) \leq \phi(s_{2})$ whenever $s_{1} < s_{2}$. We further assume
$\phi(0) = 0$ and let $\Phi$ be the canonical extension of $\phi$
to $C(X)$ given by $\Phi u = \phi \circ u$.
Note that in this case $\dom(\Phi) = C(X)$.
 With the nonlinearity $\phi$, we then let
$$\dom(\Delta \Phi) = \{ u \in C(X) \mid \Phi u \in \dom(\Delta) \}$$
denote the domain of \emph{formal porous medium-type operator} and for $u \in \dom(\Delta \Phi)$
and $x \in X$, we let the formal operator act as
$$\Delta \Phi u(x) = \Delta (\phi u)(x)= \frac{1}{\mu(x)}\sum_{y \in X}w(x,y)(\phi(u(x))-\phi(u(y))) + \frac{\kappa(x)}{\mu(x)}\phi(u(x)).$$

Finally, we introduce some restrictions of the formal porous medium-type operator. 
More specifically, we let 
$$\dom(\Deltaphip)= \{u \in \ellp \cap \dom (\Delta \Phi) \mid \Delta \Phi u \in \ellp\}$$ 
with
$$\Deltaphip u(x) = \Delta \Phi u(x)$$
for all $u \in \dom(\Deltaphip)$ and $x \in X$.
We refer to $\Deltaphip$ as a maximal restriction of $\Delta \Phi$ on $\ell^p(X,\mu)$.

As $p=1$ will be the case of greatest interest for the porous medium-type operator,
we will let $\Deltaphi=\Deltaphi^{(1)}$ throughout.

We note that $C_c(X) \subseteq \dom(\Deltaphi)$ 
as $\Phi(C_c(X)) \subseteq C_c(X)$ and
$\Delta(C_c(X)) \subseteq \ell^1(X,\mu)$ by a direct calculation.

When $\phi(s)=s^m\coloneqq s|s|^{m-1}$, the associated operator is called the \emph{porous medium operator} for $m>1$ and the \emph{fast diffusion operator} for $m\in (0,1)$. 
For $m=1$, we recover the maximal Laplacian. In this case, we write $\Dep$ for $\Deltaphip$, i.e.,
$\Dep$ acts as $\Delta$ on
$$\dom(\Dep)=\{ u \in \ell^p(X,\mu) \cap \dom(\Delta) \mid \Delta u \in \ell^p(X,\mu) \}. $$

\subsection{Comparison principle}
We now expand a comparison principle found in \cite[Theorem~A.2]{bianchi2022generalized}, see also~\cite[Lemma~3.5~and~Corollary~3.6]{schmidt2025nonlinear}. 
Here we add additional possibilities in Case~3 
and a new Case 4 which is useful for considering $\ell^\infty$.
We also include a possible rescaling of the identity operator by a strictly
positive function $\sigma$ where we write $\sigma$ for $\sigma \cdot \id$, i.e.,
$(\sigma u)(x)=\sigma(x)u(x)$ for all $u \in C(X)$ and $x \in X$.

In essence, this result gives conditions so that the operator $(\id + \lambda\Delta \Phi)^{-1}$
is  order preserving, more specifically, maps positive functions to positive functions.
In order to optimize the statement, 
we formulate the result with the minimal assumptions on $\phi$
and $\psi$ needed to make the proof work, though, for most applications, 
$\phi$ will satisfy all of the assumptions
for the porous medium-type operator and $\psi$ will be the identity.
 In the ``moreover'' statement below, we impose two
additional assumptions: that $\phi$ is strictly positive on $(0,\infty)$
and that $\psi(0)\leq0$. This statement is not needed for any of our
subsequent results.

\begin{theorem}[Sign preservation]\label{thm:min_principle}
Let $\lambda>0$, $\sigma \colon X \to (0, \infty)$ and $\phi, \,\psi \colon \R \to \R$ be monotone increasing, with
$\phi(0)=0$ and $\psi(t)<0$ for $t<0$.
Let $u \in \dom(\Delta \Phi)$ satisfy 
$$\left( \sigma \Psi + \lambda\Delta \Phi\right)u \geq 0.$$ 
We consider four cases:

\begin{itemize}
	\item[\emph{\textbf{Case 1:}}]
	There exists $x_0 \in X$ such that $u$ attains a minimum at $x_0$, i.e.,
	$$u(x_{0})=\inf_{x\in X} u(x) > -\infty.$$
	
	\item[\emph{\textbf{Case 2:}}] $G$ does not contain any infinite paths. 	

    \item[\emph{\textbf{Case 3:}}] For every infinite path $(x_{n})$, we have $\sum_{n} \mu(x_{n})=\infty$ and there exists  a monotonically increasing  function $\gamma \colon [0,\infty) \to [0,\infty)$ with  $\gamma(t)>0$ for $t>0$ such  that 
    $$
    \sum_{n}\gamma \bigl( |u(x_{n})| \bigr) \mu(x_{n})<\infty .
    $$
  \item[\emph{\textbf{Case 4:}}] The function $u$ is bounded below 
  and, for every infinite path $(x_{n})$, we have 
    $$\sum_{n} \frac{\sigma(x_n)}{\Deg(x_{n})}=\infty.$$
\end{itemize}    
	 In all four cases, $u \geq 0$. 
     
     Moreover, in all cases,
     { assuming additionally that $0=\phi(0)<\phi(t)$ for all $t>0$ and $\psi(0)\leq 0$}, if $u(x)=0$ for some $x\in X$, then  $u= 0$. 
\end{theorem}
\begin{proof}
{
We first show that $u\geq0$ by contradiction. Thus, suppose that there exists $x_0\in X$ such that
$u(x_0)<0$ so that $\Phi u(x_0)\leq0$ and $\Psi u(x_0)<0$ by the
assumptions on $\phi$ and $\psi$.
If $x_0$ has no neighbors, then
\[
 \Delta\Phi u(x_0)=\frac{\kappa(x_0)}{\mu(x_0)}\Phi u(x_0)\leq0
\]
and hence $\left(\sigma\Psi+\lambda\Delta\Phi\right)u(x_0)<0$, which
contradicts the hypothesis. Therefore, $x_0$ has at least one neighbor and
$\Deg(x_0)>0$.}

Rewriting the inequality $0 \leq \left( \sigma \Psi + \lambda\Delta \Phi\right)u(x_0)$ we get
\[
\frac{\lambda}{\mu(x_0)} 
\sum_{y \in X} w(x_0,y) \Phi u(y) \leq \sigma(x_0) \Psi u(x_0) + \lambda \Deg(x_0) \Phi u(x_0)
\]
or, equivalently,
\begin{equation}\label{eq:ineq_1}
\sum_{y \in X} w(x_0,y) \Phi u(y) \leq { \Deg(x_0)\mu(x_0)
\left[\frac{\sigma(x_0)}{\lambda \Deg(x_0)} \Psi u(x_0) +  \Phi u(x_0) \right].}
\end{equation}
{ We claim that there exists $x_1\sim x_0$ such that}
\begin{equation}
\label{eq:comp-ineq}
{
\Phi u(x_1) \leq \frac{\sigma(x_0)}{\lambda \Deg(x_0)}\Psi u(x_0) + \Phi u(x_0)=:A_{x_0}.}
\end{equation}
{
Indeed, if $\Phi u(y)>A_{x_0}$
for all $y \sim x_0$, then, since  $\sum_{y \in X} w(x_0,y) \leq \Deg(x_0)\mu(x_0)$
and
$A_{x_0}<\Phi u(x_0)\leq0$, we get
\[
\sum_{y \in X}w(x_0,y) \Phi u(y)> A_{x_0} \sum_{y\in X} w(x_0,y) \geq
\Deg(x_0)\mu(x_0)A_{x_0}
\]
which contradicts \eqref{eq:ineq_1}.}

In particular, \eqref{eq:comp-ineq} implies
{
\[
\Phi u(x_1)\leq A_{x_0} < \Phi u(x_0) \quad \textup{ and, therefore,} \quad u(x_1) < u(x_0)
\]
since $\phi$ is monotone increasing.}

{ The above descent argument proves that every vertex $x_0$ such that
$u(x_0)<0$ must have a neighbor $x_1\sim x_0$ such that $u(x_1)<u(x_0)$.}
This gives an immediate contradiction in Cases 1 and 2.

For Case 3, we iterate the argument to get an infinite path $(x_n)$ such that
$u(x_{n+1}) < u(x_n)$. 
Therefore, 
$\gamma\bigl(|u(x_n)|\bigr)\leq \gamma\bigl(|u(x_{n+1})|\bigr)$ for all $n$ and, thus,
\[
\sum_n \gamma\bigl(|u(x_n)|\bigr) \mu(x_n) \ge \gamma\bigl(|u(x_0)|\bigr) \sum_n \mu(x_n) = \infty
\]
which gives a contradiction.

For Case 4, we iterate the estimate \eqref{eq:comp-ineq} 
to get an infinite path $(x_n)$ such that
\begin{align*}
\Phi u(x_{n+1}) \leq \frac{\sigma(x_n)}{\lambda \Deg(x_n)}\Psi u(x_n) + \Phi u(x_n)
\leq \sum_{k=0}^n\frac{\sigma(x_k)}{\lambda \Deg(x_k)} \Psi u(x_k)+\Phi u(x_0) { \leq }\frac{\Psi u(x_0)}{\lambda} \sum_{k=0}^n \frac{\sigma(x_k)}{\Deg(x_k)}.
\end{align*}
As $\sum_k \sigma(x_k)/\Deg(x_k) = \infty$ and $\Psi u(x_0)<0$, we get that $\Phi u$
is not bounded below, contradicting the assumption that $u$ is bounded below.

This establishes $u \geq 0$ in all four cases. 

Assume now that $\phi(0)=0<\phi(t)$ for every $t>0$ and  $\psi(0)\leq 0$ and suppose that $u(x)=0$  for some $x \in X$. 
Since we have
already established $u(y) \geq 0$, and thus  $\Phi u(y)\geq 0$, for all $y$ the inequality 
$\left( \sigma\Psi  + \lambda\Delta \Phi\right)u(x) \geq 0$ gives
$$0 \leq \frac{\lambda}{\mu(x)} \sum_{y \in X} w(x,y) \Phi u(y)\leq \sigma(x) \Psi u(x) + \lambda \Deg(x) \Phi u(x) =\sigma(x) \psi(0)
\leq 0.$$
Hence, $\Phi u(y)=0$  and, therefore,  $u(y)=0$ for all $y \sim x$. 
Iterating the argument and using connectedness
gives $u = 0$. 
\end{proof}

{ Adapting the proof of the above theorem, we obtain the following
comparison result, which in turn yields uniqueness of solutions under the stated
hypotheses.}

\begin{theorem}[Comparison principle]\label{thm:min}
Let $\lambda>0$, $\sigma \colon X \to (0, \infty)$ and { $\phi\colon\R\to\R$ be continuous and monotone increasing with $\phi(0)=0$}. 
Let $u_k \in \dom(\Delta \Phi)$ and  $g_k \in C(X)$ satisfy
	$$
	\left(\sigma + \lambda\Delta \Phi \right) u_k = g_k \qquad \mbox{ for } k=1,2.
	$$
{ Set
\[
 u\coloneqq u_1-u_2
 \qquad \text{and} \qquad
 v\coloneqq \Phi u_1-\Phi u_2.
\]
If $g_1\geq g_2$ and $G$ and $v$ satisfy the assumptions of one of the Cases~1--3 of
Theorem~\ref{thm:min_principle}, then
$u_1\geq u_2$. The same conclusion holds in Case~4 provided that, for 
every $c>0$, there exists $\varepsilon_c>0$ such that
\begin{equation}\label{extra_assumption_case4}
 v(x)\leq-c
 \quad\Longrightarrow\quad
 u(x)\leq-\varepsilon_c \qquad \text{for all } x \in X.
\end{equation}} 

In particular, if  $u_1$ and  $u_2$ are in $\ell^p(X,\mu)$ for some $p \in [1,\infty)$, then { $v$ satisfies the summability condition in
Case~3}. 
If $u_1$ and $u_2$ are in $\ell^\infty(X)$, { then $v$ is bounded and satisfies
the additional condition \eqref{extra_assumption_case4}}.
It follows that 
\begin{itemize}
\item[\textup{(i)}] If $\sum_n \mu(x_n)=\infty$ for every infinite path $(x_n)$, then
$\sigma + \lambda\Delta \Phi$ is injective on $\ell^p(X,\mu) \cap \dom(\Delta \Phi)$
for all $p\in[1,\infty)$ { and its inverse is order preserving on its range}. 
\item[\textup{(ii)}] If $\sum_{n} \sigma(x_n)/\Deg(x_{n})=\infty$ for every infinite path $(x_n)$,
then $\sigma + \lambda\Delta \Phi$ is injective on
{ $\ell^\infty(X)\cap\dom(\Delta\Phi)$ and its inverse is order preserving on its range}.
\end{itemize}

\end{theorem}
\begin{proof}
{
Notice that   $v\in\dom(\Delta)$, and subtracting the  equations satisfied by $u_1$ and $u_2$ gives
\begin{equation}
\label{eq:comparison-differences}
 \sigma u+\lambda\Delta v=g_1-g_2\geq0. 
\end{equation}
Moreover, monotonicity of $\phi$ implies
\begin{equation}
\label{eq:weak-sign-compatibility}
 \{v<0\}\subseteq\{u<0\}\subseteq\{v\leq0\}. 
\end{equation}

We now adapt the descent step in the proof of Theorem~\ref{thm:min_principle}
to show
that if $x_0\in X$ satisfies $u(x_0)<0$, then there exists
$x_1\sim x_0$ such that $v(x_1)<v(x_0) \leq 0$.
Indeed, let $x_0\in X$ satisfy
$u(x_0)<0$, so that $v(x_0)\leq0$ by \eqref{eq:weak-sign-compatibility}. If $x_0$ has no neighbors, then
\[
 \Delta v(x_0)=\frac{\kappa(x_0)}{\mu(x_0)}v(x_0)\leq0
\]
which contradicts \eqref{eq:comparison-differences} at $x_0$ since $u(x_0)<0$.
Thus, $x_0$ has at least one neighbor and $\Deg(x_0)>0$. Rearranging
\eqref{eq:comparison-differences}, we obtain
\begin{equation}
\label{eq:ineq_1bis}
 \sum_{y\in X}w(x_0,y)v(y)
 \leq
 \Deg(x_0) \mu(x_0)\left[ \frac{\sigma(x_0)}{\lambda\Deg(x_0)}u(x_0)
 +v(x_0)\right]. 
\end{equation}
The same weighted-average argument used after \eqref{eq:ineq_1} in the
proof of Theorem~\ref{thm:min_principle}, with $\Phi u$ and $\Psi u$
there replaced by $v$ and $u$, applied to \eqref{eq:ineq_1bis} shows that
there exists $x_1\sim x_0$ such that
\begin{equation}
\label{eq:v-ineq}
 v(x_1)
 \leq v(x_0)+\frac{\sigma(x_0)}{\lambda\Deg(x_0)}u(x_0)=:A_{x_0}
 <v(x_0)\leq0. 
\end{equation}
In particular, $v(x_1)<0$ and therefore $u(x_1)<0$ by
\eqref{eq:weak-sign-compatibility}, so the descent step can be
iterated.

We now show that $u \geq 0$. Suppose, by contradiction, that $u(z)<0$ at some vertex $z$. By the preceding
descent step, there exists $x_0\sim z$ such that
\[
v(x_0)<v(z)\leq0.
\]
Thus, $v(x_0)<0$ and $u(x_0)<0$ by
\eqref{eq:weak-sign-compatibility}. In
Case~1, applying~\eqref{eq:v-ineq} at a point where
$v$ attains its negative minimum gives a still smaller value. In
Case~2, iteration gives an infinite path of descent, since $v$ strictly decreases
and hence all its vertices are distinct. In Case~3, the resulting
infinite path $(x_n)$ satisfies $v(x_{n+1})<v(x_n)<0$, and therefore
\[
 \sum_n\gamma\bigl(|v(x_n)|\bigr)\mu(x_n)
 \geq
 \gamma\bigl(|v(x_0)|\bigr)\sum_n\mu(x_n)=\infty
\]
which gives a contradiction.

In Case~4, let $c=-v(x_0)>0$. Along the descending path we have
$v(x_n)\leq-c$, so~\eqref{extra_assumption_case4} gives
$u(x_n)\leq-\varepsilon_c$. Iterating
\eqref{eq:v-ineq} yields
\[
 v(x_{n+1})
 \leq v(x_0)-\frac{\varepsilon_c}{\lambda}
 \sum_{k=0}^n\frac{\sigma(x_k)}{\Deg(x_k)}
 \to -\infty
\]
as $n \to \infty$, contradicting the fact that $v$ is bounded below.
Thus, the assumption $u(z)<0$ leads to a contradiction in Cases~1--3
and also in Case~4 under \eqref{extra_assumption_case4}. Hence, $u\geq0$.
}
{ Thus, we have shown $u_1\geq u_2$ in Cases~1--3 and in
Case~4 under condition~\eqref{extra_assumption_case4}.

We now assume that $u_1,u_2 \in \ell^p(X,\mu)$ for some $p \in [1,\infty)$
and show that $v$ satisfies the summability condition in Case 3 and that (i) holds.}
{
Define $\alpha\colon[0,\infty)\to[0,\infty)$ by
\[
 \alpha(t):=\phi(t)-\phi(-t).
\]
The assumptions on $\phi$ imply directly that $\alpha$ is continuous,
positive and monotone increasing, with $\alpha(0)=0$.
We define now the generalized  inverse by
\[
 \beta(s):=\inf\{t\geq0 \mid \alpha(t)\geq s\}
 \qquad \text{for } s\geq0,
\]
where $\inf\emptyset:=\infty$, and let
\[
 \gamma(s):=\min\{1,\beta(s)^p\}.
\]
Here and below, $\min\{1,\infty\}:=1$. Both $\beta$ and $\gamma$ are
monotone increasing and $\beta(0)=\gamma(0)=0$.

For $s>0$, continuity of $\alpha$ at zero gives the existence of $\delta_s>0$ such that
$\alpha(t)<s$ for $0\leq t<\delta_s$. Hence, $\beta(s)\geq\delta_s$ if
the defining set for $\beta(s)$ is nonempty, whereas $\beta(s)=\infty$
otherwise. Thus, $\gamma(s)>0$ for every $s>0$. Moreover, $t$ is
admissible in the definition of $\beta(\alpha(t))$, so
$\beta(\alpha(t))\leq t$ and, hence,
\begin{equation}\label{eq:gamma-alpha-estimate}
 \gamma(\alpha(t))
 =
 \min\{1,\beta(\alpha(t))^p\}
 \leq\min\{1,t^p\}
 \leq t^p.
\end{equation}

Now, for every $x\in X$, let
\[
 M(x):=\max\{|u_1(x)|,|u_2(x)|\}.
\]
Both $u_1(x)$ and $u_2(x)$ belong to $[-M(x),M(x)]$. Since $\phi$ is
monotone increasing, both $\phi(u_1(x))$ and $\phi(u_2(x))$ belong to
\[
 [\phi(-M(x)),\phi(M(x))].
\]
It follows that
\begin{align*}
 |v(x)|
 &=
 |\phi(u_1(x))-\phi(u_2(x))|\leq
 \phi(M(x))-\phi(-M(x))
 =
 \alpha(M(x)).
\end{align*}
Since $\gamma$ is monotone increasing, we obtain from
\eqref{eq:gamma-alpha-estimate} that
\[
 \gamma(|v(x)|)
 \leq
 \gamma(\alpha(M(x)))
 \leq
 M(x)^p.
\]
Finally,
\[
 M(x)^p
 =
 \max\{|u_1(x)|^p,|u_2(x)|^p\}
 \leq
 |u_1(x)|^p+|u_2(x)|^p.
\]
Therefore, for every infinite path $(x_n)$,
\begin{align*}
 \sum_n\gamma(|v(x_n)|)\mu(x_n)
 \leq
 \|u_1\|_p^p+\|u_2\|_p^p
 <\infty.
\end{align*}
Thus, $v$ satisfies the summability condition in Case~3 of
Theorem~\ref{thm:min_principle}. If every infinite path has infinite
measure, then all of the hypotheses of Case~3 are satisfied.
{ The statements of (i) now follow from what we have proved above.

We now assume $u_1,u_2 \in \ell^\infty(X)$ and show $v \in \ell^\infty(X)$, 
that $u$ and $v$ satisfy \eqref{extra_assumption_case4} and (ii) holds.}
Let
\[
 M:=\max\{\|u_1\|_\infty,\|u_2\|_\infty\}.
\]
Then, $u_1(x),u_2(x)\in[-M,M]$ for every $x\in X$ and $v$ is
bounded since $\phi$ is continuous on $[-M,M]$.

Fix $c>0$. By the uniform continuity of $\phi$ on $[-M,M]$, there exists
$\varepsilon_c>0$ such that
\[
 |r-s|<\varepsilon_c
 \quad\Longrightarrow\quad
 |\phi(r)-\phi(s)|<c
 \qquad\text{for all }r,s\in[-M,M].
\]
Equivalently, by contraposition,
\[
 |\phi(r)-\phi(s)|\geq c
 \quad\Longrightarrow\quad
 |r-s|\geq\varepsilon_c.
\]
If $v(x)\leq-c$, then
\[
 \phi(u_2(x))-\phi(u_1(x))=-v(x)\geq c>0.
\]
Since $\phi$ is monotone increasing, this also implies $u_2(x)>u_1(x)$.
Consequently,
\[
 u_2(x)-u_1(x)
 =
 |u_2(x)-u_1(x)|
 \geq\varepsilon_c
\]
and, hence,
\[
 u(x)=u_1(x)-u_2(x)\leq-\varepsilon_c.
\]
Thus, the additional condition~\eqref{extra_assumption_case4} is
satisfied.

The preceding comparison gives order preservation on the range.
Applying it in both directions when $g_1=g_2$ gives injectivity.}
{ This proves the statements in (ii) and completes the proof.}

\end{proof}

\subsection{Additional assumptions and first consequences}
In what follows, we will consider some additional assumptions on the graph 
in order to establish the m-accretivity of porous medium-type operators. 
Here, we introduce the conditions and derive some easy consequences.

More specifically, we now introduce the following additional assumptions on $G$: 
\hypertarget{ass:LF}{}\hypertarget{ass:FP}{}\hypertarget{ass:INT}{}%
\begin{enumerate}
    \item[$\IP$]\hypertarget{ass:IP}{} $\sum_n\mu(x_n) = \infty$ for every infinite path $(x_n)$. \hfill(``Infinite measure of infinite paths'')
	\item[$\UM$]\hypertarget{ass:UM}{} $\inf_{x \in X}\mu(x)>0$. \hfill(``Uniformly positive measure'')
	\item[$\B$]\hypertarget{ass:B}{} $\sup_{x \in X}\frac{\sum_{y \in X}w(x,y)}{\mu(x)}<\infty$. \hfill(``Bounded edge degree'')
    \item[$\BD$]\hypertarget{ass:BD}{} $\sup_{x \in X}\frac{\sum_{y \in X}w(x,y)+\kappa(x)}{\mu(x)} = \sup_{x \in X}\Deg(x) <
    \infty$. \hfill(``Bounded degree'')
    \item[$\EMp$]\hypertarget{ass:EM}{} $\frac{w(x,\cdot)}{\mu(\cdot)}\in \ell^{p}(X,\mu)$ for every $x \in X$.  \hfill(``$p$-edge-measure'')
\end{enumerate}
As we will most often be concerned with \EM[p] for the case of $p=\infty$,
we write \EM[] for \EM[$\infty$].

We note that $\IP$ is trivially satisfied if either $G$
does not have infinite paths or $\UM$ holds.
Furthermore, 
$\BD$ is equivalent to
$\B$ and $\kappa/\mu \in \ell^\infty(X)$.

For $\phi \colon \R \to \R$ and $\Phi$, the extension of $\phi$ to $\ell^p(X,\mu)$ for $p \in [1,\infty]$, 
we introduce the following condition which
states that $\Phi f$ is contained in $\ell^p(X,\mu)$ for every $f \in \ell^p(X,\mu)$:
\begin{enumerate}
	\item[$\Cp$]\hypertarget{ass:C}{} $\Phi(\ell^p(X,\mu))\subseteq \ell^p(X,\mu)$. \hfill(``$p$-Containment'')
\end{enumerate}
As we will most often be concerned with the case of $p=1$, we will write $\C$ for 
$\Cone$ and refer to it simply as containment.

We first show that a uniform lower bound on the measure 
and containment are equivalent when $\phi$ leads to the porous medium
operator and these properties imply
the continuity of $\Phi$ on all $\ell^p$ spaces.
\begin{proposition}\label{p:PME_measure}
{
For every continuous $\phi\colon\R\to\R$, the map
$\Phi\colon \ell^\infty(X) \to \ell^\infty(X)$ is continuous.
Furthermore, if $\phi(s)=s|s|^{m-1}$ for some $m>1$,
then $G$ satisfies $\UM$}
if and only if $\Phi$ satisfies $\Cp$ for some (equivalently, all) $p \in [1,\infty)$. In this case, 
$\Phi\colon \ell^p(X,\mu) \to \ell^p(X,\mu)$ is continuous for all $p \in [1,\infty]$.
\end{proposition}
\begin{proof}
{
If $u_n\to u$ in $\ell^\infty(X)$, then the values of $u_n$ and $u$ lie in a common compact interval. 
The uniform continuity of $\phi$ on this interval gives
$\|\Phi u_n-\Phi u\|_\infty\to0$ which gives continuity of $\Phi$ in this case.
Now let $\phi(s)=s|s|^{m-1}$ with $m>1$. We prove the equivalence of
$\UM$ and $\Cp$. Fix $p\in[1,\infty)$. If $\mu(x)\geq c>0$ for every
$x\in X$, then
\[
 \|u\|_\infty\leq c^{-1/p}\|u\|_p
\]
for every $u\in\ell^p(X,\mu)$, and hence
\[
 \|\Phi u\|_p^p
 =\sum_{x\in X}|u(x)|^{mp}\mu(x)
 \leq\|u\|_\infty^{(m-1)p}\|u\|_p^p<\infty.
\]
Thus, $\UM$ implies $\Cp$ for every $p \in [1,\infty)$.

Conversely, suppose that $\inf_{x\in X}\mu(x)=0$. We may choose distinct
vertices $(x_n)$ such that, with $\mu_n:=\mu(x_n)$,
\[
 \sum_n\mu_n^{1-1/m}<\infty.
\]
For the fixed, arbitrary $p\in[1,\infty)$, define
\[
 f(x):=
 \begin{cases}
  \mu_n^{-1/(pm)}&\text{if }x=x_n,\\
  0&\text{otherwise}.
 \end{cases}
\]
Then
\[
 \|f\|_p^p=\sum_n\mu_n^{1-1/m}<\infty,
 \qquad
 \|\Phi f\|_p^p=\sum_n1=\infty.
\]
Thus $\Cp$ fails for every $p\in[1,\infty)$ whenever $\UM$ fails.

It remains to prove continuity for finite $p$ under $\UM$. If
$u_n\to u$ in $\ell^p(X,\mu)$, the preceding embedding gives a common
bound $R$ for the values of $u_n$ and $u$. The power law $\phi$ is
Lipschitz on $[-R,R]$, so there exists $C_R>0$ such that
\[
 |\Phi u_n(x)-\Phi u(x)|\leq C_R|u_n(x)-u(x)|
 \qquad\text{for every }x\in X.
\]
Consequently,
$\|\Phi u_n-\Phi u\|_p\leq C_R\|u_n-u\|_p\to0$. Together with the
$\ell^\infty$ assertion proved above, this completes the proof.
}
\end{proof}

For the porous medium operator, the proposition above
yields that boundedness of the degree and containment imply continuity
of the operator.
\begin{corollary}\label{c:continuity}
Let $\phi(s)=s|s|^{m-1}$ for $m \geq 1$. Let $G$ satisfy $\BD$.
Then, $\dom({\mathcal{L}^{(\infty)}}) = \ell^\infty(X)$
and $\mathcal{L}^{(\infty)}$ is continuous.
If, additionally, $\Phi$ satisfies $\C$, then $\dom({\Deltaphip}) = \ell^p(X,\mu)$
and $\Deltaphip$ is continuous for $p \in [1,\infty]$.
\end{corollary}
\begin{proof}
For $m=1$ and $p \in [1,\infty]$,
Theorem~2.15 in \cite{keller2021graphs} gives that the boundedness of the degree $\BD$ is equivalent to 
boundedness of the Laplacian $\Delta$ restricted to $\ell^p(X,\mu)$. This gives the conclusion in this case.

For $m>1$ and $p=\infty$, we note that $\ell^\infty(X) \subseteq \dom(\Delta)$ by the local summability
condition on the edge weight. As $\Phi\colon \ell^\infty(X) \to \ell^\infty(X)$, this implies 
$\dom({\mathcal{L}^{(\infty)}}) = \ell^\infty(X)$
and the continuity of $\mathcal{L}^{(\infty)}$ for the porous medium operator
by Proposition~\ref{p:PME_measure} and the boundedness of the Laplacian.

For $m>1$ and $p \in [1,\infty)$, we note that assumption $\C$ gives uniformity of the measure
$\UM$ by Proposition~\ref{p:PME_measure}. By $\UM$ we then obtain 
$$\ell^p(X,\mu) \subseteq \ell^\infty(X) \subseteq \dom(\Delta)$$
for every $p \in [1,\infty)$. Now,
Proposition~\ref{p:PME_measure} gives $\Phi\colon \ell^p(X, \mu) \to \ell^p(X, \mu)$ is continuous
implying $\dom({\Deltaphip}) = \ell^p(X,\mu)$
and the continuity of $\Deltaphip$ by boundedness of the Laplacian. This completes the proof.
\end{proof}

We now explore the $p$-edge-measure condition. 
We recall that $G$ satisfies \EM[p] if for every $x\in X$
\[
  \frac{w(x,\cdot)}{\mu(\cdot)}\in \ell^{p}(X,\mu)
  \qquad\text{i.e.}\qquad
  \begin{cases}
  \sum_{y\in X}\biggl(\frac{w(x,y)}{\mu(y)}\biggr)^{p}\mu(y)<\infty &\text {if }\, p \in [1,\infty) \\
\sup_{y\in X}\frac{w(x,y)}{\mu(y)}<\infty &
  \text {if }\, p=\infty.
  \end{cases}
  \]
For $p=\infty$, we write \EM[] for \EM[$\infty$]. 

We note that $\EM[1]$ amounts to the local summability of the edge weight 
and is always satisfied. 
Hence, if $\EM[p]$ holds for some $p>1$ then,  
by interpolation, it holds for   every $r \in[1,p]$.

We now characterize the p-edge-measure condition in various ways. We also show that
this condition implies that $\Delta^{(q)}$ is closed where $q$ is the H{\"o}lder conjugate
of $p$.
See Theorem~1.29 in \cite{keller2021graphs} for the corresponding statements for $p=2$
and Lemma~3.2 in \cite{milatovic2015maximal} where the authors prove (iii) $\Longrightarrow$ (ii)
from the result below
for general $p$.
\begin{lemma}\label{lem:EMp}
Let $p\in[1,\infty]$ and let $q$ be the  conjugate exponent of $p$, i.e.,  $p^{-1}+q^{-1}=1$.
The following statements are equivalent:
\begin{itemize}
\item[\textup{(i)}]   \EM[p].
\item[\textup{(ii)}]  $\ell^{q}(X,\mu)\subseteq\dom(\Delta)$.
\item[\textup{(iii)}] $\Delta(C_c(X))\subseteq\ell^{p}(X,\mu)$.
\end{itemize}
In this case, $\Delta^{(q)}$ is closed.
In particular, \EM[1] always holds and thus $\Delta^{(\infty)}$ is always closed.
\end{lemma}

\begin{proof}
(i) $\Longrightarrow$ (ii):
Assume \EM[p] and let $u\in\ell^{q}(X,\mu)$. Let $x \in X$.
We apply H\"older's inequality on $(X,\mu)$ to obtain
\[
  \sum_{y\in X}w(x,y)\,|u(y)| = \sum_{y\in X}\frac{w(x,y)}{\mu(y)}\,|u(y)|\,\mu(y)
  \;\le\;
  \biggl\|\frac{w(x,\cdot)}{\mu(\cdot)}\biggr\|_p
  \;\|u\|_q
  \;<\;\infty.
\]
Hence, $u\in\dom(\Delta)$.

(ii) $\Longrightarrow$ (i):
Assume $\ell^{q}(X,\mu)\subseteq\dom(\Delta)$ and fix $x\in X$.
Let $N(x)=\{x\}\cup\{y\in X \mid y\sim x\}$ denote the combinatorial
neighborhood of~$x$ and define the measure $\nu_x$ on $N(x)$ by
\[
  \nu_x(y)=
  \begin{cases}
    \mu(x) & \textup{if } y=x\\
    w(x,y) & \textup{if } y\sim x.
  \end{cases}
\]
Consider the restriction map $R_x\colon\ell^{q}(X,\mu)\to\ell^{1}(N(x),\nu_x)$
given by $R_x(u)=u_{|_{N(x)}}$.
Since $\ell^{q}(X,\mu)\subseteq\dom(\Delta)$, for every
$u\in\ell^{q}(X,\mu)$ we have
\[
  \|R_x u\|_{\ell^{1}(N(x),\nu_x)}
  =|u(x)|\,\mu(x)+\sum_{y\sim x}|u(y)|\,w(x,y)<\infty
\]
so that $R_x$ is well-defined.
Moreover, $R_x$ is closed since if $u_n\to u$ in $\ell^{q}(X,\mu)$ and
$R_x u_n\to g$ in $\ell^{1}(N(x),\nu_x)$, then both convergences
imply pointwise convergence in $N(x)$, whence $g=u|_{N(x)}$.
By the closed graph theorem, $R_x$ is bounded, i.e., there exists
$C_x>0$ with
\[
  \sum_{y\in X}w(x,y)\,|u(y)|\le C_x\,\|u\|_q
\]
for all $u\in\ell^{q}(X,\mu)$.

If we let $g_x(y)=w(x,y)/\mu(y)$ for $y \in X$, 
this means that for  every $u\in \ell^q(X,\mu)$ and $x \in X$, we get
$g_x u\in \ell^1(X,\mu)$ and  $\|g_xu\|_1\leq C_x \|u\|_q$.
By a standard result in real analysis (e.g., \cite[Theorem 6.14]{Folland1999RealAnalysis}) it follows that $g_x\in \ell^p(X,\mu)$ for every $x \in X$ which is \EM[p].  

(i) $\Longleftrightarrow$ (iii):
Since $C_c(X)=\operatorname{span}\{1_x \mid x\in X\}$, for (iii) it suffices
to characterize when $\Delta 1_x\in\ell^{p}(X,\mu)$ for every $x\in X$.
A direct computation gives
\[
  \Delta 1_x(y)=
  \begin{cases}
    \Deg(x) & \textup{if } y=x\\[4pt]
    \displaystyle-\frac{w(x,y)}{\mu(y)} & \textup{if } y\neq x.
  \end{cases}
\]
It follows that 
\[
  \Delta1_x\in\ell^{p}(X,\mu)
  \quad\Longleftrightarrow\quad
    \frac{w(x,\cdot)}{\mu(\cdot)}\in \ell^p(X,\mu)
\]
which is exactly \EM[p]. This gives the equivalence between (iii) and (i).

\textbf{Closedness of $\boldsymbol{\Delta^{(q)}}$.}
Assume the equivalent conditions hold for $p \in [1,\infty]$. For any $u\in\ell^{q}(X,\mu)$
and $x\in X$, we estimate
\begin{align*}
  |\Delta u(x)|
  &\le \frac{1}{\mu(x)}\sum_{y\in X}w(x,y)\bigl(|u(x)|+|u(y)|\bigr)
       +\frac{\kappa(x)}{\mu(x)}|u(x)|\\
  &= \Deg(x)\,|u(x)|
     +\frac{1}{\mu(x)}\sum_{y\in X}w(x,y)\,|u(y)|\\
  &\le \Deg(x)\,\mu(x)^{-1/q}\,\|u\|_{q}
     +\frac{1}{\mu(x)}
       \biggl\|\frac{w(x,\cdot)}{\mu(\cdot)}\biggr\|_{p}
       \|u\|_{q}
\end{align*}
where the last step uses $|u(x)|\le\mu(x)^{-1/q}\|u\|_{q}$, with  $\mu(x)^{-1/q}=1$ if $q=\infty$, and
H\"older's inequality.
Hence, for each fixed~$x$, the point evaluation
$u\mapsto\Delta u(x)$ is a continuous linear functional on
$\ell^{q}(X,\mu)$.

Now suppose $u_n\in\dom(\Delta^{(q)})$ with $u_n\to u$ and
$\Delta u_n\to g$ in $\ell^{q}(X,\mu)$.
Convergence in $\ell^{q}(X,\mu)$ implies pointwise convergence (via
the evaluation functional $u\mapsto u(x)$) and continuity of
$u\mapsto\Delta u(x)$ gives $\Delta u_n(x)\to\Delta u(x)$ for every $x \in X$.
On the other hand, $\Delta u_n\to g$ in $\ell^{q}(X,\mu)$ forces
$\Delta u_n(x)\to g(x)$ for every $x \in X$.
Therefore, $\Delta u(x)=g(x)$ for all $x\in X$, so
$\Delta u=g\in\ell^{q}(X,\mu)$ and $u\in\dom(\Delta^{(q)})$.
\end{proof}

We remark that the $p$-edge-measure conditions hold for a large class
of graphs. First, as noted in Lemma~\ref{lem:EMp} directly above, 
$\EM[1]$ always holds.
For $p \in [1,\infty]$ we have the following connections with our other assumptions
for graphs.
\begin{corollary}\label{c:emp}
    If $G$ is locally finite or
    satisfies $\UM$ or $\B$, then $G$ satisfies \EM[p] for all $p \in [1,\infty]$.
\end{corollary}
\begin{proof}
    In the case of local finiteness, \EM[p] follows trivially from definitions.
    For $\UM$, i.e., $\mu(x) \geq C >0$ for all $x \in X$, 
    \EM[p] follows
    from the local summability of the edge weight 
    at $x \in X$, i.e., $\sum_y w(x,y)<\infty$ for all $x \in X$.
    Finally, in the case of $\B$, it follows that the Laplacian without
    the killing term is bounded on $\ell^p(X,\mu)$ for all $p \in [1,\infty]$, e.g., 
    \cite[Theorem~2.15]{keller2021graphs}. From this, $\Delta(C_c(X)) \subseteq \ell^p(X,\mu)$
    follows directly which gives \EM[p] for every $p \in [1,\infty]$ by Lemma~\ref{lem:EMp}.
\end{proof}

\begin{remark}[{\EM[] is \EM[$\infty$]}]\label{rem:EMp_p1}
We will be most
interested in the case when $\Delta^{(1)}$ is closed, i.e., 
when $p=\infty$ and $q=1$. Thus, as already noted above,
we write \EM[] for \EM[$\infty$]. 
Furthermore, there is a stronger condition than $\EM$, 
namely, that there exists $C>0$ such that $w(x,y) \leq C \mu(x) \mu(y)$
for all $x, y \in X$ whose consequences for curvature and stochastic properties have 
been explored in \cite{KM24, KMW25}.
\end{remark}

\subsection{Accretivity and \texorpdfstring{$m$}{m}-accretivity}
We now give a general definition for accretive and m-accretive operators. We also give an easy consequence of these properties 
which will be useful later. Finally, we introduce the bracket which allows us to
analyze accretivity on $\ell^p$ spaces.

Given a linear or nonlinear operator $T$ with domain 
$\dom(T) \subseteq \mathbb{E}$ where $\mathbb{E}$ is a 
Banach space and $\lambda \in \R$ we write
$$
S_\lambda = \id +\lambda T
$$
for the operator acting on $\dom(T)$
where $\id$ is the identity operator. We refer to $S_\lm$ as the \emph{shifted operator}.
\begin{definition}[Accretive and m-accretive operators]\label{def:m-accretivity}
If $\mathbb{E}=(E,\|\cdot\|)$ is a real Banach space and  $T \colon \dom(T)\subseteq E \to E$ is an operator, then $T$ is said to be \emph{accretive} if
$$ \left\|S_\lambda u - S_\lambda v \right\|\geq \|u - v\|$$	
for all $u,v \in \dom (T)$ and every $\lambda >0$.
An accretive operator $T$ is called \emph{m-accretive} if $S_\lambda$ is surjective for every $\lambda >0$. 
\end{definition}

Note that accretivity implies the injectivity of $S_\lambda$ for every $\lambda >0$.
\begin{remark}[For some/all $\lm$]\label{rem:lambda}
	The norm inequality in Definition~\ref{def:m-accretivity} need not be checked for every $\lambda>0$. Indeed, if it holds for every $\lambda\in(0,\lambda_0]$ with some fixed $\lambda_0>0$, 
    then it automatically holds for every $\lambda>0$. 
    This follows from the equivalent characterization of accretivity in terms of the bracket $[\,\cdot\,,\,\cdot\,]$, see \cite[II.~Proposition~2.5]{benilan1988evolution} and Remark~\ref{rem:bracket}
    directly below. 
    However, we will see that, for the maximal porous medium-type
    operators on $\ell^1$ and maximal Laplacians on $\ell^p$, if the accretivity norm
    inequality holds
    for one $\lm>0$, then it holds for all $\lm>0$.
    Furthermore, for general operators, 
    the surjectivity in Definition~\ref{def:m-accretivity} need not be checked 
    for every $\lambda>0$ either as an accretive operator 
    is m-accretive as soon as $S_\lambda$ is surjective for some $\lambda>0$, see \cite[Proposition~3.3]{barbu2010nonlinear}.
\end{remark}

We now give an easy consequence of these properties which will be useful later when we consider operators
and their restrictions.
\begin{lemma}\label{l:inclusion}
    Let $T_k \colon \dom(T_k)\subseteq E \to E$ be operators on a real Banach space for $k=1,2$ with $T_1 \subseteq T_2$.
    Let $S_\lm^{(k)}=\id+\lm T_k$ for $\lm>0$ and $k =1,2$. If $S_\lm^{(1)}$ is onto and $S_\lm^{(2)}$ is injective for some $\lm>0$, then $T_1=T_2$.
    In particular, this is true if $T_1$ is m-accretive and $T_2$ is accretive.
\end{lemma}
\begin{proof}
    This follows from definitions. Let $u \in \dom(T_2)$ so that $S_\lm^{(2)}u \in E$.
    Since $S_\lm^{(1)}$ is onto, there exists $v \in \dom(T_1)$ such that $S_\lm^{(1)}v=S_\lm^{(2)}u$.
    Since $T_1 \subseteq T_2$ and $S_\lm^{(2)}$ is injective, we now get $u=v \in \dom(T_1)$ which gives the conclusion.
\end{proof}

\begin{remark}[Accretivity and the bracket]\label{rem:bracket}
The accretivity condition admits an equivalent characterization in terms of the directional derivative of the norm. Setting
$$
u = f - g \qquad \mbox{and} \qquad z = Tf - Tg
$$
for $f, g \in \dom(T)$,
the operator $T$ is accretive if and only if
\begin{equation*}
 [u, z] \coloneqq \lim_{\lambda\to 0^+} \lambda^{-1}\bigl(\|u + \lambda z\| - \|u\|\bigr) \geq 0.
\end{equation*}
As mentioned in the above Remark~\ref{rem:lambda}, the expression on the left-hand side is known as the \emph{bracket} and $[u,z]$ is a shorthand for the right-hand derivative at $u$ of the norm in the direction $z$, see, for instance,~\cite[Section 2.2 and Theorem 2.15]{benilan1988evolution} or~\cite[Equation~(3.13)]{barbu2010nonlinear}. 

In the special case of $\ell^p(X,\mu)$ spaces $[u,z] \geq 0$ is equivalent to
\begin{equation}\label{eq:lp_accretivity}
[u,z] = \begin{cases}
\displaystyle\sum_{\substack{x \in X : \\ u(x) = 0}} |z(x)| \mu(x) + \displaystyle\sum_{\substack{x \in X : \\ u(x) \neq 0}} z(x) \sgn(u(x)) \mu(x) \geq 0  & \text{if } p = 1, \\[5ex] 
\|u\|_p^{1-p}\displaystyle\sum_{x \in X} z(x) u(x) |u(x)|^{p-2} \mu(x) \geq 0 & \text{if }  p \in (1, \infty)
\end{cases}
\end{equation}
where $\sgn\colon \R \to \R$ denotes the signum function defined by
$$\sgn(s)=
\begin{cases}
    1 &\textup{if } s>0 \\
    0 &\textup{if } s=0 \\
    -1 &\textup{if } s<0.
\end{cases}$$ 
A direct derivation of~\eqref{eq:lp_accretivity} can be found in~\cite[Example~(2.8)]{benilan1988evolution} or~\cite[Remark~3]{bianchi2022generalized}.
\end{remark}

\subsection{\texorpdfstring{Accretivity of the porous medium-type operator on finitely supported functions in $\ell^1$
and of the Laplacian on finitely supported functions in $\ell^p$}{Accretivity of the porous medium-type operator on finitely supported functions in ell1 and of the Laplacian on finitely supported functions in ellp}}
In this subsection, we address the accretivity of the porous medium-type operator
acting on the finitely supported functions in $\ell^1(X,\mu)$.
We then consider the Laplacian acting on the finitely
supported functions in $\ell^p(X,\mu)$ for $p \in [1,\infty]$. In particular, this covers
the case of these operators on finite graphs. While these results should be known, we offer
a proof based on the bracket characterization introduced in the previous subsection.

We start by looking at the porous medium-type operator on $\ell^1$.
Recall that the formal porous medium operator acts as $\Delta \Phi $
for $u \in C(X)$ such that $\Phi u \in \dom(\Delta)$. 

\begin{proposition}\label{p:PMTO_accretivity_Cc}
    Let $G$ be a graph. Then, $(\Delta \Phi)_c = (\Delta \Phi)_{\vert_{C_c(X)}}$
    gives an accretive operator on $\ell^1(X,\mu)$. 
\end{proposition}
\begin{proof}
{
Since $\phi(0)=0$, one has $\Phi(C_c(X))\subseteq C_c(X)$. Let $\varphi_1,\varphi_2\in C_c(X)$ and set
\[
 u:=\varphi_1-\varphi_2,\qquad
 v:=\Phi\varphi_1-\Phi\varphi_2,\qquad \text{and} \qquad
 z:=\Delta v.
\]
Define
\[
 \xi(x):=
 \begin{cases}
  \sgn(u(x))& \text{if }u(x)\neq0\\
  \sgn(z(x))& \text{if }u(x)=0.
 \end{cases}
\]
Since $\phi$ is monotone increasing,
$u(x)v(x)\geq0$ for all $x \in X$ and $u(x)=0$ implies $v(x)=0$. Consequently,
$|\xi(x)|\leq1$ and $v(x)\xi(x)=|v(x)|$ for every $x\in X$.
Moreover, for every $x,y\in X$,
\[
 (v(x)-v(y))(\xi(x)-\xi(y))
 =|v(x)| -v(x)\xi(y)+|v(y)|-v(y)\xi(x)\geq0.
\]
Since $v\in C_c(X)$ and the bounded function $\xi$ belongs to $\dom(\Delta)$,
Remark~\ref{rem:bracket} and Green's formula~\cite[Proposition~1.5]{keller2021graphs} yield
\[
\begin{aligned}
[u,z]= \sum_{x\in X}z(x)\xi(x)\mu(x) &= \sum_{x\in X}\Delta v(x)\xi(x)\mu(x)\\
&=\frac12\sum_{x,y\in X}w(x,y)(v(x)-v(y))(\xi(x)-\xi(y))
+\sum_{x\in X}\kappa(x)|v(x)|\geq0.
\end{aligned}
\]
The $\ell^1$-bracket characterization~\eqref{eq:lp_accretivity} proves the accretivity of $(\Delta\Phi)_c$.
}
\end{proof}

As a corollary, we automatically get accretivity of $\Delta \Phi$ on $\ell^1(X,\mu)$ 
over finite graphs. Recall that $\Deltaphi=\Deltaphi^{(1)}$
is the restriction of the formal operator $\Delta \Phi$ to the maximal domain in $\ell^1(X,\mu)$.
\begin{corollary}\label{c:PMTO_accretive_finite}
    If $G$ is a finite graph, then $\Deltaphi$ is accretive.
\end{corollary}

We next look at the Laplacian on finitely supported functions. Unlike the case
of $p=1$, for $p \in (1,\infty]$ we have to impose an additional condition
on the graph in order to ensure that finitely supported functions are mapped into
$\ell^p$.
\begin{proposition}\label{p:Lap_accretivity_Cc}
    Let $G$ be a graph satisfying \EM[p] for some $p \in [1, \infty]$. Then, $\Delta_c = \Delta_{\vert_{C_c(X)}}$
    gives an accretive operator on $\ell^p(X,\mu)$. In particular, $\Delta_c$ is always accretive on $\ell^1(X,\mu)$.
\end{proposition}
\begin{proof}
    We split the proof into cases. Let $\varphi \in C_c(X)$.

    Case $p=1$: This follows from Proposition~\ref{p:PMTO_accretivity_Cc}
    directly above by considering $\Phi=\id$.

    Case $p\in(1,\infty)$:
    In this case, we estimate the expression \eqref{eq:lp_accretivity} in Remark~\ref{rem:bracket} 
    with $u=\varphi$ and $z = \Delta \varphi$ by using Green's formula, \cite[Proposition~1.5]{keller2021graphs}, as follows:  
    \begin{align*}
        \sum_{x \in X} & \Delta\varphi(x)  \varphi(x) |\varphi(x)|^{p-2} \mu(x) =
        \frac{1}{2}
        \sum_{x,y \in X} w(x,y)
        \bigl(\varphi(x) |\varphi(x)|^{p-2} - \varphi(y) |\varphi(y)|^{p-2}\bigr) \cdot(\varphi(x)-\varphi(y)) \\
        & \hspace{2in} + \sum_{x \in X}\kappa(x)|\varphi(x)|^p  \\
        &\geq  \frac{1}{2}\sum_{x,y \in X} w(x,y)
        \bigl(\varphi(x) |\varphi(x)|^{p-2} - \varphi(y) |\varphi(y)|^{p-2}\bigr) \cdot(\varphi(x)-\varphi(y)) \geq 0
    \end{align*}
    { The last inequality follows because the function
    $s\mapsto s|s|^{p-2}$ is monotone increasing for $p>1$.}

    Thus, as we assume \EM[p] which is equivalent to $\Delta(C_c(X)) \subseteq \ell^p(X,\mu)$
    by Lemma~\ref{lem:EMp}, we now obtain that $\Delta_c$ gives an accretive operator
    on $\ell^p(X,\mu)$ by the characterization in terms of the bracket 
    introduced in Remark~\ref{rem:bracket} and
    linearity of the operator.

    Case $p=\infty$:
    Here, $\Delta \varphi \in \ell^\infty(X,\mu)$ by Lemma~\ref{lem:EMp} as we assume \EM[$\infty$].
    Let $x_0 \in X$ be such that $|\varphi|$ achieves a maximum
    at $x_0$.
    Without loss of generality, we assume that $\varphi(x_0) \geq 0$. It then follows
    from the definition of the Laplacian that $\Delta \varphi(x_0)\geq0$ and thus
    $(\id + \lm \Delta)\varphi(x_0) \geq \varphi(x_0)$ for all $\lm>0$. Therefore,
    $$\| (\id + \lm \Delta_c)\varphi \|_\infty \geq \|\varphi\|_\infty $$
    which gives accretivity by definition since we are in the linear case.
\end{proof}

As a corollary, we obtain that the Laplacian on finite graphs is
accretive for all $p \in [1,\infty]$.  Recall that $\Dep$ denotes the restriction
of $\Delta$ to the maximal domain
$$\dom(\Dep)= \{ u \in \ell^p(X,\mu) \cap \dom(\Delta) \mid \Delta u \in \ell^p(X,\mu)\}.$$
\begin{corollary}\label{c:accretive_finite_Lap}
    If $G$ is a finite graph, then $\Dep$ is an accretive
    operator for all $p \in [1,\infty]$.
\end{corollary}

We next show that accretivity follows from local finiteness in the case of no infinite paths
as the graph is finite in this case. As this statement may be of independent interest, we include
a short proof.
\begin{proposition}
If $G$ is locally finite and does not have infinite paths, then $G$ is finite. In particular,
$\Deltaphi$ and $\Dep$ for $p \in [1,\infty]$ are accretive.
\end{proposition}
\begin{proof}
It follows from  local finiteness that, having fixed $o\in X$, the balls $B_r=\{x \mid d(x,o)\leq r \}$ are finite
where $d$ denotes the combinatorial graph metric. 
We want to show that if the graph additionally only has finite paths, 
then there exists an $r$ such that $X=B_r$.
	
Suppose not. Then, for every $r$, there exists $x_r\in B_r \setminus B_{r-1}$ and, since $G$ is connected, there exists a minimizing path $\gamma_r$ joining $o$ to $x_r$ such that $\gamma_r(k)\in B_k\setminus B_{k-1}$ for all $k\leq r$.  Since $o$ has  only finitely many neighbors, there exists a subsequence $(n_1)$ such that  $\gamma_{n_1}(1)$ is a fixed vertex $y_1\in B_1\setminus B_0$. Iterating  the argument, for every $k$ we find a  subsequence $(n_k)$ of  $(n_{k-1})$ such that
	$\gamma_{n_k}(j)=y_j$ is a fixed vertex in $B_j \setminus B_{j-1}$  for every $j\leq k$.
	
To conclude, we consider the  vertices $(y_k)$. 
Since for  every $l\geq k$, $y_k= \gamma_{n_{l}}(k)$,
	it follows that these points form a path originating from $o$ which has length at least $k$. Since $k$ is
	arbitrary, this contradicts the fact that $G$ has no infinite paths.

    The ``in particular'' statements now follow from Corollaries~\ref{c:PMTO_accretive_finite}~and~\ref{c:accretive_finite_Lap}.
\end{proof}

In forthcoming sections, we will consider both the maximal Laplacian $\Dep$ 
as well as a minimal operator. It turns out that when  
these two operators agree, they are automatically m-accretive.
We now start this investigation by introducing the minimal operator and discussing some
basic properties.

As established in Proposition~\ref{p:Lap_accretivity_Cc} directly above, whenever the graph satisfies \EM[p],
$\Delta_c$ gives an accretive operator on $\ell^p(X,\mu)$. By general theory,
densely defined accretive linear operators are always closable, thus, we can define
the closure for $p \in [1,\infty)$. 
For the case of $p =\infty$, although $\Delta_c$ is not densely defined,
it is still closable as we now show.
We denote the closure of $\Delta_c$ on $\ell^p(X,\mu)$ by $\deltam^{(p)}$. That is, whenever $\Delta_c$ is closable
on $\ell^p(X,\mu)$, we let
$$\dom(\deltam^{(p)}) = \{ u \in \ell^p(X,\mu) \mid \exists \varphi_n \in C_c(X) \textup{ with }  \varphi_n  \to u \textup{ and } \Delta \varphi_n \to v  \textup{ in } \ell^p(X,\mu) \} $$
and then let $\deltam^{(p)}u =v$ for $u \in \dom(\deltam^{(p)})$ as above.

\begin{lemma}\label{l:min_accretivity}
Let $G$ satisfy \EM[p] for some $p \in [1,\infty]$. 
Then, $\Delta_c$
is accretive and closable on $\ell^p(X,\mu)$. Consequently, $\deltam^{(p)}$
is defined and accretive. If $\Delta^{(p)}$ is closed, then $\deltam^{(p)} \subseteq \Delta^{(p)}$.
In particular, this happens if $G$ additionally satisfies \EM[q] for $p^{-1}+q^{-1}=1$.
\end{lemma}
\begin{proof}  

The accretivity of $\Delta_c$ on $\ell^p(X,\mu)$ under \EM[p]
was established in Proposition~\ref{p:Lap_accretivity_Cc} above. 
We now discuss closability. For $p \in [1, \infty)$, since $C_c(X) \subseteq \ell^p(X,\mu)$ is
dense, closability follows as densely defined accretive linear operators are closable, e.g.,  \cite[Proposition~II.3.14]{EngelNagel2000}. For $p=\infty$,
$C_c(X)$ is not dense in $\ell^\infty(X)$, however, as \EM[1] holds for all graphs, it follows that $\Delta^{(\infty)}$ is closed by Lemma~\ref{lem:EMp}.
Therefore, as we assume \EM[$\infty$], $\Delta_c$ acting on $\ell^\infty(X)$ has 
$\Delta^{(\infty)}$ as a closed extension and is thus closable. Hence, $\deltam^{(p)}$ is defined
as soon as \EM[p] holds for any $p \in [1,\infty]$ and, by taking limits, $\deltam^{(p)}$ inherits accretivity from $\Delta_c$.
The remaining statements then follow by definitions and Lemma~\ref{lem:EMp}.
\end{proof}

\section{\texorpdfstring{Accretivity and m-accretivity of porous medium-type operators on $\ell^1$}{Accretivity and m-accretivity of porous medium-type operators on ell1}}\label{sec:main}
In this section we provide sufficient conditions under which porous medium-type operators 
are m-accretive on $\ell^1$.
In the case of finite graphs, we have just discussed 
how the Laplacian $\Delta^{(p)}$ is accretive for every $p \in [1, \infty]$, see Corollary~\ref{c:accretive_finite_Lap} directly above. 
However, for general porous medium-type operators,
accretivity may not hold even for finite graphs.
For instance, accretivity already fails 
for the porous medium operator $\mathcal{L}^{(2)}$  with $\phi(s) = s|s|^3$
in the case of a graph with only four vertices,
see \cite[Example following Proposition 2.5]{bianchi2022generalized}.
On the other hand, $\Deltaphi$ is always accretive on finite graphs
by Corollary~\ref{c:PMTO_accretive_finite} above. 
Therefore, we will focus on the case of $\Deltaphi$ and infinite graphs.

We will first show that $\Deltaphi$ is accretive under conditions $\B$ and $\C$. 
We then construct a dense subset $\Omega \subseteq \ell^1(X,\mu)$ 
such that $\Deltaphiom$ is m-accretive and
give conditions for $\Deltaphi = \Deltaphiom$ 
which yields m-accretivity of $\Deltaphi$. 
A key insight is that accretivity of $\Deltaphi$ already suffices for m-accretivity,
in fact,  it suffices to show that the shifted operator
$\mathcal{S}_\lambda = \id + \lambda\Deltaphi$ is injective for $\lambda > 0$.

Finally, we will  introduce
the minimal operator $\deltaphim$ which is the closure of the restriction
of $\Deltaphi$ to the finitely supported functions, 
and give conditions
under which $\Deltaphi$ and $\deltaphim$ coincide.
We point out that $\deltaphim$ is always accretive when it is defined, thus,
when these operators agree, they are automatically m-accretive.
Along the way, we also investigate these questions for the Laplacian on $\ell^1$ where
we show that the minimal and maximal operators agree under some weaker assumptions.

\subsection{\texorpdfstring{Some preliminary results on the accretivity of $\Deltaphi$}{Some preliminary results on the accretivity of Delta Phi}}

We begin by showing that
boundedness of the edge degree and containment imply that $\Deltaphi$ is accretive on its domain
$$\dom(\Deltaphi)=\{u \in \ell^1(X,\mu) \cap \dom (\Delta \Phi) \mid \Delta \Phi u \in \ell^1(X,\mu)\}.$$
Indeed, boundedness of the edge degree
implies the boundedness of the Laplacian without the killing term 
which can be used to show the accretivity
of $\Deltaphi$.
\begin{theorem}\label{t:accretivity_bd}
Let $G$ satisfy $\B$ and $\Phi$ satisfy $\C$. Then, $\Deltaphi$ is
accretive.
\end{theorem}

\begin{proof}
{
Let $\Delta'$ denote the formal Laplacian without the killing term.
Under $\B$, $\Delta'$ is bounded on $\ell^1(X,\mu)$
by~\cite[Theorem~2.15]{keller2021graphs}. Let
$\varphi_1,\varphi_2\in\dom(\Deltaphi)$ and set
\[
 u:=\varphi_1-\varphi_2,\qquad
 v:=\Phi\varphi_1-\Phi\varphi_2,\qquad \text{and} \qquad
 z:=\Delta v.
\]
By $\C$, $v\in\ell^1(X,\mu)$, while $z\in\ell^1(X,\mu)$ by the
definition of the maximal domain. Define
\[
 \xi(x):=
 \begin{cases}
  \sgn(u(x))& \text{if }u(x)\neq0\\
  \sgn(z(x))& \text{if }u(x)=0.
 \end{cases}
\]
Monotonicity of $\phi$ gives
$u(x)v(x)\geq0$ for all $x \in X$ and $u(x)=0$ implies $v(x)=0$. Hence,
$|\xi(x)|\leq1$ and $v(x)\xi(x)=|v(x)|$.
{ Choose $C>0$ such that
$\sum_yw(x,y)\leq C\mu(x)$ for every $x\in X$. Then
\[
\begin{aligned}
&\sum_{x,y\in X}w(x,y)|v(x)-v(y)|\,|\xi(x)|
 \leq2C\|v\|_1<\infty  \qquad \text{ and }\\
&\sum_{x,y\in X}w(x,y)
 |(v(x)-v(y))(\xi(x)-\xi(y))| \leq4\sum_{x\in X}|v(x)|\sum_{y\in X}w(x,y)
\leq4C\|v\|_1<\infty.
\end{aligned}
\]
Thus, all rearrangements below are justified by absolute convergence, and
symmetry of $w$ gives}
\[
\begin{aligned}
\sum_{x\in X}\Delta'v(x)\xi(x)\mu(x)
&=\frac12\sum_{x,y\in X}w(x,y)(v(x)-v(y))(\xi(x)-\xi(y))\\
&=\frac12\sum_{x,y\in X}w(x,y)
\bigl(|v(x)|-v(x)\xi(y)+|v(y)|-v(y)\xi(x)\bigr)\geq0.
\end{aligned}
\]
Moreover, $(\kappa/\mu)v=z-\Delta'v\in\ell^1(X,\mu)$. Therefore, the
$\ell^1$-bracket formula~\eqref{eq:lp_accretivity} gives
\[
[u,z]= \sum_{x\in X}z(x)\xi(x)\mu(x) = \sum_{x\in X}\Delta'v(x)\xi(x)\mu(x)
+\sum_{x\in X}\kappa(x)|v(x)|\geq0,
\]
which proves the claim.
}
\end{proof}

\subsection{\texorpdfstring{The operators $\Deltaphiom$ and $\Deltaphi$}{The operators Delta Phi on Omega and Delta Phi}}\label{SSect:Omega}
We now introduce the set $\Omega$ and establish some basic properties
of $\Deltaphiom$, the restriction of $\Deltaphi$ to $\Omega$.
For full details, see~\cite{bianchi2022generalized}.
For $G$ infinite, we let $(X_n)$ be an exhaustion of $X$.

For our presentation, it is convenient to introduce some auxiliary operators. 
We first let $\Deltadir[n] \colon C(X_n) \to C(X_n)$ be the \emph{Dirichlet Laplacian} on $X_n$,
which acts as
\begin{equation}\label{def:dir_laplacian}
\Deltadir[n] u(x) = \frac{1}{\mu(x)} \sum_{y \in X_n} w(x,y) (u(x)-u(y)) +
\frac{1}{\mu(x)} \left( \sum_{y \not \in X_n} w(x,y) +\kappa(x) \right) u(x) 
\end{equation}
for $u \in C(X_n)$ and $x \in X_n$. Note that, if $X_n = X$ is finite, then $\Deltadir[n] = \Delta$.

{
Equivalently, define
\[
 G_n^{D}:=
 \bigl(X_n,w|_{X_n\times X_n},\kappa_n,\mu|_{X_n}\bigr)
 \qquad \text{with }
 \kappa_n(x):=\kappa(x)+\sum_{y\notin X_n}w(x,y).
\]
We call $G_n^{D}$ the \emph{finite Dirichlet graph associated with
$X_n$}. Its formal Laplacian is precisely $\Deltadir[n]$.
}

Next, we denote by $\boldsymbol{\mathfrak{i}}_{n}$ 
the canonical embedding and by $\boldsymbol{\pi}_n$ the canonical projection for each $X_n$, namely,
\begin{align*}
&\boldsymbol{\mathfrak{i}}_{n} \colon C(X_n)\to C(X) \quad &\boldsymbol{\mathfrak{i}}_{n}u(x) = \begin{cases}
u(x) & \mbox{if } x \in X_n\\
0   & \mbox{if } x \in X\setminus X_n
\end{cases}\\
&\boldsymbol{\pi}_n \colon C(X) \to C(X_n) \quad &\boldsymbol{\pi}_n u(x) = u(x) \mbox{ for every } x \in X_n.
\end{align*}
A direct computation yields $\Deltadir[n] u(x) = \Delta(\boldsymbol{\mathfrak{i}}_n u)(x)$ for all $u \in C(X_n)$ and $x \in X_n$.

Finally, we define the operator $\Lmin \colon C_c(X) \to C_c(X)$, with domain
$\dom(\Lmin) = C_c(X)$, by
\begin{equation}\label{eq:Deltaphi_n}
    \Lmin \varphi = \boldsymbol{\mathfrak{i}}_{n}\Deltadir[n]\Phi\boldsymbol{\pi}_{n} \varphi.
\end{equation}

We are now ready to define $\Omega$.

\begin{definition}[$\Omega$ and the restriction $\Deltaphiom$]\label{def:Omega}
Let $\Omega \subseteq \dom(\Deltaphi)$ be defined by
$$
\Omega = \bigl\{ u \in \dom(\Deltaphi) \mid \exists \varphi_n \in C_c(X) \text{ with } \operatorname{supp}\varphi_n \subseteq X_n,\; \varphi_n  \to u \textup{ and } 
\Lmin  \varphi_n \to \Deltaphi u \text{ in } \ell^1(X,\mu) \bigr\}.
$$  
We let $\Deltaphiom$ denote the restriction of $\Deltaphi$ to $\Omega$. 
\end{definition}

\begin{remark}[$\Lmin  = 1_{X_n} \Deltaphi $]\label{rem:Cutoff_Laplacian}
In the $\ell^1$ setting we have 
$C_c(X) \subseteq \dom(\Deltaphi)$ and a direct computation shows that
$$
\Lmin \varphi = 1_{X_n} \Deltaphi \varphi
$$
for every $\varphi \in C_c(X)$ with $\operatorname{supp}\varphi \subseteq X_n$ where $1_{X_n}$ denotes the characteristic function of $X_n$. Consequently, $\Omega$ admits an equivalent definition obtained by replacing $\Lmin\varphi_n$ with $1_{X_n}\Deltaphi \varphi_n$ in Definition~\ref{def:Omega}.
\end{remark}

\begin{remark}[On the dependence of $\Omega$ on the exhaustion]
We note that the  definition of the set $\Omega$ may depend on the exhaustion $(X_n)$ and we write
$\Omega_{(X_n)}$ when this dependence needs to be made explicit. We will see that, from
any starting exhaustion $(X_n)$, one can pass to a subsequence $(Y_k)=(X_{n_k})$ such that
$\Deltaphi_{\vert \Omega_{(Y_k)}}$ is m-accretive. Moreover, whenever the shifted
operator $\mathcal{S}_\lambda = \id + \lambda \Deltaphi$ is injective, it turns out that
$\Omega_{(X_n)}=\dom(\Deltaphi)$ for \emph{any} exhaustion. Therefore, in this case, $\Omega$ is
independent of the exhaustion.
\end{remark}

We now state some basic properties of the set $\Omega$ which hold for every exhaustion.
First, we note that $\Omega$ always contains the finitely supported functions
and is thus dense in the domain of $\Deltaphi$.
\begin{lemma}[{\cite[Lemma~B.7]{bianchi2022generalized}}]\label{l:omega_denseness}
    For every graph $G$, $C_c(X) \subseteq \Omega \subseteq \dom(\Deltaphi)$. In particular, 
    $$\overline{\Omega}^{\ell^1} = \overline{\dom(\Deltaphi)}^{\ell^1} = \ell^1(X,\mu).$$
\end{lemma}
\begin{proof}
    Let $\varphi \in C_c(X)$. Then, $\Phi \varphi \in C_c(X)$ and thus 
    $\Delta \Phi\varphi \in \ell^1(X,\mu)$ since \EM[1] always holds as noted in
    Lemma~\ref{lem:EMp}. Thus, $\varphi \in \dom(\Deltaphi) \subseteq \ell^1(X,\mu)$.
    
    Now, as $(X_n)$ is an exhaustion, there exists $N$ such that
    $\operatorname{supp}\varphi \subseteq X_n$ for all $n \geq N$. We let $\varphi_n \in C_c(X)$
    be defined by 
    $$ \varphi_n = \begin{cases}
        0 & \textup{if } n < N \\
        \varphi & \textup{otherwise} 
    \end{cases}
    $$
    so that $\operatorname{supp}\varphi_n \subseteq X_n$ for all $n$ and $\varphi_n \to \varphi$
    in $\ell^1(X,\mu)$ as $n \to \infty$. 
    Furthermore, by Remark~\ref{rem:Cutoff_Laplacian}, we get for $n \geq N$
    $$\|\Deltaphi \varphi -\Deltaphi_n \varphi_n\|_1 = \|\Deltaphi \varphi -\Deltaphi_n \varphi\|_1 =
    \| \Deltaphi \varphi - 1_{X_n} \Deltaphi \varphi \|_1 = \|1_{X \setminus X_n} \Deltaphi \varphi\|_1.
    $$
    This implies $\Deltaphi_n \varphi_n \to \Deltaphi \varphi$ in $\ell^1(X,\mu)$
    as $n \to \infty$
    by Lebesgue's dominated convergence theorem as $(X_n)$ is an exhaustion and 
    $\Deltaphi \varphi \in \ell^1(X,\mu)$. Therefore, $\varphi \in \Omega$
    which completes the proof.
\end{proof}

Next, we note that, as $\Lmin$ is essentially given by the cutting off the Dirichlet Laplacian
on a finite graph, it follows that $\Lmin$ is accretive.

\begin{proposition}[{\cite[Corollary~B.4]{bianchi2022generalized}}]\label{prop:accretivity_lmin}
The operator $\Lmin$ is accretive in $\ell^1(X,\mu)$ for every $n$.
\end{proposition}
\begin{proof}
{
For fixed $n$, consider the finite Dirichlet graph $G_n^{D}$ defined
above. By
Proposition~\ref{p:PMTO_accretivity_Cc}, the operator
$\Deltadir[n]\Phi$ is accretive on $\ell^1(X_n,\mu)$.

For $\varphi,\psi\in C_c(X)$ and $\lambda>0$, the parts inside and outside of $X_n$ have disjoint supports and, therefore,
\[
 \begin{split}
 \|(\id+\lambda\Lmin)\varphi-(\id+\lambda\Lmin)\psi\|_1 &=\|(\id_n+\lambda\Deltadir[n]\Phi)\boldsymbol{\pi}_n\varphi
 -(\id_n+\lambda\Deltadir[n]\Phi)\boldsymbol{\pi}_n\psi\|_{1,n}
 +\|1_{X\setminus X_n}(\varphi-\psi)\|_1\\
 &\geq\|\boldsymbol{\pi}_n(\varphi-\psi)\|_{1,n}
 +\|1_{X\setminus X_n}(\varphi-\psi)\|_1
 =\|\varphi-\psi\|_1
 \end{split}
\]
where $\id_n$ is the restriction of $\id$ to $X_n$
and $\|\cdot\|_{1,n}$ is the $\ell^1$ norm on $X_n$.
Thus, $\Lmin$ is accretive.
}
\end{proof}

As a consequence of the previous proposition it follows that $\Deltaphiom$ is always accretive.

 \begin{proposition}[{\cite[Lemma~B.8]{bianchi2022generalized}}]\label{p:accretivityDeltaphiom}
   $\Deltaphiom$ is accretive.
 \end{proposition}
\begin{proof}
 For convenience, we give a brief sketch of the proof. When $G$ is finite the statement reduces to Corollary~\ref{c:PMTO_accretive_finite}. 
 Thus, assume $G$ is infinite and let $u,v\in\Omega$. By the definition of $\Omega$, there exist sequences $(u_n),(v_n)$ in $\dom(\Lmin)=C_c(X)$ with $u_n\to u$, $v_n\to v$ and $\Lmin u_n\to\Deltaphi u$, $\Lmin v_n\to\Deltaphi v$ in $\ell^1(X,\mu)$. Since $\Lmin$ is accretive by Proposition~\ref{prop:accretivity_lmin}, for every $\lambda > 0$ we have
$$ \|(\id+ \lambda\Lmin )u_n -(\id+ \lambda\Lmin ) v_n\|_1 \geq \|u_n - v_n\|_1.$$
Taking $n\to\infty$ yields $\|(\id+ \lambda\Deltaphi )u -(\id+ \lambda\Deltaphi) v\|_1 \geq \|u-v\|_1$, i.e.,\ $\Deltaphiom$ is accretive.
\end{proof}

We note that one of the motivating reasons for the definition
of $\Omega$ is to have a dense set in $\ell^1(X,\mu)$ 
such that the restriction of $\Deltaphi$ to this set is accretive. 
It is not clear that restricting $\Deltaphi$
to $C_c(X)$ instead of using $\Deltaphi_n$ on $C_c(X)$ works for this purpose, see Remark~10 in~\cite{bianchi2022generalized}. In some sense, $\Deltaphi_n$ is the natural operator to enforce Dirichlet boundary conditions in an exhaustion argument.

\subsubsection{\texorpdfstring{On the m-accretivity of $\Deltaphiom$}{On the m-accretivity of Delta Phi on Omega}}\label{sec:m-accretivity_Lomega}
While $\Deltaphiom$ is always accretive, 
in~\cite[Theorem 1]{bianchi2022generalized} the m-accretivity of $\Deltaphiom$
is asserted under the conditions of local finiteness or uniform measure or 
bounded edge degree and containment.  In this subsection we  refine the original proof and  strengthen the conclusions of \cite[Theorem 1]{bianchi2022generalized}.  
Indeed, we will show 
that, by passing to a subsequence of the exhaustion sequence, 
there always exists a dense subset $\Omega$ such that
$\Deltaphiom$ is m-accretive. For extra details on the original 
proof in \cite{bianchi2022generalized} and the key refinements, see Appendix~\ref{sec:proof_Theo1}.

Given $\lambda>0$ and $n \in \N$, for notational convenience we let 
$$
\mathcal{S}_\lambda = \id + \lambda \Deltaphi \qquad \mbox{and} \qquad \mathcal{S}_{\lambda,n} = \id + \lambda \Lmin.
$$

We need a few  preliminary lemmas. The first lemma states that
we can solve the problem of interest on finite subgraphs.
\begin{lemma}\label{lem:finite-resolvent}
Let $X_n\subseteq X$ be finite, nonempty and connected. Let $\lambda>0$ and
$g\in C_c(X)$ satisfy $\operatorname{supp}g\subseteq X_n$.
{ Define} $T_n\colon C(X_n)\to C(X_n)$ by
\[
T_n =\operatorname{id}_n+\lambda\Deltadir[n]\Phi.
\]
Then, there exists a unique $\vartheta_n\in C(X_n)$ such that
\[
T_n\vartheta_n=\boldsymbol{\pi}_n g.
\]
Moreover, setting $\varphi_n:=\boldsymbol{\mathfrak i}_n\vartheta_n$, the function $\varphi_n$ is the unique  solution in $C_c(X)$ of
\[
\mathcal{S}_{\lambda,n}\varphi=g
\]
and satisfies
\begin{equation}\label{eq:properties_finite-resolvent-solution}
\operatorname{supp}\varphi_n\subseteq X_n
\qquad \textup{ and } \qquad
\Lmin\varphi_n(x)=\Delta\Phi\varphi_n(x)
\end{equation}
for every $x\in X_n.$

Finally, let $g,h\in C_c(X)$ with $\operatorname{supp}g,\operatorname{supp}h\subseteq X_n$ and let
$\vartheta_n^g,\vartheta_n^h\in C(X_n)$ denote the unique solutions of
\[
T_n\vartheta_n^g=\boldsymbol{\pi}_n g\qquad \text{and} \qquad
T_n\vartheta_n^h=\boldsymbol{\pi}_n h.
\]
Set $\varphi_n^g:=\boldsymbol{\mathfrak i}_n\vartheta_n^g$ and $\varphi_n^h:=\boldsymbol{\mathfrak i}_n\vartheta_n^h$. Then, we get:
\begin{enumerate}[label=\textup{(\alph*)},leftmargin=2.4em]
\item If $g\le h$, then $\varphi_n^g\le \varphi_n^h$.

\item $\|\varphi_n^g-\varphi_n^h\|_1 \leq \|g-h\|_1$.
\end{enumerate}
\end{lemma}

\begin{proof}
{
Write $A_n:=\Deltadir[n]\Phi$ and equip $C(X_n)$ with $\|\cdot\|_{1,n}$, that is, the $\ell^1(X_n,\mu)$ norm. By
Proposition~\ref{p:PMTO_accretivity_Cc} applied to the finite Dirichlet graph
$G_n^{D}$,
$A_n$ is accretive and $A_n0=0$.

\emph{Existence.} Identify $C(X_n)$ with $\mathbb R^{X_n}$ and consider the map
\[
 F\colon\mathbb R^{X_n}\to\mathbb R^{X_n}
 \qquad  \text{given by} \qquad 
 F(\vartheta):=T_n\vartheta-\boldsymbol{\pi}_ng
 =\vartheta+\lambda A_n\vartheta-\boldsymbol{\pi}_ng
\]
which is continuous by the continuity of $\phi$. Set
$R:=1+\|\boldsymbol{\pi}_ng\|_\infty$ and
\[
 Q:=\bigl\{\vartheta\in\mathbb R^{X_n} \mid \ |\vartheta(x)|\leq R\ \text{ for all } x\in X_n\bigr\}.
\]
Let $x\in X_n$. By the definition of the Dirichlet
Laplacian~\eqref{def:dir_laplacian},
\[
 A_n\vartheta(x)
 =\frac{1}{\mu(x)}\sum_{y\in X_n}w(x,y)
 \bigl(\phi(\vartheta(x))-\phi(\vartheta(y))\bigr)
 +\frac{\kappa_n(x)}{\mu(x)}\phi(\vartheta(x)).
\]
If $\vartheta\in Q$ and $\vartheta(x)=R$, then
$\vartheta(x)=R\geq\vartheta(y)$ for every $y\in X_n$. Since $\phi$ is
monotone increasing and $\phi(R)\geq\phi(0)=0$, both the edge terms and
the killing term in the preceding display are nonnegative.
Hence,
$A_n\vartheta(x)\geq 0$ and
\[
 F(\vartheta)(x)=R+\lambda A_n\vartheta(x)-\boldsymbol{\pi}_ng(x)
 \geq R-\|\boldsymbol{\pi}_ng\|_\infty\geq 0 .
\]
Symmetrically, if $\vartheta\in Q$ and $\vartheta(x)=-R$, then
$\vartheta(x)\leq\vartheta(y)$ for every $y\in X_n$ and
$\phi(-R)\leq\phi(0)=0$, so
$A_n\vartheta(x)\leq 0$ and
\[
 F(\vartheta)(x)=-R+\lambda A_n\vartheta(x)-\boldsymbol{\pi}_ng(x)
 \leq -R+\|\boldsymbol{\pi}_ng\|_\infty\leq 0 .
\]
These are exactly the sign conditions of the Poincaré--Miranda theorem~\cite{Mawhin2013} on the
opposite faces of the box $Q$, therefore, $F$ vanishes at some
$\vartheta_n\in Q$, that is,
\[
 T_n\vartheta_n=\boldsymbol{\pi}_ng .
\]

\emph{Uniqueness.} The accretivity of $A_n$ yields
$\|\vartheta-\eta\|_{1,n}\leq\|T_n\vartheta-T_n\eta\|_{1,n}$ for all
$\vartheta,\eta\in C(X_n)$, whence $\vartheta_n$ is the unique solution of
$T_n\vartheta=\boldsymbol{\pi}_ng$.
}

{
Set $\varphi_n:=\boldsymbol{\mathfrak i}_n\vartheta_n$. By the definition of $\Lmin$,
\begin{align*}
    \mathcal{S}_{\lambda,n}\varphi_n
=
(\operatorname{id}+\lambda\mathcal L_n)\boldsymbol{\mathfrak i}_n\vartheta_n &=  
(\operatorname{id} + \lambda\boldsymbol{\mathfrak i}_n\Deltadir[n]\Phi\boldsymbol{\pi}_n)\boldsymbol{\mathfrak i}_n\vartheta_n\\
&=
\boldsymbol{\mathfrak i}_n
(\operatorname{id}_n+\lambda\Deltadir[n]\Phi)\vartheta_n\\
&=
\boldsymbol{\mathfrak i}_n\boldsymbol{\pi}_n g.
\end{align*}
Since $\operatorname{supp}g\subseteq X_n$, one has $\boldsymbol{\mathfrak i}_n\boldsymbol{\pi}_n g=g$ and, therefore,
\[
\mathcal{S}_{\lambda,n}\varphi_n=g.
\]
Uniqueness of $\varphi_n$ as a solution of $\mathcal{S}_{\lambda,n}\varphi=g$ follows from the accretivity of $\Lmin$ established in Proposition~\ref{prop:accretivity_lmin}. The inclusion $\operatorname{supp}\varphi_n\subseteq X_n$ is immediate from the definition of $\boldsymbol{\mathfrak i}_n$. Moreover, $\varphi_n\in\dom(\Lmin)\subseteq\dom(\Deltaphi)$ and, by the definition of the Dirichlet Laplacian~\eqref{def:dir_laplacian}, for $x \in X_n$ we get
\begin{align*}
\Lmin\varphi_n(x)
=\boldsymbol{\mathfrak{i}}_n\Deltadir[n]\Phi\boldsymbol{\pi}_n\varphi_n(x)
=\boldsymbol{\mathfrak{i}}_n\Deltadir[n]\Phi\vartheta_n(x)
=\Delta\Phi\varphi_n(x).
\end{align*}

To prove (a), recall that $\vartheta_n^g,\vartheta_n^h\in C(X_n)$ are the unique solutions of
\[
T_n\vartheta_n^g=\boldsymbol{\pi}_n g
\qquad \text{and} \qquad
T_n\vartheta_n^h=\boldsymbol{\pi}_n h.
\]
By Theorem~\ref{thm:min}, applied to the finite Dirichlet graph
$G_n^{D}$,
the inverse $T_n^{-1}$ is order preserving. Hence,
\[
g\le h
\quad\Longrightarrow\quad
\boldsymbol{\pi}_n g\le \boldsymbol{\pi}_n h
\quad\Longrightarrow\quad
\vartheta_n^g\le \vartheta_n^h.
\]
Applying $\boldsymbol{\mathfrak i}_n$ to both sides gives $\varphi_n^g\le\varphi_n^h$.

To prove (b), write $\|\cdot\|_{1,n}$ for the $\ell^1(X_n,\mu)$ norm. By
Proposition~\ref{p:PMTO_accretivity_Cc} applied to $G_n^{D}$,
$\Deltadir[n]\Phi$ is accretive on $\ell^1(X_n,\mu)$ and, therefore,
\[
\|\vartheta_n^g-\vartheta_n^h\|_{1,n}
\le
\|T_n\vartheta_n^g-T_n\vartheta_n^h\|_{1,n}
=
\|\boldsymbol{\pi}_n g-\boldsymbol{\pi}_n h\|_{1,n}.
\]
Consequently,
\[
\|\varphi_n^g-\varphi_n^h\|_1
=
\|\boldsymbol{\mathfrak i}_n\vartheta_n^g-\boldsymbol{\mathfrak i}_n\vartheta_n^h\|_1
=
\|\vartheta_n^g-\vartheta_n^h\|_{1,n}
\le
\|\boldsymbol{\pi}_n g-\boldsymbol{\pi}_n h\|_{1,n}
=
\|g-h\|_1
\]
since $\operatorname{supp}g, \operatorname{supp}h \subseteq X_n$.
This concludes the proof.
}
\end{proof}

The next lemma states that solutions enjoy a monotonicity property
as we take larger and larger sets in the exhaustion sequence.
\begin{lemma}\label{lem:monotone-resolvent}
Let $X_n\subseteq X_{n+1}$ be finite, connected and nonempty subsets of $X$.
Let $\lambda>0$ and $g\in \ell^1(X,\mu)$. Let $g_k$ be the Dirichlet cutoff of $g$
given by
\[
g_k:=\boldsymbol{\mathfrak i}_k\boldsymbol{\pi}_k g
\]
for $k \in \{n, n+1\}$
so that $g_k\in C_c(X)$ with $\operatorname{supp}g_k\subseteq X_k$. 
Let $\varphi_k\in C_c(X)$, with $\operatorname{supp}\varphi_k \subseteq X_k$, be the unique solution of
\[
\mathcal{S}_{\lambda,k}\varphi_k=g_k
\]
for $k\in\{ n,n+1\}$.
If $g\ge 0$, then $0\le \varphi_n \le \varphi_{n+1}$. If $g\le 0$, then
$\varphi_{n+1} \le \varphi_{n} \le 0$.
\end{lemma}

\begin{proof}
Fix $g\geq 0$. Set $\vartheta_k:=\boldsymbol{\pi}_k\varphi_k\in C(X_k)$ 
and let
\[
T_k:=\operatorname{id}_k+\lambda\Deltadir[k]\Phi
\]
for $k \in \{n,n+1\}$.
By Lemma~\ref{lem:finite-resolvent}, we have $\varphi_k=\boldsymbol{\mathfrak i}_k\vartheta_k$
so that $\operatorname{supp} \varphi_k \subseteq X_k$
and
\[
T_k\vartheta_k=\boldsymbol{\pi}_k g
\]
where we used $\boldsymbol{\pi}_k g_k=\boldsymbol{\pi}_k g$ for $k\in \{n,n+1\}$.

Since $g\ge 0$, we have $g_n\ge 0$. Applying Lemma~\ref{lem:finite-resolvent}~(a) on $X_n$ 
for $0$ and $g_n$ gives
\[
0\le\varphi_n.
\]

Restrict $\varphi_n$ to $X_{n+1}$, note that it is zero outside of $X_n$ and let
\[
\eta_{n+1}:=\boldsymbol{\pi}_{n+1}\varphi_n\in C(X_{n+1}).
\]
Let
\[
h_{n+1}:=T_{n+1}\eta_{n+1}\in C(X_{n+1}).
\]
We claim that
\[
h_{n+1}\le \boldsymbol{\pi}_{n+1} g
\]
on $X_{n+1}$.

Since $\operatorname{supp}\varphi_n\subseteq X_n$, if $x\in X_n$, 
then $\eta_{n+1}(x)=\varphi_n(x)=\vartheta_n(x)$ and if $x\in X_{n+1}\setminus X_n$,
then $\eta_{n+1}(x)=0$. Hence, by the definition of the Dirichlet Laplacian~\eqref{def:dir_laplacian},
\[
\Deltadir[n+1]\Phi \eta_{n+1}(x)
=
\Deltadir[n]\Phi \vartheta_n(x)
\]
for all $x\in X_n$.
Therefore, for $x\in X_n$,
\begin{align*}
   h_{n+1}(x)
=T_{n+1}\eta_{n+1}(x)
&=(\operatorname{id}_{n+1}+\lambda\Deltadir[n+1]\Phi)\eta_{n+1}(x)\\
&=\vartheta_n(x)+\lambda\Deltadir[n]\Phi \vartheta_n(x)\\
&=T_n\vartheta_n(x)\\
&=\boldsymbol{\pi}_n g(x)\\
&=\boldsymbol{\pi}_{n+1} g(x).
\end{align*}
If $x\in X_{n+1}\setminus X_n$, then $\eta_{n+1}(x)=0$, so
\[
h_{n+1}(x)
=
\lambda\Deltadir[n+1]\Phi \eta_{n+1}(x)
=
-\frac{\lambda}{\mu(x)}
\sum_{y\in X_n} w(x,y)\,\Phi \vartheta_n(y).
\]
Since $\vartheta_n\ge 0$ and {$\Phi$ is monotone increasing with $\Phi(0)=0$}, every term in the sum is non-negative. Thus,
\[
h_{n+1}(x)\le 0\le \boldsymbol{\pi}_{n+1} g(x)
\]
for all $x\in X_{n+1}\setminus X_n$.
Hence,
\[
h_{n+1}\le \boldsymbol{\pi}_{n+1} g
\]
on $X_{n+1}$ as claimed.

Now, $\eta_{n+1}$ solves $T_{n+1}\eta_{n+1}=h_{n+1}$ 
and $\vartheta_{n+1}$ solves $T_{n+1}\vartheta_{n+1}=\boldsymbol{\pi}_{n+1} g$. By Lemma~\ref{lem:finite-resolvent}~(a) applied on $X_{n+1}$, we obtain
\[
\eta_{n+1}\le \vartheta_{n+1}.
\]
Applying $\boldsymbol{\mathfrak i}_{n+1}$ and using that $\operatorname{supp}\varphi_n\subseteq X_n\subseteq X_{n+1}$, so that
\[
\boldsymbol{\mathfrak i}_{n+1}\eta_{n+1}
=\boldsymbol{\mathfrak i}_{n+1}\boldsymbol{\pi}_{n+1}\varphi_n
=\varphi_n
\]
we conclude
\[
0\le \varphi_n\le \varphi_{n+1}.
\]
A similar argument works in case $g \leq 0$. This completes the proof.
\end{proof}

The following lemma is a discrete instance of the Kolmogorov–Riesz compactness theorem, see for example~\cite[Theorem 4]{hanche2010kolmogorov}. For the convenience of the reader, we provide a short proof.
\begin{lemma}\label{lem:sequentially-compact}
Let $c\in \ell^1(X,\mu)$ and set
\[
K_c:=\{f\in \ell^1(X,\mu) \mid |f|\leq c\}.
\]
Then, $K_c$ is compact in $\ell^1(X,\mu)$.
\end{lemma}

\begin{proof}
If $c(x)<0$ for some $x$, then $K_c = \emptyset$ and the statement is trivial. Therefore, suppose that $c\geq 0$. Since $X$ is countable, we write $X=\{x_1,x_2,\dots\}$. Let $(f_n)$ be a sequence in $K_c$. For every fixed $j$, the sequence $(f_n(x_j))_n$ is bounded in $\mathbb R$ as $|f_n(x_j)|\leq c(x_j)$. Hence, by a diagonal extraction, there exists a subsequence $(f_{n_k})$ and a function $f\colon X\to\mathbb R$ such that $f_{n_k}(x)\to f(x)$ for every $x\in X$. Clearly, $|f(x)|\leq c(x)$ for all $x$, so $f\in K_c$. Since
\[
|f_{n_k}-f|\leq 2c\in \ell^1(X,\mu)
\]
dominated convergence yields $\|f_{n_k}-f\|_1\to 0$. Thus, every sequence in $K_c$ has a convergent subsequence, therefore, 
$K_c$ is sequentially compact. Since $\ell^1(X,\mu)$ is a metric space, $K_c$ is compact as well.
\end{proof}

We are now in position to prove the main theorem of this subsection.
\begin{theorem}\label{thm:new}
Let $G$ be a  graph. Then, there exists a dense subset $\Omega\subseteq\dom(\Deltaphi)$ such that $\Deltaphi_{|\Omega}$ is m-accretive. In particular, for every $\lambda>0$ and 
every $g\in \ell^{1}(X,\mu)$, there exists a unique $u\in \Omega$ such that
\[
\mathcal{S}_\lambda u=g
\]
and the solution $u$ satisfies the contractivity estimate
\[
\|u\|_1\leq \|g\|_1.
\]
Moreover, if $g\geq 0$, then $u\geq 0$ and if $g\leq 0$, then $u\leq 0$.
\end{theorem}
\begin{proof}
The proof is a significant refinement of the proof of Theorem~1~in~\cite{bianchi2022generalized}.  In Appendix~\ref{sec:proof_Theo1} we compare the two statements and  point out precisely where the new observations occur.

Assume first that $G$ is finite.  In this case, setting $X_n = X$ and $\Omega = C(X)=\dom(\Deltaphi)$, the statement follows from Corollary~\ref{c:PMTO_accretive_finite},
which gives the accretivity of $\Deltaphi$, and Lemma~\ref{lem:finite-resolvent}, 
which gives the surjectivity of $\mathcal{S}_{\lambda,n} = \mathcal{S}_\lambda$, 
the sign-preserving property by part (a) and the contraction estimates by part (b).

Assume now that $G$ is infinite. Fix an exhaustion $(X_n)$ by finite connected sets,  
see \cite[Lemma A.4]{bianchi2022generalized} for the existence of such an exhaustion. 
For each $n$, let the Dirichlet cutoff of $g \in \ell^1(X,\mu)$ be defined by
$$g_n \coloneqq \boldsymbol{\mathfrak{i}}_n \boldsymbol{\pi}_n g$$
so that $g_n \in C_c(X)$, $\operatorname{supp}g_n \subseteq X_n$
and $ \|g_n - g\|_1 \to 0$ as $n \to \infty$.

\medskip
\noindent\textbf{Resolvent notation.}
By Proposition~\ref{prop:accretivity_lmin}, for every $\lambda>0$, the 
operator $\mathcal{S}_{\lambda,n}=\id+\lambda\,\Lmin$ is injective as
$\Lmin$ is accretive on $\ell^1(X,\mu)$ and, 
by Lemma~\ref{lem:finite-resolvent}, its range contains every $g\in C_c(X)$ with $\operatorname{supp} g\subseteq X_n$. We denote the \emph{resolvent operator} of $\Lmin$ by
\[
J^\lambda_n \;:=\; (\id+\lambda\,\Lmin)^{-1}
\]
which, by the injectivity of $\mathcal{S}_{\lambda,n}$,
is well-defined on $\operatorname{Rg}(\mathcal{S}_{\lambda,n})$, the range
of $\mathcal{S}_{\lambda,n}$. The map $(\lambda,g)\mapsto J^\lambda_n g$ highlights the dependence of the solution of $\mathcal{S}_{\lambda,n}\varphi=g$ on the data $g$ and $\lm$
and will allow us to handle the limit $n\to\infty$ in a compact manner.

\medskip
\noindent\textbf{Properties of $J^\lambda_n g_n$.} Applied to $\lambda>0$ and $ g_n$
as above, 
Lemma~\ref{lem:finite-resolvent} yields a unique $\varphi_n\in C_c(X)$ such that $\mathcal{S}_{\lambda,n}\varphi_n=g_n$ which, in the resolvent notation, reads as
\[
\varphi_n = J^\lambda_n g_n.
\]
Combining Lemma~\ref{lem:finite-resolvent} (existence, support, contraction, and order preservation) and Lemma~\ref{lem:monotone-resolvent} (monotonicity along the exhaustion), the family $(J^\lambda_n g_n)_n$ satisfies
\begin{align}
\textup{(i)} \;\; & \mathcal{S}_{\lambda,n}\bigl(J^\lambda_n g_n\bigr) = g_n,
& \textup{(ii)} \;\; & \operatorname{supp}\bigl(J^\lambda_n g_n\bigr)\subseteq X_n, \label{eq:finite-resolvent-1}\\[2pt]
\textup{(iii)} \;\; & \bigl\|J^\lambda_n g_n\bigr\|_1 \le \|g_n\|_1 \le \|g\|_1,
& \textup{(iv)} \;\; & \mbox{If } g\geq 0, \mbox{ then } 0\leq J^\lambda_n g_n \le J^\lambda_{n+1} g_{n+1}. \label{eq:finite-resolvent-2}
\end{align}
Here (iii) uses Lemma~\ref{lem:finite-resolvent}-(b) with $h=0$ together with $\|g_n\|_1\le\|g\|_1$, while (iv) is precisely the conclusion of Lemma~\ref{lem:monotone-resolvent}.

We now divide  the argument into several substeps.

\medskip
\noindent\textbf{Step 1: For every $\lambda>0$ and $g\geq 0$, there exists $u\in \Omega_{(X_n)}$ such that $u \geq 0$ and $\mathcal{S}_\lambda u=g $.}

By \eqref{eq:finite-resolvent-2}-(iv), $J^\lm_n g_n = \varphi_n \geq 0$ since $g \geq 0$.
By \eqref{eq:finite-resolvent-2}-(iii) and \eqref{eq:finite-resolvent-2}-(iv), $(J^\lambda_n g_n)$ is pointwise 
increasing and uniformly bounded in $\ell^1(X,\mu)$ and
hence converges monotonically pointwise to a function $u \geq 0$, that is,
$$
 \varphi_n (x) = J^\lambda_n g_n (x)  \nearrow u(x)
$$
as $n \to \infty$ for every $x\in X$.
By Fatou's lemma 
and the  inequality  \eqref{eq:finite-resolvent-2}-(iii), 
$$
\|u\|_1 \leq  \liminf_{n \to \infty} \|J^\lambda_n g_n\|_1 \leq \|g\|_1,
$$
in particular, $u\in \ell^1(X,\mu)$. By $0\leq J^\lambda_n g_n \leq u$ 
and dominated convergence, we get convergence in $\ell^1(X,\mu)$, i.e., 
\begin{equation}\label{eq:varphi_n-->u}
\|J^\lambda_n g_n - u\|_1 = \|\varphi_n - u \|_1 \to 0
\end{equation}
as $n \to \infty$.
By~\eqref{eq:properties_finite-resolvent-solution} in Lemma~\ref{lem:finite-resolvent}, when $x\in X_n$
equality \eqref{eq:finite-resolvent-1}-(i) reduces to
\begin{align*}
g_n(x) &= \varphi_n(x)+\lambda \Deltaphi_n \varphi_n (x)\\
&= \varphi_n(x)+\lambda \Delta\Phi \varphi_n(x)\\
&= \varphi_n(x)+\lambda\Deg(x)\Phi\varphi_n(x)  -\frac{\lambda}{\mu(x)}\sum_{y\in X}w(x,y)\Phi\varphi_n(y).
\end{align*}

Notice that, as $(X_n)$ is an exhaustion, for every $x\in X$ there exists $N$ large enough such that $x\in X_n$ for every $n\geq N$. { Since $0\leq\varphi_n\nearrow u$ and $\Phi$ is monotone increasing with $\Phi(0)=0$, applying $\Phi$ preserves both sign and monotonicity: $0=\Phi(0)\leq\Phi\varphi_n$ and $\Phi\varphi_n$ is monotone increasing in $n$. Continuity of $\Phi$ then gives $\Phi\varphi_n\to\Phi u$ pointwise, so $0\leq\Phi\varphi_n\nearrow \Phi u$ pointwise.} Rearranging terms, passing to the limit, and invoking monotone convergence, we obtain
\[
\begin{split}
\frac{\lambda}{\mu(x)} \sum_{y \in X} w(x,y) \Phi u (y) & = \lim_{n \to \infty} \frac{\lambda}{\mu(x)} \sum_{y \in X} w(x,y) \Phi \varphi_n (y) \\
&= 
\lim_{n \to \infty} [- g_n(x) + \varphi_n(x)+\lambda \Deg(x) \Phi \varphi_n(x)] \\
&= - g(x) + u(x) +\lambda \Deg(x) \Phi u(x)
\end{split}
\]
for every  $x\in X$,
showing that $u \in \dom(\Delta\Phi)$ and
$\mathcal{S}_\lambda u = (\id + \lambda \Deltaphi) u  = g$. By the fact that $\lambda \Deltaphi u = g - u$ and $g, u \in \ell^1(X, \mu)$, we obtain $\Deltaphi u \in \ell^1(X, \mu)$. Therefore, $u\in \dom(\Deltaphi)$. Moreover,
\begin{align*}
\|\Lmin \varphi_n - \Deltaphi u\|_1 = \frac{1}{\lambda} \| \lambda\Lmin \varphi_n + u - g\|_1 
&\leq \frac{1}{\lambda} \left( \|(\operatorname{id} + \lambda \Lmin) \varphi_n - g\|_1 + \| \varphi_n - u\|_1 \right) \nonumber\\
&=\frac{1}{\lambda} \left( \|g_n - g\|_1 + \| \varphi_n - u\|_1 \right)  \to 0
\end{align*}
as $n \to \infty$
by the definition of $g_n$ and \eqref{eq:varphi_n-->u}.

Summarizing, for every $\lambda >0$ and $g\geq 0$, there exist $u \geq 0$ and a sequence $(\varphi_n)$ such that
\begin{itemize}
    \item $u\in \dom(\Deltaphi)$ and $\operatorname{supp}\varphi_n \subseteq X_n$,
    \item $\lim_{n\to\infty} \| \varphi_n - u\|_1 = 0$,
    \item $\lim_{n\to\infty}   \| \Lmin \varphi_n - \Deltaphi u\|_1 = 0$,
    \item $\mathcal{S}_\lambda u = g$.
\end{itemize}
Therefore, by Definition \ref{def:Omega}, $u \in \Omega_{(X_n)}$ 
and $u$ is a solution of $\mathcal{S}_\lambda u = g$.

We remark that, by using standard comparison principle techniques, the solution
$u$ does not depend on the choice of the exhaustion $(X_n)$ whenever $g \geq 0$. 

\medskip
\noindent\textbf{Step 2: For every $\lambda>0$ and $g\leq 0$, 
there exists $u\in \Omega_{(X_n)}$ such that $u \leq 0$ and $\mathcal{S}_\lambda u=g $.}

This case is analogous to Step 1. If $g \leq 0$, Lemma~\ref{lem:finite-resolvent}-(a) gives $J^{\lambda}_n g_n \leq J^{\lambda}_n 0 = 0$ and the same argument as 
in Lemma~\ref{lem:monotone-resolvent}, with the inequalities reversed, yields
\[
0 \ge J_n^\lambda g_n \ge J_{n+1}^\lambda g_{n+1}.
\]
Proceeding as in Step 1 with the appropriate modifications, 
we obtain a sequence $\varphi_n \leq 0$ decreasing pointwise to a { negative}
$u \in \Omega_{(X_n)}$ satisfying $\mathcal{S}_\lambda u = g$.

\medskip
\noindent
\textbf{Step 3: Choice of a fixed diagonal subsequence $(Y_k)$  of $(X_n)$ and definition of $\Omega \coloneqq \Omega_{(Y_k)}\supseteq \Omega_{(X_n)}$.}

Let $g$ now be sign-changing. We write
\[
g=g^++g^- \quad \mbox{where} \quad
 g^+:=\max\{g,0\}  \text{ and } 
 g^-:=\min\{g,0\}.
\]
Applying Step 1 and  Step 2 to the data $g^+$ and $g^-$, respectively, we obtain monotone sequences 
$$
(J_n^\lambda g_n^+) = (\varphi_n^+) \qquad \mbox{and} \qquad (J_n^\lambda g_n^-) = (\varphi_n^-)
$$
converging in $\ell^1(X,\mu)$ to functions $u^+$ and $u^-$, respectively, with
\[
u^{\pm}\in \Omega_{(X_n)}
\qquad \textup{ and } \qquad
u^-\leq 0\leq u^+.
\]
By monotonicity of the resolvent established in Lemma~\ref{lem:finite-resolvent} (a) and monotone convergence,
\begin{equation}\label{eq:sandwich}
u^-\leq J_n^\lambda g_n^- \leq J_n^\lambda g_n\leq J_n^\lambda g_n^+ \leq  u^+
\end{equation}
for all $n$.
Hence, if we set
\[
c_{\lambda,g}:=|u^-|+|u^+|\in \ell^1(X,\mu),
\]
then
\begin{equation}\label{eq:dominating-function}
|J_n^\lambda g_n|\leq c_{\lambda,g}
\end{equation}
for all $n$.
By Lemma~\ref{lem:sequentially-compact}, the sequence $(J_n^\lambda g_n)$ is therefore relatively compact in $\ell^1(X,\mu)$.

Fix an enumeration $\mathbb Q_+=\{q_j\}_{j\geq 1}$ of the positive rationals and, since $\ell^1(X,\mu)$ is separable,  a countable dense subset~$D=\{g^m\}_{m\geq 1}\subseteq \ell^1(X,\mu)$. For each pair $(q_j, g^m)$, the sequence $(J_n^{q_j}g_n^m)_n$ is relatively compact by~\eqref{eq:dominating-function} and 
Lemma~\ref{lem:sequentially-compact}. 
Hence, by a diagonalization argument, 
there exists a strictly increasing sequence $(n_k)$ such that
\[
(J_{n_k}^{q_j}g_{n_k}^m)_k
\]
converges in $\ell^1(X,\mu)$ for every pair $(q_j, g^m)$.

We now set
\[
Y_k:=X_{n_k}
\]
for $k\geq 1$.
Then $(Y_k)$ is again an exhaustion of $X$ by finite connected sets. From now on, we define $\Omega = \Omega_{(Y_k)}$ by means of this exhaustion $(Y_k)$. Note that, as $(Y_k)$ is a subsequence of $(X_n)$, we get $\Omega_{(Y_k)} \supseteq \Omega_{(X_n)}$.

At this stage, it is important to emphasize that the exhaustion used to define \(\Omega\) is chosen \emph{once and for all}.  In particular, the exhaustion \((Y_k)\), and hence the set \(\Omega\), depend neither on the parameter \(\lambda\) nor on the datum~\(g\).

\medskip
\noindent\textbf{Step 4: For every $\lambda>0$ and every  $g\in \ell^1(X,\mu)$, there exists $u\in \Omega$ such that $\mathcal{S}_\lambda u=g $.}

With the diagonal exhaustion $(Y_k)$ fixed in Step 3, it remains 
to prove the  surjectivity of $\mathcal{S}_\lambda \colon \Omega \to \ell^1(X,\mu)$ 
for arbitrary data \(g\in \ell^1(X,\mu)\) and arbitrary \(\lambda>0\). 
We proceed in three sub-steps: First, we establish the convergence of \((J_{n_k}^{q}g_{n_k})_k\) for rational \(q>0\). Then, we pass from rational parameters to a general \(\lambda>0\) by a continuity estimate in \(\lambda\). Finally, we identify the limit as an element \(u\in \Omega\) solving \(\mathcal{S}_\lambda u=g\) with \(\|u\|_1\le \|g\|_1\).

\medskip
\noindent
\textbf{Step 4 (a): $\ell^1$-convergence of $(J_{n_k}^q g_{n_k})$ for rational $q>0$ and arbitrary data $g \in \ell^1(X,\mu)$.}

Fix $\lambda = q\in \mathbb Q_+$ and $g\in \ell^1(X,\mu)$. Choose a sequence $(g^{m_r})_r\subseteq D$ such that $g^{m_r}\to g$ in $\ell^1(X,\mu)$ as $r \to \infty$. For every $k,\ell$ and every $r$, Lemma~\ref{lem:finite-resolvent}-(b) gives
\begin{align*}
\|J_{n_k}^q g_{n_k}-J_{n_\ell}^q g_{n_\ell}\|_1
&\leq \|J_{n_k}^q g_{n_k} - J_{n_k}^q g_{n_k}^{m_r}\|_1
+\|J_{n_k}^q g_{n_k}^{m_r}-J_{n_\ell}^q g_{n_\ell}^{m_r}\|_1
+\|J_{n_\ell}^q g_{n_\ell}^{m_r}-J_{n_\ell}^q g_{n_\ell}\|_1 \\
&\leq 2\|g-g^{m_r}\|_1+\|J_{n_k}^q g_{n_k}^{m_r}-J_{n_\ell}^q g_{n_\ell}^{m_r}\|_1.
\end{align*}
Since $(J_{n_k}^q g_{n_k}^{m_r})_k$ converges for each fixed $r$, the second term tends to zero as $k,\ell\to\infty$. Hence, $(J_{n_k}^q g_{n_k})_k$ is Cauchy in $\ell^1(X,\mu)$ and, therefore, converges.

\medskip
\noindent
\textbf{Step 4 (b): $\ell^1$-convergence of $(J_{n_k}^\lambda g_{n_k})$ for arbitrary $\lambda >0$ and arbitrary data $g \in \ell^1(X,\mu)$.}

Fix $g\in \ell^1(X,\mu)$ and let, for $\lm, \tau >0$,
\[
\varphi_n^\lambda:=J_n^\lambda g_n
\qquad \text{and} \qquad
\varphi_n^\tau:=J_n^\tau g_n.
\]
Subtracting the two resolvent equations gives
\begin{align*}
   \varphi_n^\lambda-\varphi_n^\tau+\lambda\bigl(\Lmin\varphi_n^\lambda-\Lmin\varphi_n^\tau\bigr)
&= \varphi_n^\lambda + \lambda \Lmin\varphi_n^\lambda - (\varphi_n^\tau + \tau\Lmin\varphi_n^\tau)  +  (\tau - \lambda) \Lmin\varphi_n^\tau\\
&= g_n - g_n +  (\tau - \lambda) \Lmin\varphi_n^\tau \\
&=  (\tau - \lambda) \Lmin\varphi_n^\tau.
\end{align*}
Since $\Lmin$ is accretive by Proposition~\ref{prop:accretivity_lmin}, 
applying the definition of accretivity
and the above identity, we obtain
\begin{equation}\label{eq:step3_1}
 \|\varphi_n^\lambda-\varphi_n^\tau\|_1
\leq \| \varphi_n^\lambda-\varphi_n^\tau+\lambda\bigl(\Lmin\varphi_n^\lambda-\Lmin\varphi_n^\tau\bigr) \|_1 = |\tau-\lambda|\|\Lmin\varphi_n^\tau\|_1
\end{equation}
for every  $\lambda, \tau >0$.
Using the resolvent equation once more, we get
\[
\Lmin\varphi_n^\tau=\frac{g_n  -\varphi_n^\tau}{\tau}
\]
and using the estimate $\|\varphi_n^\tau\|_1 \leq  \|g\|_1$ from 
Lemma~\ref{lem:finite-resolvent}-(b), we get
\begin{equation}\label{eq:step3_2}
\|\Lmin\varphi_n^\tau\|_1\leq \frac{2\|g\|_1}{\tau}.
\end{equation}
Therefore, combining \eqref{eq:step3_1} and \eqref{eq:step3_2}, we get
\begin{equation*}
\|J_n^\lambda g_n-J_n^\tau g_n\|_1
\leq 2\|g\|_1\frac{|\lambda-\tau|}{\min\{\lambda,\tau\}}
\end{equation*}
for all $n$.

Let $\lambda>0$ and $q \in \mathbb{Q}_+$. Then,
\begin{align}
  \|J_{n_k}^\lambda g_{n_k} - J_{n_\ell}^\lambda g_{n_\ell}\|_1 &\leq   \|J_{n_k}^\lambda g_{n_k} - J_{n_k}^q g_{n_k}\|_1 +  \|J_{n_k}^q g_{n_k} - J_{n_\ell}^q g_{n_\ell}\|_1 +  \|J_{n_\ell}^q g_{n_\ell} - J_{n_\ell}^\lambda g_{n_\ell}\|_1\nonumber \\
  &\leq 4\|g\|_1\frac{|\lambda-q|}{\min\{\lambda, q\}}
+ \|J_{n_k}^q g_{n_k} - J_{n_\ell}^q g_{n_\ell}\|_1. \label{eq:continuity-lambda}
\end{align}
Choose $q\in \mathbb Q_+$ close to $\lambda$. Since $(J_{n_k}^q g_{n_k})$ converges 
by Step~4(a), estimate \eqref{eq:continuity-lambda} shows that $(J_{n_k}^\lambda g_{n_k})$ is Cauchy in $\ell^1(X,\mu)$ and hence converges. Denote this limit by $u$, that is,
\begin{equation}\label{eq:varphi_n-->u_2}
 \lim_{k\to \infty} \|  J_{n_k}^\lambda g_{n_k} - u\|_1 = \lim_{k\to \infty} \|  \varphi_{n_k} - u\|_1 =0.
\end{equation}
In particular, from~\eqref{eq:sandwich} notice that
$$
u^- \leq u \leq u^+.
$$

\medskip
\noindent
\textbf{Step 4 (c): $u$ is a solution of $\mathcal{S}_\lambda u=g$, 
$u \in \Omega$, and $\|u\|_1 \leq \|g\|_1$.}

This follows from essentially the same arguments used in Step 1. From \eqref{eq:varphi_n-->u_2} it follows that \(\varphi_{n_k}(x)\to u(x)\) pointwise for every \(x\in X\). Indeed, for every fixed \(x\in X\),
\[
|\varphi_{n_k}(x)-u(x)|\,\mu(x)\le \|\varphi_{n_k}-u\|_1
\]
and, therefore, $\varphi_{n_k}(x)\to u(x)$ as $k\to\infty$. 
As $\phi\colon \mathbb R\to\mathbb R$ is continuous, we also obtain
\begin{equation}\label{eq:continuity}
\Phi\varphi_{n_k}(x)\to \Phi u(x)
\end{equation}
as $k \to \infty$ for every $x\in X$.

Since $u^- \le \varphi_{n_k} \le u^+$ by \eqref{eq:sandwich} and
the fact that {$\phi$ is monotone increasing with $\phi(0)=0$}, we have
\[
\Phi u^- \le \Phi \varphi_{n_k} \le \Phi u^+
\qquad\text{hence }\qquad
|\Phi \varphi_{n_k}| \le |\Phi u^-|+|\Phi u^+|.
\]
As $u^\pm  \in \dom(\Deltaphi) \subseteq \dom(\Delta\Phi)$, it follows that for every $x\in X$,
\begin{equation}\label{eq:dominate_convergence_2}
\sum_{y\in X} w(x,y)|\Phi \varphi_{n_k}(y)|
\le
\sum_{y\in X} w(x,y)\bigl(|\Phi u^-(y)|+|\Phi u^+(y)|\bigr)
<\infty.
\end{equation}
In particular, dominated convergence yields
$$
 \sum_{y\in X} w(x,y)|\Phi  u(y)| = \lim_{k\to\infty} \sum_{y\in X} w(x,y)|\Phi  \varphi_{n_k}(y)| < \infty.
$$
Thus, $\Phi u\in \dom(\Delta)$, that is, $u\in \dom(\Delta\Phi)$.

By~\eqref{eq:properties_finite-resolvent-solution} in Lemma~\ref{lem:finite-resolvent}, when $x\in Y_k$
equality \eqref{eq:finite-resolvent-1}-(i) reduces to
\begin{align*}
g_{n_k}(x) = \varphi_{n_k}(x)+\lambda\Deg(x)\Phi\varphi_{n_k}(x)  -\frac{\lambda}{\mu(x)}\sum_{y\in X}w(x,y)\Phi\varphi_{n_k}(y).
\end{align*}
Notice that, since $(Y_k)$ is an exhaustion, for every $x\in X$ there exists $K$ large enough such that $x\in X_{n_k}$ for every $k\geq K$.  Rearranging terms, passing to the limit, and invoking dominated convergence thanks to \eqref{eq:dominate_convergence_2}, together with the continuity property~\eqref{eq:continuity} we obtain
\begin{align}
\frac{\lambda}{\mu(x)} \sum_{y \in X} w(x,y) \Phi u (y) & = \lim_{k\to\infty} \frac{\lambda}{\mu(x)} \sum_{y \in X} w(x,y) \Phi \varphi_{n_k} (y) \label{eq:extraHP}\\
&= 
\lim_{k\to\infty} [- g_{n_k}(x) + \varphi_{n_k}(x)+\lambda \Deg(x) \Phi \varphi_{n_k}(x)] \nonumber\\
&= - g(x) + u(x) +\lambda \Deg(x) \Phi u(x)  \nonumber
\end{align}
for every  $x\in X$,
showing that
$\mathcal{S}_\lambda u = (\id + \lambda \Deltaphi) u  = g$. By the fact that $\lambda \Deltaphi u = g - u$ and $g, u \in \ell^1(X, \mu)$, we obtain $\Deltaphi u \in \ell^1(X, \mu)$. Therefore, $u\in \dom(\Deltaphi)$. Moreover,
\begin{align}
\|\mathcal{L}_{n_k} \varphi_{n_k} - \Deltaphi u\|_1 = \frac{1}{\lambda} \| \lambda\mathcal{L}_{n_k} \varphi_{n_k} + u - g\|_1 
&\leq \frac{1}{\lambda} \left( \|(\operatorname{id} + \lambda \mathcal{L}_{n_k}) \varphi_{n_k} - g\|_1 + \| \varphi_{n_k} - u\|_1 \right) \nonumber\\
&=\frac{1}{\lambda} \left( \|g_{n_k} - g\|_1 + \| \varphi_{n_k} - u\|_1 \right)  \to 0
\end{align}
as $k \to \infty$
by the definition of $g_{n_k}$ and \eqref{eq:varphi_n-->u_2}.

Summarizing, for every $\lambda >0$ and $g\in \ell^1(X,\mu)$, there exist $u$ and a sequence $(\varphi_{n_k})$ such that
\begin{itemize}
    \item $u\in \dom(\Deltaphi)$ and $\operatorname{supp}\varphi_{n_k} \subseteq Y_k$,
    \item $\lim_{k\to\infty} \| \varphi_{n_k} - u\|_1 = 0$,
    \item $\lim_{k\to\infty}    \| \mathcal{L}_{n_k} \varphi_{n_k} - \Deltaphi u\|_1 = 0$,
    \item $\mathcal{S}_\lambda u = g$.
\end{itemize}
Therefore, by Definition \ref{def:Omega}, $u \in \Omega_{(Y_k)}$ and $u$ 
is a solution of $\mathcal{S}_\lambda u = g$. 
The contraction estimate follows by Fatou's lemma 
and the  inequalities in  \eqref{eq:finite-resolvent-2}-(iii), which give
$$
\|u\|_1 \leq  \liminf_{k \to \infty} \|\varphi_{n_k}\|_1 \leq \|g\|_1.
$$
{ This completes the construction of the solution $u$
and proves all properties stated in Step 4.}

\medskip
\noindent
\textbf{Step 5: $\Omega$ is dense, $\Deltaphiom$ is m-accretive and $u$ preserves the sign of $g$.}

By Lemma~\ref{l:omega_denseness}, $\Omega = \Omega_{(Y_k)}$ 
is dense in $\ell^1(X,\mu)$. 
Since Proposition~\ref{p:accretivityDeltaphiom} applies to any exhaustion, $\Deltaphi_{|\Omega}$ is accretive. Combined with the surjectivity of $\mathcal{S}_\lambda \colon \Omega \to \ell^1(X,\mu)$ established in Step 4(c), this shows that $\Deltaphi_{|\Omega}$ is m-accretive.

The sign-preserving property follows from Step 1 and Step 2, taking into account that $\Omega = \Omega_{(Y_k)}\supseteq \Omega_{(X_n)}$ and that $\mathcal{S}_\lambda \colon \Omega \to \ell^1(X,\mu)$ is bijective.  This completes the proof of the theorem.
\end{proof}

We emphasize that Theorem~\ref{thm:new} does not establish m-accretivity on a set
$\Omega$ tied to an arbitrarily prescribed exhaustion. Rather, it shows that from \emph{any}
initial exhaustion $(X_n)$ one can extract a subsequence $(Y_k)$ for which the
conclusion holds by setting $\Omega = \Omega_{(Y_k)}$;
the resulting solution $u$ may therefore depend on the original choice of
$(X_n)$. In general, one can only guarantee the existence of a dense subset
$\Omega \subseteq \dom(\Deltaphi)$ on which $\mathcal{S}_\lambda \colon \Omega \to
\ell^1(X,\mu)$ is bijective, on the full domain, $\mathcal{S}_\lambda \colon
\dom(\Deltaphi) \to \ell^1(X,\mu)$ need not be injective as we will see. Surjectivity, however, always
holds, yielding the following corollary.
\begin{corollary}\label{cor:S_lambda_surjectivity}
The shifted operator $\mathcal{S}_{\lambda} \colon \dom(\Deltaphi) \to \ell^1(X,\mu)$ is
surjective for every $\lambda >0$.
\end{corollary}
\begin{proof}
    This follows directly from Step~4(c) in the proof of Theorem~\ref{thm:new} above.
\end{proof}

In terms of the injectivity of  $\mathcal{S}_{\lambda}$, we now show that
this implies $\Omega = \Omega_{(X_n)} = \dom(\Deltaphi)$ for \emph{any} exhaustion $(X_n)$. 
Consequently, in this case, the solution $u$ of $\mathcal{S}_\lambda u=g$ is independent of  the choice of the exhaustion $(X_n)$.
We will see further consequences of this injectivity in the next subsection.

\begin{corollary}
    If $\mathcal{S}_{\lambda_0} \colon \dom(\Deltaphi) \to \ell^1(X,\mu)$ 
    is injective for some $\lambda_0>0$, then $\Omega_{(X_n)}=\dom(\Deltaphi)$
    for any exhaustion $(X_n)$.
\end{corollary}
\begin{proof}
Assume that $\mathcal{S}_{\lambda_0}$ is
injective for  $\lambda_0>0$ and fix an exhaustion~$(X_n)$. 
Let $u\in\dom(\Deltaphi)$ and set
\[
g=\mathcal{S}_{\lambda_0}u.
\]

With the notation of Theorem~\ref{thm:new},  let $g_n$ be the Dirichlet cutoff of $g$ and let  $\varphi_n$ be the finitely supported solution of
\[
\mathcal{S}_{\lambda_0,n}\varphi_n=g_n, \qquad \text{with } \operatorname{supp}\varphi_n\subseteq X_n.
\]

By the estimates in Step 3 of the proof of Theorem~\ref{thm:new}, applied to this fixed
exhaustion, the sequence $(\varphi_n)$ is relatively compact in
$\ell^1(X,\mu)$. Let $(\varphi_{n_k})$ be any subsequence. Passing to a further subsequence if
necessary, we may assume that
$\varphi_{n_k}\to v$
in $\ell^1(X,\mu)$ as $k \to \infty$.
Passing to the limit in
\[
\mathcal{S}_{\lambda_0,n_k}\varphi_{n_k}=g_{n_k}
\]
as in Step~4(c) of the proof of Theorem~\ref{thm:new} gives $v\in\dom(\Deltaphi)$ and
\[
\mathcal{S}_{\lambda_0}v=g.
\]
Since $\mathcal{S}_{\lambda_0}$ is injective and $\mathcal{S}_{\lambda_0}u=g = \mathcal{S}_{\lambda_0}v$, we obtain
$v=u$. Thus, every subsequence of $(\varphi_n)$ has a further subsequence
converging to $u$. Hence, by a standard topological argument,
$\varphi_n \to u$
in $\ell^1(X,\mu)$ as $n \to \infty$.

Finally,
\[
\mathcal{L}_n \varphi_n=\frac{g_n-\varphi_n}{\lambda_0}
\to
\frac{g-u}{\lambda_0}
=\Deltaphi u
\]
in $\ell^1(X,\mu)$ as $n \to \infty$.
Therefore, $u\in\Omega_{(X_n)}$. Since $u\in\dom(\Deltaphi)$ was arbitrary, we get
$\dom(\Deltaphi)\subseteq\Omega_{(X_n)}.$
The opposite inclusion is part of the definition of $\Omega_{(X_n)}$, so $\Omega_{(X_n)}=\dom(\Deltaphi)$ as claimed.
\end{proof}

\subsubsection{\texorpdfstring{On the m-accretivity of $\Deltaphi$}{On the m-accretivity of Delta Phi}}

We note that the accretivity of $\Deltaphi$ implies that
the shifted operator $\mathcal{S}_\lambda=\id +\lm \Deltaphi$
is injective on $\dom(\Deltaphi)$ by definition.
A surprising fact, which follows from our previous results, is that the injectivity
of $\mathcal{S}_\lm$ is already sufficient for both the accretivity and the m-accretivity of $\Deltaphi$. 
That is, for $\Deltaphi$, accretivity, m-accretivity, and injectivity of $\mathcal{S}_\lm$ are all equivalent.
This allows us to conclude the m-accretivity of $\Deltaphi$ on $\dom(\Deltaphi)$
under assumptions that imply the injectivity of $\mathcal{S}_\lm$.

We first establish the equivalence of m-accretivity, accretivity
and the injectivity of $\mathcal{S}_\lm$ for the maximal operator
$\Deltaphi$.

\begin{theorem}\label{t:fix} Let $G$ be a graph.
Then the following statements are equivalent:
    \begin{itemize}
        \item[\textup{(i)}] $\Deltaphi$ is m-accretive.
        \item[\textup{(ii)}] $\Deltaphi$ is accretive.
        \item[\textup{(iii)}] $\mathcal{S}_\lambda$ is injective on $\dom(\Deltaphi)$
        for some (equivalently, all) $\lambda >0$.
    \end{itemize}
\end{theorem}
\begin{proof}
The implications (i) $\Longrightarrow$ (ii) $\Longrightarrow$ (iii) are immediate. So, it remains to prove (iii) $\Longrightarrow$ (i). 

Let $\Omega \subseteq \dom(\Deltaphi)$ 
be as in Theorem~\ref{thm:new} so that $\Deltaphiom$ is m-accretive.
Let $\lambda_0 >0$ be such that (iii) $\mathcal{S}_{\lambda_0}$ is injective on $\dom(\Deltaphi)$.
Then, $\mathcal{S}_{\lm_0}\colon \Omega \to \ell^1(X,\mu)$
is onto by m-accretivity while $\mathcal{S}_{\lm_0}\colon \dom(\Deltaphi) \to \ell^1(X,\mu)$ is injective
by assumption.
Since $\Omega \subseteq \dom(\Deltaphi)$, the conclusion follows by Lemma~\ref{l:inclusion}
which gives $\Deltaphi=\Deltaphiom$.
\end{proof}

As an easy consequence of the above, we now observe that injectivity
of the shifted operator is equivalent to the fact that $\Omega$ such
that $\Deltaphiom$ is m-accretive is the entire domain of $\Deltaphi$.

\begin{corollary}
    $\mathcal{S}_\lambda$ is injective for some (equivalently, all) $\lambda>0$ if and only
    if $\Deltaphi = \Deltaphiom$ with $\Omega$ given by Theorem~\ref{thm:new}.
\end{corollary}
\begin{proof}
    If $\mathcal{S}_\lm$ is injective for some $\lm>0$, then $\Deltaphi = \Deltaphiom$
    by Lemma~\ref{l:inclusion} as in the proof of (iii) $\Longrightarrow$ (i)
    above.

    Conversely, if $\Deltaphi = \Deltaphiom$, then $\Deltaphi$ is m-accretive
    since $\Deltaphiom$ is m-accretive by Theorem~\ref{thm:new}. Then, $\mathcal{S}_\lm$
    is injective for all $\lm>0$ by Theorem~\ref{t:fix}.
\end{proof}

With all of the tools at hand, we now harvest results for the maximal
porous medium-type operator. We first start with the case of all infinite
paths having infinite measure.

\begin{theorem}\label{t:lf}
Let $G$ satisfy $\IP$. Then,  $\Deltaphi$  
is m-accretive. In particular, this holds if $G$ satisfies $\UM$.
\end{theorem}
\begin{proof}
Since $\mathcal{S}_\lambda$ acts on
$\dom(\Deltaphi) \subseteq \ell^1(X,\mu)\cap \dom(\Delta \Phi)$, 
if $\IP$ holds, Theorem~\ref{thm:min}~(i) gives that
$\mathcal{S}_\lambda$ is injective. Then, Theorem~\ref{t:fix} implies
that $\Deltaphi$ is m-accretive.
\end{proof}

Next, we use the accretivity result in 
Theorem~\ref{t:accretivity_bd} to conclude that $\Deltaphi$ is m-accretive
on its  domain under bounded edge degree and containment.  
\begin{theorem}\label{t:accretive_bd}
Let $G$ satisfy $\B$ and $\Phi$ satisfy $\C$. Then, 
$\Deltaphi$ is m-accretive.
\end{theorem}
\begin{proof}
By Theorem~\ref{t:accretivity_bd}, we know that $\Deltaphi$ is accretive.
The conclusion now follows by Theorem~\ref{t:fix}.
\end{proof}

\begin{remark}[Non-accretive $\Deltaphi$]
    We point out that there exist examples where $\Deltaphi$ is not accretive and, in particular, 
    $\mathcal{S}_{\lambda}$ is not injective and
    $\Deltaphi \neq \Deltaphiom$. See Theorem~\ref{t:counter} and Remark~\ref{rem:LneqL|omega}. 
\end{remark}

\subsection{\texorpdfstring{The operators $\deltaphim$ and $\Deltaphi$}{The operators Delta Phi min and Delta Phi}}
Summarizing the results in the previous subsection, we see that $\Deltaphi$ is 
m-accretive when infinite paths 
have infinite measure (in particular, when the measure is uniformly bounded from 
below) or when the edge degree is bounded and $\Phi$ preserves $\ell^1(X,\mu)$.
This follows as there always exists a subset $\Omega$ such that the restriction 
$\Deltaphiom$ to this subset is m-accretive 
and $\Deltaphi$ and $\Deltaphiom$
agree under these additional assumptions.
In this subsection, we rather consider the minimal extension
of the restriction of $\Deltaphi$ to the finitely supported 
functions and show that it is m-accretive whenever
it agrees with $\Deltaphi$. We then give a metric
completeness criterion for this agreement.

We let $\Deltaphi_c =\Deltaphi_{|C_c(X)}$ and,
whenever $\Deltaphi_c$ is closable, we  let $\deltaphim$ be the closure
of $\Deltaphi_c$ in $\ell^1(X,\mu)$. That is, $\deltaphim$ has domain
$$\dom(\deltaphim) =\{ u \in \ell^1(X,\mu) \mid \exists \varphi_n \in C_c(X) \textup{ with }
\varphi_n \to u \textup { and } \Deltaphi \varphi_n \to v \textup{ in } \ell^1(X,\mu) \}$$
and acts as $\deltaphim u = v$ for $u \in \dom(\deltaphim)$. 
We note that $\Deltaphi_c$ is always accretive by Proposition~\ref{p:PMTO_accretivity_Cc}.
Therefore, $\deltaphim$ is accretive whenever defined by taking limits. 
This is certainly the case if $\Deltaphi$ is a closed operator in which case 
$\deltaphim \subseteq \Deltaphi$ and thus $\Deltaphi_c$ is closable.

We remark in this connection that unlike what happens in the linear case, a nonlinear densely
defined accretive operator is not necessarily closable. 
Thus, it is not clear if the operator $\deltaphim$ is defined 
without further assumptions on the graph. 
Similarly, it is not clear when $\Deltaphi$ is closed. 
However, we note that whenever $\deltaphim$ is defined and agrees with the maximal
operator, they are automatically m-accretive.
\begin{lemma}\label{l:min_equal_accretive}
If $\Deltaphi_c$ is closable, then $\deltaphim$ is accretive. 
In particular, if $\Deltaphi$ is closed, then $\Deltaphi_c$ is closable
and $\deltaphim \subseteq \Deltaphi$. If $\Deltaphi$ is accretive, then
$\deltaphim =\Deltaphi$ if and only if $\deltaphim$ is m-accretive.
\end{lemma}

\begin{proof}
$\Deltaphi_c$ is accretive by Proposition~\ref{p:PMTO_accretivity_Cc}. Thus, when
$\Deltaphi_c$ is closable, $\deltaphim$ is accretive, being the closure of an accretive
operator by taking limits. Moreover, since
$C_c(X) \subseteq \dom(\Deltaphi)$, we have $\Deltaphi_c \subseteq \Deltaphi$, and hence
$\deltaphim \subseteq \Deltaphi$ whenever $\Deltaphi$ is closed.

If $\Deltaphi$ is accretive and $\Deltaphi=\deltaphim$, then Theorem~\ref{t:fix} gives
m-accretivity of the operators.

Conversely, assume $\deltaphim$ is m-accretive and $\Deltaphi$ is accretive. By
Theorem~\ref{t:fix}, $\Deltaphi$ is m-accretive, hence closed, since m-accretive
operators are always closed, see \cite[III.~Proposition~3.4]{barbu2010nonlinear}. As
we have already shown, closedness of $\Deltaphi$ together with $\Deltaphi_c \subseteq \Deltaphi$
yields $\deltaphim \subseteq \Deltaphi$. Lemma~\ref{l:inclusion} then gives
$\deltaphim =\Deltaphi$.
\end{proof}

In the rest of this subsection, 
we will give various conditions under which these two operators agree.
First, we note the obvious fact that boundedness of degree along with continuity
of $\Phi$ imply this agreement.
\begin{proposition}\label{p:agreement}
    If $G$ satisfies $\BD$ and $\Phi\colon \ell^1(X,\mu) \to \ell^1(X,\mu)$ is
    continuous, then $\Deltaphi = \deltaphim$.
    In particular, $\deltaphim$ is m-accretive.
\end{proposition}
\begin{proof}
    We note $\BD$ implies that $\Delta$ gives a bounded operator
    on $\ell^1(X,\mu)$, e.g., \cite[Theorem~2.15]{keller2021graphs}.
    Thus, $\BD$ along with continuity of $\Phi$ implies that $\Deltaphi$
    is continuous. Therefore, 
    $\Deltaphi$ is closed and the equality $\Deltaphi=\deltaphim$ now follows
    easily from the continuity of $\Deltaphi$ as $C_c(X)$ is dense in $\dom(\Deltaphi)$. 
    The m-accretivity then follows by Lemma~\ref{l:min_equal_accretive} as $\Deltaphi$
    is accretive by Theorem~\ref{t:accretivity_bd}.
\end{proof}

As an example of the above,
in the porous medium case it is sufficient for the degree to be  bounded and 
for $\Phi$ to satisfy the containment property, see Corollary~\ref{c:continuity}. 

We now discuss when $\Deltaphi$ is closed which will imply
that $\Deltaphi_c$ is closable.
We begin by investigating this property 
for $\Deltaphip$ and ultimately specialize to $p=1$.

We recall that for $p \geq 1$, $\Deltaphip$ is the operator with domain
$$\dom(\Deltaphip)= \{u \in \ellp \cap \dom (\Delta \Phi) \mid \Delta \Phi u \in \ellp\}$$ 
which acts as
$$\Deltaphip u(x) = \Delta \Phi u(x)$$
for all $u \in \dom(\Deltaphip)$ and $x \in X$.
For the Laplacian case, that is, when $\phi(s)=s$, we will write $\Delta^{(p)}$ for $\Deltaphip$.
We first give a general lemma
which follows by definitions.

\begin{lemma}\label{l:PT_closed}
Let $p \in [1,\infty]$.
If $\Phi\colon \ell^p(X,\mu) \to \ell^p(X,\mu)$ is continuous
and $\Delta^{(p)}$ is closed, then $\Deltaphip$ is closed.
\end{lemma}
\begin{proof}
This follows by definitions. More specifically, let $u_n \in \dom(\Deltaphip)$
be such that $u_n \to u$ and $\Deltaphip u_n \to v$ in $\ell^p(X,\mu)$.
From $u_n \in \dom(\Deltaphip)$ and $\Phi \colon\ell^p(X,\mu) \to \ell^p(X,\mu)$, we have 
$\Phi u_n \in \dom(\Delta^{(p)})$. Furthermore, since $\Phi$ is continuous,
we have $\Phi u_n \to \Phi u$ and $\Deltaphip u_n \to v$ is equivalent to
$\Delta^{(p)}\Phi u_n \to v$. As $\Delta^{(p)}$
is assumed to be closed, it follows that $\Phi u \in \dom(\Delta^{(p)})$ 
and $\Delta^{(p)}\Phi u =v$. Therefore, $u \in \dom(\Deltaphip)$ and $\Deltaphip u =v$.
\end{proof}

Combining the above along with the closedness of the Laplacian under the infinite measure of paths condition,
we now show that the porous
medium operator is closed if the measure is uniformly bounded below.

\begin{corollary}\label{c:PT_closed}
Let $G$ be a graph satisfying $\IP$ and $p \in [1,\infty)$. 
If $\Phi\colon \ell^p(X,\mu) \to \ell^p(X,\mu)$ is continuous,
then $\Deltaphip$ is closed. In particular, for $\phi(s)=s |s|^{m-1}$ for $m \geq 1$, if $G$ satisfies $\UM$, then $\Deltaphip$
is closed.
\end{corollary}
\begin{proof}
The infinite paths condition $\IP$ implies that $\Delta^{(p)}$ is closed
by Proposition~\ref{p:Lap_accretive} below.
Thus, the first statement follows
by Lemma~\ref{l:PT_closed}.
The statement for the porous medium operator follows as $\UM$ implies $\IP$
and that $\Phi$ is continuous on all $\ell^p$ spaces in this case
by Proposition~\ref{p:PME_measure}.
\end{proof}

In the rest of this subsection,
we will give some conditions under which 
the minimal and maximal 
operators agree and  are m-accretive. For this, we introduce a new geometric condition.

More specifically, we let $\varrho: X\times X \to [0,\infty)$ 
be a pseudo-metric and 
say that $\varrho$ is \emph{$\ell^1$-intrinsic} if
$$\sum_{y \in X}w(x,y)\varrho(x,y) \leq C \mu(x)$$ 
for some $C>0$ and all $x \in X$.
For such a pseudo-metric, we then consider the following condition:
\begin{enumerate}
\item[$\FB$]\hypertarget{ass:FB}{} All distance balls defined with respect to $\varrho$ are finite. \hfill(``Finite balls'')
\end{enumerate}

\smallskip

\begin{remark}[Intrinsic metrics]
We note that the $\ell^1$-intrinsic condition is a 
variant on the usual condition for a metric to be intrinsic on $\ell^2(X,\mu)$
where we take $\varrho$ instead of $\varrho^2$. This usual definition is used for geometric
analysis on $\ell^2(X,\mu)$, see \cite{Kel15, keller2021graphs} and Section~\ref{sec:Laplacian}
for some examples of this. 
Furthermore, the requirement for distance balls to be finite
can be thought of as a metric completeness assumption in the locally finite case
when the metric comes from paths,
see Theorem~11.16 in \cite{keller2021graphs}.
\end{remark}

\begin{example}[Combinatorial graph distance]~\label{ex:combinatorial}
Consider the combinatorial graph distance $d$ for which $d(x,y)=1$
for all $x \sim y$. Then $d$ is $\ell^1$-intrinsic if and only if
the bounded edge degree
condition $\B$ holds. Furthermore,
$d$ has finite balls if and only if the graph is locally finite.
\end{example}

Under the assumptions of an $\ell^1$-intrinsic metric having finite balls
and uniformity of measure,
we will show that the minimal operator agrees with the maximal operator for
the porous medium equation by using cutoff functions.  

We begin with a lemma concerning cutoff functions.
We let $\varrho$ be an $\ell^1$-intrinsic pseudo-metric as above. We let $R>0$, $x_0 \in X$ and let
$\eta_R \colon X \to [0, 1]$ denote the cutoff function given by
$$\eta_R(x) = \left(1 - \frac{\varrho(x_0, x)}{R} \right)_+$$
for $x \in X$.
We note that $\eta_R \leq 1_{B_R(x_0)}$ where
$B_R(x_0)=\{ x \in X \mid \varrho(x, x_0) \leq R \}$. In particular,
assuming the finite distance balls condition implies
that $\eta_R \in C_c(X)$.

For $x,y \in X$ and $f \in C(X)$, we let
$$\nabla_{x,y} f = f(x)-f(y)$$
and
$$|\nabla f|(x) = \frac{1}{\mu(x)} \sum_{y \in X} w(x,y)|\nabla_{x,y}f|= \frac{1}{\mu(x)} \sum_{y \in X} w(x,y)|f(x)-f(y)|.$$
With these notations we then have the following basic result.

\begin{lemma}\label{l:cutoff}
Let $\phi\colon \R \to \R$ be Lipschitz on $[0,1]$ and let $\varrho$ be an $\ell^1$-intrinsic pseudo-metric.
For
\[
\eta_R(\cdot) = \left(1 - {\varrho(x_0, \cdot)}/{R} \right)_+
\]
there exists $C>0$ such that
\[
| \nabla \Phi\eta_R | \leq \frac{C}{R}.
\]
\end{lemma}
\begin{proof}
As $0 \leq \eta_R \leq 1$ and $\phi$ is Lipschitz on $[0,1]$,
this is a straightforward calculation using the Lipschitz condition,  
the triangle inequality, and the $\ell^1$-intrinsic condition. More specifically,
we let $C_1$ be the Lipschitz constant of $\phi$ on $[0,1]$ and calculate as follows:
\begin{align*}
|\nabla \Phi \eta_R |(x) &=  
\frac{1}{\mu(x)} \sum_{y \in X}  w(x,y) \left| \Phi \eta_R(x) - \Phi\eta_R(y)\right| \\
&\leq \frac{C_1}{\mu(x)} \sum_{y \in X}  w(x,y) \left| \eta_R(x) - \eta_R(y)\right| \\
& \leq \frac{C_1}{\mu(x)} \sum_{y \in X} w(x,y)\left|\frac{\varrho(x_0,x) -\varrho(x_0,y)}{R}\right|  \\
&\leq \frac{C_1}{R}\left(\frac{
1}{\mu(x)} \sum_{y \in X} w(x,y)\varrho(x,y)\right) \leq \frac{C}{R}.
\end{align*}

This completes the proof.
\end{proof}

We now give a formula for
the Laplacian applied to products
of functions. The proof consists of straightforward calculations. 

\begin{lemma}\label{l:product}
Let $f, g \in \dom (\Delta)$ be such that $fg \in \dom (\Delta)$ and let $x \in X$. Then,
\begin{align*}
\Delta(fg)(x) &= f(x) \Delta g(x) + \frac{1}{\mu(x)}
\sum_{y \in X} w(x,y) g(y) \nabla_{x,y}f.
\end{align*}
\end{lemma}

We now turn our attention to showing that the minimal and maximal operators
agree.
We first show that a function cutoff by $\eta_R$ converges in $\ell^1$ for the Laplacian and 
porous medium operator cases, i.e., when $\phi(s) = s |s|^{m-1}$ for $m\geq1$,
under the finite balls assumption.
Here, we have to assume that $\Phi$ maps $\ell^1(X,\mu)$ to itself, i.e., 
condition $\C$. We recall from Proposition \ref{p:PME_measure} that if $m>1$, then the containment condition
$\C$ is actually equivalent to the uniform lower bound on the measure
condition $\UM$.

\begin{lemma}\label{l:main_approx}
Let $\phi\colon \R \to \R$ by $\phi(s) = s |s|^{m-1}$ for $m\geq1$
and assume that $\Phi$ satisfies $\C$.
Let $\varrho$ be an $\ell^1$-intrinsic metric satisfying $\FB$ 
and let $\eta_R(\cdot) = \left(1 - {\varrho(x_0, \cdot)}/{R} \right)_+$. 
For $u \in \dom(\Deltaphi)$, we have
$$ u\eta_R \to u \qquad \text{ and } \qquad \Deltaphi(u\eta_R) \to \Deltaphi u$$
in $\ell^1(X,\mu)$ as $R \to \infty$.
Consequently, if $\Deltaphi_c$ is closable, then $\dom(\Deltaphi) \subseteq \dom(\deltaphim)$ and
if $\Deltaphi$ is closed, then $\Deltaphi=\deltaphim$ and both operators are m-accretive.
\end{lemma}
\begin{proof}
It is clear that $|(u\eta_R)(x)| \nearrow |u(x)|$ as $R \to \infty$ for $u \in \dom(\Deltaphi)$
and all $x \in X$ and thus $ \|u \eta_R - u \|_1 \to 0$ as $R \to \infty$ by Lebesgue's dominated convergence theorem.

We now show that $\|\Deltaphi (u \eta_R) - \Deltaphi u\|_1 \to 0$ as $R \to \infty$ as well.
From the fact that $\phi(st)=\phi(s)\phi(t)$ which yields
$\Phi (u \eta_R) =\Phi u \Phi \eta_R$ and the formula in Lemma~\ref{l:product}, we have
for all $x \in X$ and $u \in \dom(\Deltaphi)$
$$
\Deltaphi(u\eta_R)(x) = \Delta (\Phi \eta_R \Phi u)(x)= \Phi \eta_R(x) \Delta \Phi u(x) +
\frac{1}{\mu(x)} \sum_{y \in X} w(x,y) \Phi u(y)\nabla_{x,y}\Phi \eta_R.$$
Therefore,
\begin{align*}
\Deltaphi u(x)-\Deltaphi(u \eta_R)(x)
&= \left(1 - \Phi \eta_R(x) \right) \Delta \Phi u(x) -
\frac{1}{\mu(x)} \sum_{y \in X} w(x,y) \Phi u(y)\nabla_{x,y}\Phi \eta_R.
\end{align*}

We now show that each of the terms above goes to $0$ in $\ell^1(X,\mu)$ as $R \to \infty$.
For the first term, we note that $0 \leq \eta_R(x) \leq 1$ and 
$\eta_R(x) \to 1$ as $R \to \infty$ for every $x \in X$.
Therefore,
$0 \leq \Phi \eta_R(x) \leq 1$ and
$\Phi \eta_R(x) \to 1$ as $R \to \infty$.
As $u \in \dom(\Deltaphi)$,
we have $\Delta \Phi u \in \ell^1(X,\mu)$ and thus
$\| \left(1 - \Phi \eta_R \right) \Delta \Phi u\|_1 \to 0$ as $R \to \infty$ by
Lebesgue's dominated convergence theorem.

For the second term, we let $T_R(x)= \frac{1}{\mu(x)} \sum_{y \in X} w(x,y) \Phi u(y)\nabla_{x,y}\Phi \eta_R$ 
and use Tonelli's theorem and Lemma~\ref{l:cutoff} as follows:
\begin{align*}
\|T_R\|_1 &\leq \sum_{x \in X} \sum_{y \in X} w(x,y) |\Phi u(y)| |\nabla_{x,y}\Phi \eta_R|   = \sum_{y \in X} |\Phi u(y)| \sum_{x \in X} w(x,y)  |\nabla_{x,y}\Phi \eta_R| \\
&= \sum_{y \in X} |\Phi u(y)| \mu(y) |\nabla \Phi \eta_R |(y) \leq \frac{C}{R} \|\Phi u \|_1.
\end{align*}
This shows that $T_R \to 0$ in $\ell^1(X,\mu)$ as $R \to \infty$ and completes the proof of the
first statement in the theorem. The remaining statements follow by definitions and
Lemma~\ref{l:min_equal_accretive}.
\end{proof}

We now come to the first major result
for the Laplacian case, i.e., when $\phi(s)=s$.
\begin{theorem}\label{t:Laplacian_min_max}
Suppose there exists an $\ell^1$-intrinsic metric $\varrho$ satisfying $\FB$.
If $\Delta^{(1)}$ is closed, then
$$\Delta^{(1)}=\deltam^{(1)}$$
and
$\Delta_{\min}^{(1)}$ is m-accretive. In particular, the conclusion holds if $G$ additionally satisfies
$\IP$ or \EM[] or $\B$.
\end{theorem}
\begin{proof}
As mentioned previously, $C_c(X) \subseteq \dom(\Delta^{(1)})$
and $\Delta_c$ is always closable on $\ell^1(X,\mu)$, see Lemma~\ref{l:min_accretivity}
and note that \EM[1] always holds.
Now, $\FB$ with respect to an $\ell^1$-intrinsic pseudo-metric gives
$\Delta^{(1)} = \Delta_{\min}^{(1)}$ and m-accretivity by Lemma~\ref{l:main_approx}
as we assume that $\Delta^{(1)}$ is closed.

The ``in particular'' statements follow since
if $G$ satisfies $\IP$, then $\Delta^{(1)}$ is m-accretive by
Theorem~\ref{t:lf} and any m-accretive operator
is closed, see \cite[III.~Proposition~3.4]{barbu2010nonlinear}.
If $G$ satisfies $\EM$, then $\Delta^{(1)}$ is closed by Lemma~\ref{lem:EMp}.
Finally, $\B$ implies $\EM$, see Corollary~\ref{c:emp}, so any of
$\IP$ or \EM[] or $\B$ imply that $\Delta^{(1)}$ is closed.
\end{proof}

\begin{remark}[Boundedness of the operator]
    We note that if $G$ satisfies $\BD$, it follows that $\Delta^{(1)}$ is a bounded
    operator on $\ell^1(X,\mu)$ from which it follows easily that $\Delta^{(1)}=\deltam^{(1)}$,
    see Proposition~\ref{p:agreement}. 
    The result above shows that $\BD$ can be replaced by $\B$ 
    by adding the finite balls assumption $\FB$.
\end{remark}

Next, we put the preceding results together to 
show that the minimal and maximal operators agree for the porous
medium operator when we have lower bounded measure and finite balls
with respect to an $\ell^1$-intrinsic metric.

\begin{theorem}\label{t:min=max}
Suppose there exists an $\ell^1$-intrinsic metric $\varrho$ satisfying $\FB$.
If $G$ satisfies $\UM$ and $\phi\colon \R \to \R$ by $\phi(s) = s |s|^{m-1}$ for $m\geq1$, then
$$ \Deltaphi= \deltaphim.$$
In particular, $\deltaphim$ is m-accretive.
\end{theorem}
\begin{proof}
If $m=1$, then this follows by Theorem~\ref{t:Laplacian_min_max} above
since $\UM$ implies $\IP$.
If $m>1$, then the proof follows along the same lines as the proof of Theorem~\ref{t:Laplacian_min_max}.
In particular, $\Deltaphi$ is closed by Corollary~\ref{c:PT_closed} so that $\Deltaphi_c$
is closable and $\Deltaphi_{\min} \subseteq \Deltaphi$.
The fact that $\Deltaphi=\Deltaphi_{\min}$ and that both operators
are m-accretive now follows by Lemma~\ref{l:main_approx}
as we assume $\FB$ for an $\ell^1$-intrinsic metric. 
\end{proof}

Finally, we give a consequence for the combinatorial
graph metric, i.e., the path metric arising when we set the length of each edge 
to be $1$. In this case, the $\ell^1$-intrinsic condition translates to 
bounded edge degree condition and finiteness of balls translates to local finiteness
as pointed out in Example~\ref{ex:combinatorial}.
Thus, we get the following immediate corollary.

\begin{corollary}\label{c:graph_metric}
Let $G$ be locally finite and satisfy $\UM$ and $\B$.
Let $\phi\colon \R \to \R$ by $\phi(s) = s |s|^{m-1}$ for $m\geq1$.
Then,
$$\Deltaphi = \deltaphim.$$
In particular, $\deltaphim$ is  m-accretive.
\end{corollary}
\begin{proof}
    The result follows by letting $\varrho$ be the combinatorial
    graph metric in Theorem~\ref{t:min=max}.
\end{proof}

\begin{remark}[Boundedness]
    We note that if $G$ satisfies $\UM$ and $\BD$ in the above,
    we also get the same conclusion as $\Phi$ is continuous and $\Delta$ is bounded
    on $\ell^1(X,\mu)$ in this case, see Proposition~\ref{p:agreement}. 
    In the corollary above, $\BD$ is replaced by the weaker assumption
    $\B$ by adding the assumption of local finiteness.
\end{remark}

\section{\texorpdfstring{Accretivity and m-accretivity of the Laplacian on $\ell^p$}{Accretivity and m-accretivity of the Laplacian on ellp}}\label{s:Laplacian_accretivity}
We now present some results for the maximal Laplacian $\Delta^{(p)}$ acting as $\Delta$ on
$$\dom(\Dep) = \{ u \in \ell^p(X,\mu) \cap \dom(\Delta) \mid \Delta u \in \ell^p(X,\mu)\}$$
for $p \in [1,\infty]$ as well as various restrictions of $\Dep$.
In particular, we first adapt the definition of the set $\Omega$ to $\ell^p$ to get 
$\Omega^p$ and show that the restriction
of $\Delta^{(p)}$ to $\Omega^p$ is always m-accretive. We then prove that 
accretivity, m-accretivity and
injectivity of the shifted operator are all equivalent for $\Delta^{(p)}$, 
in parallel to what we
established for $\Deltaphi$ in the previous section. 
As a consequence, we show that the infinite measure of paths condition
implies m-accretivity of $\Delta^{(p)}$ for $p \in[1,\infty)$
and that the non-summability of the reciprocal of the vertex degree over paths
implies the m-accretivity of $\Delta^{(\infty)}$. 
Finally, we discuss how a finiteness of balls condition for $\ell^2$-intrinsic
metrics also implies m-accretivity of $\Delta^{(p)}$ for $p \in (1,\infty)$.

\subsection{\texorpdfstring{On the set $\Omega^p$ and the accretivity of $\Delta^{(p)}_{|\Omega^p}$}{On the set Omegap and the accretivity of Delta p on Omegap}}
In this subsection, we introduce the set $\Omega^p$ and show that 
the restriction of $\Dep$ to $\Omega^p$ is always accretive.

We first adapt the set $\Omega$ and the operator $\Lmin$ introduced in Section~\ref{SSect:Omega} to the operators $\Delta^{(p)}$. Given an exhaustion $(X_n)$, we define 
  \begin{equation}\label{eq:Omegap}
  \Omega^p = \{ u\in \dom(\Delta^{(p)}) \mid \exists \varphi_n \text{ with } \operatorname{supp}\varphi_n\subseteq X_n,\; \varphi_n \to u \text{ and } \Delta_n \varphi_n 
  \to \Delta^{(p)}u \text{ in } \ell^p(X,\mu)\}
  \end{equation} 
where the operator $\Delta_n\colon C_c(X) \to C_c(X)$, with domain $\dom(\Delta_n) = C_c(X)$,
is defined by 
\begin{equation}\label{eq:exhaustion_Lap}
 \Delta_n \varphi = \boldsymbol{\mathfrak{i}}_{n}\Deltadir[n]\boldsymbol{\pi}_{n} \varphi
\end{equation}
and $\Deltadir[n]$ is the Dirichlet Laplacian of the subgraph $X_n$ as defined in \eqref{def:dir_laplacian}.

We note that, in contrast to the nonlinear case of $\Deltaphi$, where in general one
must pass to a subsequence $(Y_k)$ of $(X_n)$ to construct a set
$\Omega=\Omega_{(Y_k)}$ on which $\Deltaphi_{|\Omega}$ is m-accretive, the proof below
shows that every exhaustion $(X_n)$ already yields such a set
$\Omega^p = \Omega^p_{(X_n)}$ on which $\Delta^{(p)}_{|\Omega^p}$ is m-accretive,
without passing to a subsequence.

By linearity, $\Omega^p$ is a vector subspace of $\dom(\Delta^{(p)})$. Moreover, unlike
the case of $\ell^1(X,\mu)$ discussed in Section~\ref{SSect:Omega}, the set $C_c(X)$ is
not necessarily contained in $\dom(\Delta^{(p)})$ for $p>1$, since the Laplacian need not
map finitely supported functions into $\ell^p(X,\mu)$. In particular,
Lemma~\ref{lem:EMp} shows that the condition \EM[p] is equivalent to the inclusion
$C_c(X) \subseteq \dom(\Delta^{(p)})$. We now show that \EM[p] is also equivalent to the inclusion
$C_c(X) \subseteq \Omega^p$.

\begin{lemma}\label{lem:Cc_density_Omegap}
Let $p\in [1,\infty)$.  Then, $C_c(X) \subseteq \Omega^p$ if and only if \EM[p] holds.
\end{lemma}
\begin{proof}
If $C_c(X) \subseteq \Omega^p$, then $C_c(X) \subseteq \dom(\Dep)$ by the definition of $\Omega^p$. Thus, $\Delta(C_c(X)) \subseteq \ell^p(X,\mu)$ and  \EM[p] holds by Lemma~\ref{lem:EMp}.

Conversely, assume \EM[p] and fix $\varphi\in C_c(X)$. Then, $\varphi \in \dom(\Delta)$ and,  
by Lemma~\ref{lem:EMp}, $\Delta \varphi \in \ell^p(X,\mu)$
by \EM[p]. Therefore, $\varphi \in\dom(\Delta^{(p)})$.

Now choose $N$ such that $\operatorname{supp} \varphi \subseteq X_N$ and let 
$$
\varphi_n \coloneqq \begin{cases}
    0 &\mbox{if } n < N,\\
    \varphi &\mbox{otherwise}.
\end{cases}
$$
Clearly, by construction, $\operatorname{supp} \varphi_n \subseteq X_n$ for every $n$ and $\|\varphi_n - \varphi\|_p\to 0$ as $n\to\infty$. It remains to show that $\|\Delta_n\varphi_n - \Delta^{(p)} \varphi \|_p\to 0$
as $n \to \infty$.

For $n \geq N$, we have $\varphi_n=\varphi$ and $\varphi=0$ on $X\setminus X_n$. Therefore, for $x\in X_n$,
$$
\Delta_n \varphi(x) = \Deltadir[n](\boldsymbol{\pi}_n \varphi)(x).
$$
Using formula \eqref{def:dir_laplacian} in the definition of the Dirichlet Laplacian, we get
for $x \in X_n$
$$
\Deltadir[n](\boldsymbol{\pi}_n \varphi)(x) = \frac{1}{\mu(x)} \sum_{y\in X_n} w(x,y)\bigl(\varphi(x)-\varphi(y)\bigr) + \frac{1}{\mu(x)} \left( \sum_{y\notin X_n} w(x,y)+\kappa(x) \right)\varphi(x).
$$
Since $\varphi(y)=0$ for $y\notin X_n$, this is exactly
$$
\frac{1}{\mu(x)} \sum_{y\in X} w(x,y)\bigl(\varphi(x)-\varphi(y)\bigr) + \frac{\kappa(x)}{\mu(x)}\varphi(x) = \Delta \varphi(x).
$$
Thus, for $n \geq N$,
$$
\Delta_n\varphi_n=\Delta_n \varphi= \boldsymbol{\mathfrak{i}}_{n}\Deltadir[n]\boldsymbol{\pi}_{n} \varphi = \mathbf{1}_{X_n}\Delta \varphi.
$$
Consequently,
$$\|\Delta_n\varphi_n-\Delta \varphi\|_p = \|\mathbf{1}_{X\setminus X_n}\Delta \varphi\|_p\to 0 
$$
as $n \to \infty$ since $\Delta \varphi\in \ell^p(X,\mu)$ and $(X_n)$ is an exhaustion.
\end{proof}

As a counterpart to Proposition~\ref{prop:accretivity_lmin}
for $\mathcal{L}_n$ on $\ell^1(X,\mu)$ we next discuss how $\Delta_n$ is accretive for all $n$.
\begin{proposition}\label{prop:accretivity_deltamin}
The operator $\Delta_n$ is accretive on $\ell^p(X,\mu)$ for every $n$ and $p\in[1,\infty)$.
\end{proposition}
\begin{proof}
    This follows by a variant on the proof of Proposition~\ref{p:Lap_accretivity_Cc}.
    More specifically, for $p=1$, this follows from Proposition~\ref{prop:accretivity_lmin} with $\Phi=\id$.
    For $p \in (1,\infty)$, 
    we get that
    $$\sum_{x \in X} \Delta_n\varphi(x)  \varphi(x) |\varphi(x)|^{p-2} \mu(x)=
    \sum_{x \in X_n} \Deltadir[n] \varphi(x)  \varphi(x) |\varphi(x)|^{p-2} \mu(x) \geq 0$$
    which follows by Corollary~\ref{c:accretive_finite_Lap}
    as $\Deltadir[n]$ is a Laplacian on a finite graph.
\end{proof}

We denote by $\Delta^{(p)}_{|\Omega^p}$ the restriction of $\Delta^{(p)}$ to $\Omega^p$ as defined in~\eqref{eq:Omegap}. We then have the following result.
\begin{proposition}\label{p:accretivityDeltap}
   $\Delta^{(p)}_{|\Omega^p}$ is accretive for every $p\in [1,\infty)$. 
\end{proposition}
\begin{proof}
This follows from the same proof as that of 
Proposition~\ref{p:accretivityDeltaphiom} by using the accretivity of $\Delta_n$ established in Proposition~\ref{prop:accretivity_deltamin} and taking limits.
\end{proof}

\subsection{\texorpdfstring{On the m-accretivity of $\Delta^{(p)}_{|\Omega^p}$ for $p\in[1,\infty)$}{On the m-accretivity of Delta p on Omegap for p in [1,infty)}}\label{sec:m-accretivity_Deltap_omega}
In this subsection we consider the restriction of $\Delta^{(p)}$ to $\Omega^p$ for $p \in [1,\infty)$.
In parallel to what happens for $\Deltaphi$ on $\ell^1(X,\mu)$, we show that
this restriction is always m-accretive.

We first  adapt some of the notations and  preliminary lemmas of Subsection~\ref{sec:m-accretivity_Lomega} to 
the $\ell^p$ setting for the Laplacian. As usual, we consider an exhaustion $(X_n)$ and let
$\Delta_n$ denote the operator based on the Dirichlet Laplacian on $X_n$.
Given $\lambda>0$, for notational convenience, we let 
$$
S^{(p)}_\lambda = \id + \lambda \Delta^{(p)} \qquad \mbox{and} \qquad S_{\lambda,n} = \id + \lambda \Delta_n.
$$
With these notations, we now state a result on the existence of solutions on finite subgraphs.

\begin{lemma}\label{lem:finite-resolvent2}
Let $X_n$ be a finite, connected and nonempty subset of $X$, let
$\lambda>0$ and $p\in[1,\infty)$. For every $g\in\ell^p(X,\mu)$
with $\operatorname{supp}g\subseteq X_n$, there exists a unique
$\varphi_n^g\in\dom(\Delta_n)$ such that $\operatorname{supp}\varphi_n^g\subseteq X_n$ and
\[
S_{\lambda,n}\varphi_n^g=g.
\]

The map $g\mapsto\varphi_n^g$ satisfies, for all
$g,h\in\ell^p(X,\mu)$ with $\operatorname{supp}g,\operatorname{supp}h\subseteq X_n$, the
following:
\begin{enumerate}[label=\textup{(\alph*)},leftmargin=2.4em]
\item If $g\le h$, then $\varphi_n^g\le \varphi_n^h$.
\item $\|\varphi_n^g-\varphi_n^h\|_p\leq \|g-h\|_p$.
\end{enumerate}
Finally, let $X_n\subseteq X_{n+1}$ be finite, connected and nonempty subsets of $X$ and let the 
Dirichlet cutoffs be given by
$g_k:=\boldsymbol{\mathfrak i}_k\boldsymbol{\pi}_k g$ for $k\in\{n,n+1\}$. Then the following holds:
\begin{enumerate}[label=\textup{(\alph*)},leftmargin=2.4em,start=3]
\item If $g\ge 0$, then $0\le \varphi_n^{g_n}\le \varphi_{n+1}^{g_{n+1}}$ and if $g\le 0$,
then $\varphi_{n+1}^{g_{n+1}}\le \varphi_n^{g_n}\le 0$.
\end{enumerate}
\end{lemma}
\begin{proof}
The proof proceeds along the same lines as those of Lemmas~\ref{lem:finite-resolvent}
and~\ref{lem:monotone-resolvent} in the nonlinear case for $\ell^1$, with
$\Phi=\operatorname{id}$ and the following substitutions. Existence in
$\dom(\Delta_n)$, together with the support property, follow as before. Furthermore,
the accretivity of $\Lmin$ on $\ell^1(X_n,\mu)$ is replaced by that of $\Delta_n$ on
$\ell^p(X,\mu)$ given by Proposition~\ref{prop:accretivity_deltamin}, which yields
uniqueness in $\dom(\Delta_n)$.  The order preservation~(a) follows from
{ Theorem~\ref{thm:min} applied to the finite Dirichlet graph
$G_n^{D}$ associated with $X_n$}. The contraction
estimate~(b) is obtained by replacing the accretivity of $\Deltadir[n]\Phi$ on
$\ell^1(X_n,\mu)$ with that of $\Deltadir[n]$ on $\ell^p(X_n,\mu)$ established
via Corollary~\ref{c:accretive_finite_Lap}.  Part~(c) follows from the same extension-by-zero argument
as in Lemma~\ref{lem:monotone-resolvent}.
\end{proof}

Given the lemma above,
we now establish the m-accretivity of $\Delta^{(p)}_{|\Omega^p}$ for all $p \in [1,\infty)$
in parallel to that of $\Deltaphi$ on $\ell^1(X,\mu)$ given in Theorem~\ref{thm:new}.
\begin{theorem}\label{thm:new2}
Let $G$ be a  graph and  $p\in [1, \infty)$. Then, there exists a dense subset $\Omega^p\subseteq\dom(\Delta^{(p)})$ such that $\Delta^{(p)}_{|\Omega^p}$ is m-accretive. In particular, for every $\lambda>0$ and every $g\in \ell^{p}(X,\mu)$, there exists a unique $u\in \Omega^p$ such that
\[
S^{(p)}_\lambda u=g
\]
and the solution $u$ satisfies the contractivity estimate
\[
\|u\|_p\leq \|g\|_p.
\]
Moreover, if $g\geq 0$, then $u\geq 0$, and if $g\leq 0$, then $u\leq 0$.
\end{theorem}
\begin{proof}
The proof follows the same scheme as the proof of Theorem~\ref{thm:new} and is mostly identical. The linearity of $\Delta^{(p)}$ simplifies the argument and removes the need for the diagonal 
subexhaustion introduced in Step~3 of that proof. We briefly sketch the details for the convenience of the reader.

If $G$ is finite, the conclusion is immediate by Corollary~\ref{c:accretive_finite_Lap} and Lemma~\ref{lem:finite-resolvent2}. Hence, we assume that $G$ is infinite, fix an exhaustion $(X_n)$ by finite connected sets and let $\Omega^p=\Omega^p_{(X_n)}$.

Let $\lambda>0$ and $g\in\ell^p(X,\mu)$. Define as usual the Dirichlet cutoff of $g$ by
\[
g_n:=\boldsymbol{\mathfrak{i}}_n\boldsymbol{\pi}_n g\in C_c(X)
\]
so that
$$\operatorname{supp} g_n\subseteq X_n \qquad \textup{and} \qquad \|g_n - g\|_p \to 0$$
as $n \to \infty$.
From Lemma~\ref{lem:finite-resolvent2}, let $\varphi_n$ be the solution of the finite-set equation. The key ingredients are: 
\begin{equation}\label{eq:linear_1}
S_{\lambda,n}\varphi_n=g_n,\qquad
\operatorname{supp}\varphi_n\subseteq X_n, \qquad \|\varphi_n\|_p\leq  \|g\|_p, \qquad \mbox{and} \qquad  \begin{cases}
    0\leq \varphi_n\leq \varphi_{n+1} & \mbox{ if } \; g\geq 0,\\
    \varphi_{n+1} \leq \varphi_n\leq 0  & \mbox{ if } \; g\leq 0.
\end{cases}
\end{equation}

Assume first that $g\geq 0$. By the above properties and arguing as in Step 1 of the
proof of Theorem~\ref{thm:new},  $\varphi_n$ converges pointwise monotonically to some $u\geq 0$. In particular, $u\in \Omega^p$, $\|u\|_p \leq \|g\|_p$ and $u$ is a solution of $S^{(p)}_\lambda u=g$.

For general $g \in \ell^p(X,\mu)$, define
\[
g^+=\max\{g,0\}\qquad \mbox{and} \qquad g^-=\max\{-g,0\}
\]
so that $g=g^+ - g^-$ with $g^\pm\geq 0$. Applying the preceding argument to $g^+$ and $g^-$ gives $u^\pm\in\Omega^p$ with $S^{(p)}_\lambda u^\pm=g^\pm$. By linearity and accretivity on $\Omega^p$,
\[
u:=u^+-u^-\in\Omega^p,
\qquad
S^{(p)}_\lambda u=g, \qquad \|u\|_p = \|u-0\|_p \le \|S_\lambda^{(p)}u - S_\lambda^{(p)}0\|_p = \|g\|_p.
\]
This concludes the proof of the existence of $u$ with the stated properties.

It remains to prove that $\Omega^p$ is dense in $\ell^p(X,\mu)$. For $p=1$, this is immediate by  Lemma~\ref{lem:Cc_density_Omegap} and the fact that \EM[1] always holds.  

Let now $1<p<\infty$. Since $C_c(X)$ is dense in $\ell^p(X,\mu)$, it suffices to prove
\[
C_c(X)\subseteq \overline{\Omega^p}^{\,\ell^p}.
\]
Fix $g\in C_c(X)$ and let $u_\lambda\in\Omega^p$ solve
\[
S^{(p)}_\lambda u_\lambda=g.
\]
Notice that $g\in\ell^p(X,\mu)$ for any $p$ and that the finite-set solutions 
$\varphi_n$ that give $u_\lm$ do 
not depend on $p$. Applying the case $p=1$  to the same datum $g$, the pointwise limit $u_\lambda$ of $\varphi_n$ is also the $\ell^1$-solution. Hence, $u_\lambda\in\Omega^1$ and
\[
 S^{(1)}_\lambda u_\lambda=g.
\]
Since $g\in C_c(X)\subseteq\Omega^1$, the accretivity of $\Delta^{(1)}_{|\Omega^1}$ gives
\begin{equation}\label{eq:p-density-1}
\|u_\lambda-g\|_1
\leq
\| S^{(1)}_\lambda u_\lambda- S^{(1)}_\lambda g\|_1
=
\lambda\|\Delta^{(1)} g\|_1. 
\end{equation}
Now choose an  exponent $r>p$. Again, the pointwise limit $u_\lambda$ of $\varphi_n$ is also the $\ell^r$ solution for the datum $g$ and 
\[
\|\varphi_n\|_r \leq \|g\|_r \qquad \text{and} \qquad \|u_\lambda\|_r \leq \|g\|_r.
\]
Hence,
\begin{equation}\label{eq:p-density-2}
\|u_\lambda-g\|_r \leq 2\|g\|_r. 
\end{equation}
Now interpolate between $\ell^1$ and $\ell^r$. Let $\theta\in(0,1)$ be defined by
\[
\frac{1}{p}=\theta+\frac{1-\theta}{r}, \qquad\text{i.e.,}\qquad \theta=\frac{r-p}{p(r-1)}.
\]
Then, {by Littlewood's inequality, which is an interpolation of H{\"o}lder},
\[
\|u_\lambda-g\|_p \leq \|u_\lambda-g\|_1^{\theta}\,\|u_\lambda-g\|_r^{1-\theta}.
\]
Using \eqref{eq:p-density-1} and \eqref{eq:p-density-2}, we get
\[
\|u_\lambda-g\|_p \leq \bigl(\lambda \|\Delta^{(1)} g\|_1\bigr)^{\theta} \bigl(2\|g\|_r\bigr)^{1-\theta}.
\]
Therefore,
\[
\|u_\lambda-g\|_p\to0
\]
as $\lambda \to 0$.
So, every $g\in C_c(X)$ lies in the $\ell^p$-closure of $\Omega^p$. Since $C_c(X)$ is 
dense in $\ell^p(X,\mu)$ for $p \in [1,\infty)$, this proves that $\Omega^p$ is dense in $\ell^p(X,\mu)$.

Finally, $\Delta^{(p)}_{|\Omega^p}$ is accretive by Proposition~\ref{p:accretivityDeltap}, and we have just shown that $S^{(p)}_\lambda \colon \Omega^p \to \ell^p(X,\mu)$ is surjective for every $\lambda>0$. Hence, $\Delta^{(p)}_{|\Omega^p}$ is m-accretive. Uniqueness follows from accretivity.
\end{proof}

\subsection{\texorpdfstring{On the m-accretivity of $\Delta^{(p)}$ for $p\in[1,\infty]$}{On the m-accretivity of Delta p for p in [1,infty]}}
We will show that for $\Dep$ accretivity, m-accretivity
and the injectivity of $\Sep=\id+\lm \Dep$ 
for some $\lm>0$ are all equivalent. This holds even for $p=\infty$.  
As consequences, we derive the m-accretivity of $\Dep$ under 
the infinite measure of paths condition
and the m-accretivity of $\Delta^{(\infty)}$ if the reciprocal of the vertex degree is not summable
over paths.

In order to cover the $p=\infty$ case, we first identify the adjoint operator of $\Delta^{(\infty)}$.
We first note that $\ell^\infty(X) \subseteq \dom(\Delta)$ and thus the domain of $\Delta^{(\infty)}$
reads as
$$\dom(\Delta^{(\infty)})= \{ u\in \ell^\infty(X) \mid \Delta u \in \ell^\infty(X)\}.$$
We recall that \EM[1] always holds, i.e., $\Delta(C_c(X))\subseteq \ell^1(X,\mu)$. Furthermore, 
$\Delta_{\min}^{(1)}$, the closure of $\Delta_c=\Delta_{|C_c(X)}$ on $\ell^1(X,\mu)$, 
is a closed, densely defined and accretive operator by Lemma~\ref{l:min_accretivity}. 

We identify $(\ell^1(X,\mu))^{\ast}$ with $\ell^\infty(X)$ via the pairing
\[
(f,g)=\sum_{x\in X}f(x)g(x)\mu(x)
\]
for $f\in\ell^1(X,\mu)$ and $g\in\ell^\infty(X)$. With these preliminary statements, we now
show that $\Delta^{(\infty)}$ and $\Delta_{\min}^{(1)}$ are dual to each other.

\begin{lemma}\label{l:infty_adjoint} 
Let $G$ be a graph.
Then, $\bigl(\Delta_{\min}^{(1)}\bigr)^{\ast}=\Delta^{(\infty)}$.
\end{lemma}
\begin{proof}
Since $C_c(X)$ is dense in $\ell^1(X,\mu)$ and $\Delta_{\min}^{(1)}$ is the closure of $\Delta_c$ on $\ell^1(X,\mu)$, we have $\bigl(\Delta_{\min}^{(1)}\bigr)^{\ast}=(\Delta_c)^{\ast}$. Green's formula \cite[Proposition~1.5]{keller2021graphs} gives for $\varphi\in C_c(X)$ and $u\in\ell^\infty(X)$
\[
\sum_{x\in X}\Delta\varphi(x)u(x)\mu(x)
=
\sum_{x\in X}\varphi(x)\Delta u(x)\mu(x).
\]
The right-hand side defines a bounded linear functional on $\ell^1(X,\mu)$ in the variable $\varphi$ if and only if $\Delta u\in\ell^\infty(X)$, i.e., $u \in \dom(\Delta^{(\infty)})$. This completes the proof.
\end{proof}

Having identified the needed duality for $\ell^\infty(X)$, we can now
give a unifying statement concerning the m-accretivity of the maximal
Laplacians on $\ell^p$.

\begin{theorem}\label{t:acc_inj}
    Let $p \in [1,\infty]$. The following statements are equivalent: 
    \begin{itemize}
        \item[\textup{(i)}] $\Dep$ is m-accretive.
        \item[\textup{(ii)}] $\Dep$ is accretive.
        \item[\textup{(iii)}] $\Sep$ is injective on $\dom(\Delta^{(p)})$
        for some (equivalently, all) $\lambda >0$.
    \end{itemize}
\end{theorem}
\begin{proof}
The implications (i) $\Longrightarrow$ (ii) $\Longrightarrow$ (iii) are immediate. So, it remains to prove (iii) $\Longrightarrow$ (i). We break this down into two cases.

\medskip
\noindent \textbf{Case $p\in [1,\infty)$}:
Let $\Omega^p$ be as in Theorem~\ref{thm:new2} so that $\Delta^{(p)}_{|\Omega^p}$ is m-accretive.
Let $\lambda_0 >0$ be such that (iii) holds. Then, $S^{(p)}_{\lm_0}\colon \Omega^p \to \ell^p(X,\mu)$
is onto by m-accretivity while $S^{(p)}_{\lm_0}\colon \dom(\Delta^{(p)}) \to \ell^p(X,\mu)$ is injective
by assumption.
Since $\Omega^p \subseteq \dom(\Dep)$, the conclusion follows by Lemma~\ref{l:inclusion}.

\medskip
\noindent \textbf{Case $p=\infty$}: 
By Lemma~\ref{l:infty_adjoint}, we have
$\bigl(\Delta_{\min}^{(1)}\bigr)^{\ast}=\Delta^{(\infty)}$.

We next claim that $\operatorname{Rg}\bigl(S^{(1)}_{\min, \lambda}\bigr)$, the
range of $S^{(1)}_{\min, \lambda}$,
is closed for every $\lambda>0$ where $S^{(1)}_{\min, \lambda}:=\id+\lambda\Delta_{\min}^{(1)}$.
Since, by Lemma~\ref{l:min_accretivity}, $\Delta_{\min}^{(1)}$ is accretive on $\ell^1(X,\mu)$,
\[
\|u\|_1\le \|S^{(1)}_{\min, \lambda} u\|_1
\]
for every $u\in\dom(\Delta_{\min}^{(1)})$.
Therefore, $S^{(1)}_{\min, \lambda}$ is bounded below. 
Let $g_n \in \operatorname{Rg}(S^{(1)}_{\min, \lambda})$ with $g_n=S^{(1)}_{\min, \lambda} u_n\to g$ in $\ell^1(X,\mu)$. The lower bound gives
\[
\|u_n-u_m\|_1\le \|S^{(1)}_{\min, \lambda}(u_n-u_m)\|_1=\|g_n-g_m\|_1.
\]
Hence, $(u_n)$ is Cauchy and $u_n\to u$ for some $u\in\ell^1(X,\mu)$. Then,
\[
\lambda\,\Delta_{\min}^{(1)} u_n=g_n-u_n\to g-u.
\]
Since $\Delta_{\min}^{(1)}$ is closed, we conclude $u\in\dom(\Delta_{\min}^{(1)})$ and $\lambda\,\Delta_{\min}^{(1)} u=g-u$, i.e., $g=S^{(1)}_{\min, \lambda} u\in\operatorname{Rg}(S^{(1)}_{\min, \lambda})$. This shows that $\operatorname{Rg}(S^{(1)}_{\min, \lambda})$ 
is closed as claimed.

Now assume (iii) holds, that is, $S^{(\infty)}_\lambda= \id +\lambda\Delta^{(\infty)}$ is injective for some $\lambda>0$. By the closed range theorem combined with $\bigl(\Delta_{\min}^{(1)}\bigr)^{\ast}=\Delta^{(\infty)}$ from Lemma~\ref{l:infty_adjoint}, we get
\[
\operatorname{Rg}\bigl(S^{(1)}_{\min, \lambda}\bigr)
=\bigl(\ker(\id+\lambda\Delta^{(\infty)})\bigr)^{\perp}
=\bigl(\ker S^{(\infty)}_\lambda\bigr)^{\perp}
=\ell^1(X,\mu)
\]
where the third equality uses injectivity of $S^{(\infty)}_\lambda$. Hence, $S^{(1)}_{\min, \lambda}$ is surjective for some $\lambda>0$, so $\Delta_{\min}^{(1)}$ is m-accretive on $\ell^1(X,\mu)$, see Remark~\ref{rem:lambda}.

Finally, we transfer the m-accretivity of $\Delta_{\min}^{(1)}$ to its Banach space adjoint. 
For every $\lambda>0$, the m-accretivity of $\Delta_{\min}^{(1)}$ means that $S^{(1)}_{\min, \lambda}$ is bijective from $\dom(\Delta_{\min}^{(1)})$ onto $\ell^1(X,\mu)$ with $\bigl(S^{(1)}_{\min, \lambda}\bigr)^{-1}$ being a bounded operator on $\ell^1(X,\mu)$ 
with $\bigl\|(S^{(1)}_{\min, \lambda})^{-1}\bigr\|\le 1$. 
Since $\Delta_{\min}^{(1)}$ is densely defined and closed, so is $S^{(1)}_{\min, \lambda}$ 
and \cite[Proposition~B.11]{EngelNagel2000} yields 
that $\bigl(S^{(1)}_{\min, \lambda}\bigr)^{\ast}$ is bijective with
\[
\Bigl(\bigl(S^{(1)}_{\min, \lambda}\bigr)^{\ast}\Bigr)^{-1}
=\Bigl(\bigl(S^{(1)}_{\min, \lambda}\bigr)^{-1}\Bigr)^{\ast}\]
and is bounded on $\ell^\infty(X)$ with 
\[ \Bigl\|\Bigl(\bigl(S^{(1)}_{\min, \lambda}\bigr)^{\ast}\Bigr)^{-1}\Bigr\|
=\bigl\|(S^{(1)}_{\min, \lambda})^{-1}\bigr\| \le 1.
\]
Combined with $\bigl(\Delta_{\min}^{(1)}\bigr)^{\ast}=\Delta^{(\infty)}$, this gives
\[
\bigl(S^{(1)}_{\min, \lambda}\bigr)^{\ast}
=\id+\lambda\bigl(\Delta_{\min}^{(1)}\bigr)^{\ast}
=\id+\lambda\Delta^{(\infty)}
=S^{(\infty)}_\lambda
\]
so $S^{(\infty)}_\lambda$ is bijective with contractive inverse for every $\lambda>0$. Therefore, 
$\Delta^{(\infty)}$ is m-accretive, proving (i) for~$p=\infty$.
\end{proof}

\begin{remark}[Connection to stochastic completeness]
In the above proof we have shown that if $S^{(\infty)}_\lambda$ is injective for some $\lambda$, i.e., if $\Delta$ is stochastically complete at infinity, see Definition~\ref{def:stochastic_completeness} below, then $\Delta^{(1)}_{\min}$ is m-accretive. This is the content, in our setting and for $\kappa=0$, of  \cite[Theorem 2.11 (ii)] {milatovic2015maximal}.
A minor variation of the above argument also shows that if $\EM[p]$ holds  for some  $1<p<\infty$, then $(\Delta_{\min}^{(p)})^\ast= \Delta^{(q)}$ for $q$ such that $p^{-1}+q^{-1}=1$.  Moreover, if we  also assume that $S^{(q)}_\lambda$ is injective for some $\lambda >0$, then $\Delta^{(p)}_{\min}$ is m-accretive. Noting that, by  Theorem~\ref{thm:min}, $\IP$ implies the  injectivity of $S^{(q)}_\lambda$ for every  $\lambda >0$ and  $1\leq q<\infty$, we recover the conclusion, in our setting,  of  \cite[Theorem 2.11 (i)] {milatovic2015maximal}.
\end{remark}

We  have the following straightforward consequence.
\begin{corollary}
   Let $p\in[1,\infty)$.  Then, $S^{(p)}_\lambda$ is injective for some $\lambda>0$
   if and only if $\Delta^{(p)} = \Delta^{(p)}_{|\Omega^p}$ where $\Omega^p$ is such that $\Delta^{(p)}_{|\Omega^p}$ is m-accretive.
\end{corollary}
\begin{proof}
    Since $\Delta^{(p)}_{\vert \Omega^p} \subseteq \Dep$, the statement follows by combining Theorems~\ref{thm:new2}~and~\ref{t:acc_inj} with Lemma~\ref{l:inclusion}.
\end{proof}

Given Theorem~\ref{t:acc_inj}, we can derive some immediate conditions for the
m-accretivity of $\Dep$ by looking at the injectivity
of the shifted operator. In fact, we have already investigated this injectivity as a consequence
of our comparison principle, see Theorem~\ref{thm:min}.
The first result gives that the infinite measure of infinite paths condition implies that the operator $\Delta^{(p)}$
is m-accretive for all $p \in [1,\infty)$. This extends Theorem~\ref{t:lf} for the porous medium-type operator
on $\ell^1(X,\mu)$. Furthermore, m-accretivity always implies that the operator is closed 
which has been used in our previous considerations and will play a role later on as well.
Thus, we note it below.
\begin{proposition}\label{p:Lap_accretive}
If $G$ satisfies $\IP$, then
$\Delta^{(p)}$ is 
m-accretive and closed for $p \in [1,\infty)$.
\end{proposition}
\begin{proof}
The m-accretivity follows immediately from Theorem~\ref{t:acc_inj} and Theorem~\ref{thm:min}~(i).
Furthermore, any m-accretive operator
is automatically closed, see \cite[III.~Proposition~3.4]{barbu2010nonlinear}. 
\end{proof}

\begin{remark} 
In the more general setting of Schrödinger operators on vector bundles over weighted graphs, 
by combining Theorems~2.11~and~2.15 in \cite{milatovic2015maximal}, the
authors establish the m-accretivity of $\Delta^{(1)}$
assuming stochastic completeness (see Definition~\ref{def:stochastic_completeness} below), 
\EM, and infinite measure of infinite
paths. We note that Proposition~\ref{p:Lap_accretive} 
gives m-accretivity for all $p\in[1,\infty)$ while dropping the assumptions of stochastic completeness
and \EM.
\end{remark}

The next statement gives a criterion for the $p=\infty$ case.

\begin{proposition}\label{p:Lap_accretive2}
If $G$ satisfies $\sum_{n} \frac{1}{\Deg(x_{n})}=\infty$ for every infinite path $(x_n)$, then
$\Delta^{(\infty)}$ is m-accretive.
\end{proposition}
\begin{proof}
    This follows immediately from Theorem~\ref{t:acc_inj} and Theorem~\ref{thm:min}~(ii).
\end{proof}

A stronger assumption on the vertex degree gives m-accretivity for
all maximal Laplacians as we will now discuss.
In order to place this result in a proper context, we now recall a notion
for a metric to be intrinsic for $\ell^2$.

We let $\varrho \colon X\times X \to [0,\infty)$ 
be a pseudo-metric and 
say that $\varrho$ is \emph{$\ell^2$-intrinsic} if
$$\sum_{y \in X}w(x,y)\varrho^2(x,y) \leq C \mu(x)$$ 
for some $C>0$ and all $x \in X$. 

We recall the finite balls 
condition $\FB$ which states that all distance balls defined with respect to a pseudo-metric
are finite which has played a role in our previous considerations.
There is a weaker condition than $\FB$ which states that $\Deg$ is 
bounded on distance balls defined with respect to an $\ell^2$-intrinsic metric. 
In this case, Liouville-type theorems which give the constancy of positive
subharmonic functions, imply that the shifted operator is injective on $\ell^p$.
\begin{lemma}\label{l:Liouv}
    If there exists an $\ell^2$-intrinsic pseudo-metric $\varrho$ such that $\Deg$ is bounded
    on distance balls defined with respect to $\varrho$, then $S_\lm^{(p)}$ is injective
    for $p \in (1,\infty)$.
\end{lemma}
\begin{proof}
    This follows by a Liouville-type theorem and general principles, see
    \cite[Theorem~12.15 and Lemma~12.19]{keller2021graphs} for a proof.
\end{proof}

As a consequence, we now give two results. One says that bounded degree implies m-accretivity
across all $p$.
\begin{corollary}\label{c:Lap_accretive3}
    If $G$ satisfies $\BD$, then $\Dep$ is m-accretive for all $p \in [1,\infty]$.
\end{corollary}
\begin{proof}
    For $p=1$, this follows from Theorem~\ref{t:accretivity_bd} which gives
    accretivity and only requires $\B$ along with Theorem~\ref{t:acc_inj}.
    For $p=\infty$, this follows from Proposition~\ref{p:Lap_accretive2} directly above.
    For $p \in (1,\infty)$, this follows from Lemma~\ref{l:Liouv} and Theorem~\ref{t:acc_inj}
    as $\Deg$ is bounded on the entire vertex set. 
\end{proof}

As another application of the Liouville theorems,
we also get a statement for metric completeness.
\begin{corollary}\label{c:Lap_accretive_intrinsic}
If there exists an $\ell^2$-intrinsic metric satisfying $\FB$, then
$\Dep$ is m-accretive for all $p \in (1,\infty)$.
\end{corollary}
\begin{proof}
    This follows from Lemma~\ref{l:Liouv} and Theorem~\ref{t:acc_inj}
    as all distance balls are finite and thus $\Deg$ is bounded on distance balls. 
\end{proof}

We now summarize the preceding results for the maximal Laplacian on $\ell^p$.
\begin{theorem}\label{t:Lap_accretive}
Let $G$ be a graph.
\begin{enumerate}
\item[\textup{(1)}] If $G$ satisfies $\IP$, then $\Dep$ is m-accretive for all $p \in [1,\infty)$.
\item[\textup{(2)}] If $G$ satisfies $\BD$, then $\Dep$ is m-accretive for all $p \in [1,\infty]$. 
\item[\textup{(3)}] If there exists an $\ell^2$-intrinsic metric satisfying $\FB$, then
$\Dep$ is m-accretive for all $p \in (1,\infty)$.
\end{enumerate}
\end{theorem}
\begin{proof}
Statement (1) is Proposition~\ref{p:Lap_accretive}, (2) is Corollary~\ref{c:Lap_accretive3}, and
(3) is Corollary~\ref{c:Lap_accretive_intrinsic}. 
\end{proof}

\subsection{\texorpdfstring{On the m-accretivity of $\deltam^{(p)}$ for $p\in[1,\infty]$}{On the m-accretivity of Delta min p for p in [1,infty]}}
We now focus on the minimal Laplacians and give some conditions for them to be
m-accretive. The m-accretivity follows as soon as they are equal to the maximal Laplacians so
we can use the results in the previous subsection. We also show that accretivity
and m-accretivity are not equivalent for the minimal Laplacians.

We recall that if condition \EM[p] holds, then $\Delta_c=\Delta_{|C_c(X)}$ maps into $\ell^p(X, \mu)$
and $\Delta_{\min}^{(p)}$, the closure of $\Delta_c$ in $\ell^p(X,\mu)$,
is always accretive. Furthermore, if $\Delta^{(p)}$ is closed, then $\deltam^{(p)} \subseteq \Delta^{(p)}$, see Lemma~\ref{l:min_accretivity}.
We note that \EM[p]
is satisfied, for example, whenever the graph is locally finite or when the edge degree
is bounded or whenever we have a uniform lower bound
on the measure, see Corollary~\ref{c:emp}.

In order to leverage our previous results for $\Dep$, we now show that
whenever $\Delta_{\min}^{(p)}=\Dep$ these operators are automatically m-accretive.
This gives a counterpart to Lemma~\ref{l:min_equal_accretive} concerning $\deltaphim$.
\begin{lemma}\label{l:min_accretivity2}
Let $G$ satisfy \EM[p] for some $p \in [1,\infty]$. 
{ Then, $\Delta^{(p)}=\deltam^{(p)}$ if and only
if $\Delta^{(p)}$ and $\deltam^{(p)}$ are both m-accretive.}
\end{lemma}
\begin{proof}  
If \EM[p] holds for some $p \in [1,\infty]$, then $\deltam^{(p)}$ is accretive by Lemma~\ref{l:min_accretivity}. Hence, if $\Delta^{(p)}=\deltam^{(p)}$, then
$\Delta^{(p)}$ is accretive and, therefore, m-accretive by Theorem~\ref{t:acc_inj}; the equality then shows that $\deltam^{(p)}$ is also m-accretive. Conversely, if
$\Delta^{(p)}$ and $\deltam^{(p)}$ are both m-accretive, then $\Delta^{(p)}$
is closed so that $\deltam^{(p)} \subseteq \Dep$ by Lemma~\ref{l:min_accretivity}
and equality now holds by Lemma~\ref{l:inclusion}.
\end{proof}

In the bounded degree case, 
the Laplacian gives a bounded operator on all $\ell^p$ spaces and thus the minimal and
maximal Laplacian agree whenever the finitely supported functions are dense in $\ell^p$.
\begin{proposition}\label{p:bounded_Lap}
    If $G$ satisfies $\BD$, then $\Delta^{(p)}=\Delta^{(p)}_{\min}$ for all $p \in [1,\infty)$. 
    In particular, $\Delta^{(p)}_{\min}$ is m-accretive.
\end{proposition}
\begin{proof}
    Under assumption $\BD$, it follows that $\Delta$ gives a bounded operator
    on all $\ell^p(X,\mu)$ spaces for $p\in[1,\infty]$, see \cite[Theorem~2.15]{keller2021graphs}. 
    In particular, $\dom(\Dep) = \ell^p(X,\mu)$ and \EM[p] holds for all $p \in [1,\infty]$, see Corollary~\ref{c:emp}.
    Therefore, $\deltam^{(p)}$ is defined and the stated equality of operators holds
    whenever $C_c(X)$ is dense in the domain of the maximal operator  which is the case
    for $p \in [1,\infty)$. The m-accretivity then
    follows from Lemma~\ref{l:min_accretivity2}.
\end{proof}

\begin{remark}[The case $p=\infty$] 
    We note that the result above does not extend to the case of $p=\infty$.
    In particular, we note that if $\kappa=0$, then $1 \in \dom(\Delta^{(\infty)})$
    but $1 \not \in \dom(\deltam^{(\infty)})$ whenever $X$ is infinite since 
    a non-zero constant function cannot be approximated by finitely supported functions
    in the $\ell^\infty$ norm.
\end{remark}

We now improve a result from \cite{milatovic2015maximal}. Specifically,
for  $p\in(1,\infty)$, \cite[Theorem~2.11(i) and Remark~2.13]{milatovic2015maximal} gives
that if \EM[p] and $\IP$ hold, then $-\Delta_{\min}^{(p)}$ generates a strongly continuous contraction semigroup on $\ell^p(X,\mu)$ and, in particular, $\Delta_{\min}^{(p)}$ is m-accretive. If, additionally, \EM[q] holds, where $p^{-1} + q^{-1} =1$, then $\Delta^{(p)}_{\min}=\Delta^{(p)}$ and $\Delta^{(p)}$ is m-accretive (\cite[Theorem~2.15(i)]{milatovic2015maximal}). We can improve this result by removing the assumption of \EM[q] as we now show.

\begin{proposition}\label{p:IP_Lap}
    If $G$ satisfies $\IP$ and \EM[p] for $p \in (1,\infty)$, then $\Delta^{(p)}= \Delta^{(p)}_{\min}$. 
    In particular, $\Delta^{(p)}_{\min}$ is m-accretive.
\end{proposition}
\begin{proof} 
    {By Proposition~\ref{p:Lap_accretive}, $\IP$ gives
    that $\Delta^{(p)}$ is m-accretive while
    \EM[p] and $\IP$ imply that $\deltam^{(p)}$ is m-accretive by \cite[Theorem~2.11(i) and Remark~2.13]{milatovic2015maximal}.
    Thus, the equality of operators follows from Lemma~\ref{l:min_accretivity2}.}
\end{proof}

On the other hand, we now give a family of examples where $\Delta_{\min}^{(p)}$ is not m-accretive
for all $p \in [1,\infty]$.
This shows that the phenomenon of the equivalence of accretivity, m-accretivity
and injectivity of $S_\lm$ that we have observed for both $\Delta^{(p)}$ and
$\Deltaphi$ does not hold for the minimal operator since $\deltam^{(p)}$ is always accretive
whenever it is defined.

We first introduce the needed notions. 
Given a graph $G$, we consider the \emph{energy form} $\Q$ acting on $C(X)$ as
$$\Q(u) = \frac{1}{2} \sum_{x, y \in X} w(x,y) (u(x)-u(y))^2+ \sum_{x \in X} \kappa(x)u^2(x)$$
and denote by $\D$ the space of functions of \emph{finite energy} given as
$$\D= \{ u \in C(X) \mid \Q(u) < \infty \}.$$
If $\varphi \in C_c(X)$, Green's formula states
$$Q(\varphi) = \sum_{x \in X} \Delta \varphi(x) \varphi(x) \mu(x)$$
see Proposition~1.5 in \cite{keller2021graphs}.

In order to achieve our results, 
we now introduce the $p$-Sobolev condition,
see \cite{HKSW23, HM15} for $p\in [1,\infty]$ as well as \cite{KLSW17} for the case of $p=\infty$. 
\begin{enumerate}
\item[{\SP[p]}]\hypertarget{ass:SP}{} There exists a constant $C_p> 0$ such that $C_p\|\varphi\|_p^2 \leq \Q(\varphi)$ for 
all $\varphi \in C_c(X)$. \hfill(``$p$-Sobolev'')
\end{enumerate}
When $p=\infty$, we will write \SP[] for \SP[$\infty$] and call such graphs \emph{uniformly transient}.
For $p=2$, this condition is equivalent to the positivity of the bottom of
the spectrum of the Dirichlet Laplacian, i.e., $\lm_0(\Delta_D) >0$, see Section~\ref{sec:Laplacian}
below for the definition of $\Delta_D$.

For our unifying result, 
we will assume that the measure of the entire vertex set is finite, i.e., $\mu(X)<\infty$
and refer to such graphs as \emph{graphs of finite measure}, see \cite{GHKLW15} for some consequences.

We first explore some consequences of the finite measure assumption.
Whenever $\mu(X)<\infty$, H{\"o}lder's inequality implies that if $p \leq r$, then
$\ell^{r}(X,\mu) \subseteq \ell^{p}(X,\mu)$ and there exists a constant $C>0$ depending on $p$
and $r$ such that
\begin{equation}\label{eq:holder}
\|f\|_{p} \leq C \|f\|_{r}
\end{equation}
for all $f \in \ell^{r}(X,\mu)$.
This has the following immediate implications for the $p$-Sobolev inequality.

\begin{lemma}\label{l:pSob}
    Let $G$ satisfy $\mu(X)<\infty$.
    If $G$ satisfies \SP[r] for some $r \in [1,\infty]$,
    then $G$ satisfies \SP[p] for all $p \leq r$. In particular, if $G$ satisfies
    \SP[], then $G$ satisfies \SP[p] for all $p \in [1,\infty]$.
\end{lemma}
\begin{proof}
    This follows immediately from the definition of \SP[p] and the norm comparison \eqref{eq:holder}.
\end{proof}

In terms of the minimal Laplacian, finiteness of measure has the following consequences.
We recall our convention that \EM[] stands for \EM[$\infty$], i.e., that
$\Delta(C_c(X)) \subseteq \ell^\infty(X)$.

\begin{lemma}\label{l:em_finitemeasure}
    Let $G$ satisfy $\mu(X)<\infty$. 
    If $G$ satisfies \EM[r] for some $r \in [1,\infty]$ and $p \leq r$,
    then $G$ satisfies \EM[p] and 
    $$\dom(\deltam^{(r)})\subseteq \dom(\deltam^{(p)}).$$
    In particular, if $G$ satisfies \EM[], then
    $G$ satisfies \EM[p], $\deltam^{(p)}$ is defined, satisfies
    $$\deltam^{(p)}\subseteq \deltam^{(1)}$$
    and is accretive for all $p \in [1,\infty]$. 
\end{lemma}
    \begin{proof}
        If $G$ satisfies $\mu(X)<\infty$ and \EM[r], then $G$ satisfies
        \EM[p] for $p \in [1,r]$ by the norm comparison \eqref{eq:holder}. Therefore,
        $\deltam^{(p)}$ is defined for all $p \in [1,r]$
        by Lemma~\ref{l:min_accretivity}. Furthermore, $\dom(\deltam^{(r)})\subseteq \dom(\deltam^{(p)})$
        by the definition of the domain
        along with the norm comparison \eqref{eq:holder}.
        The existence and inclusion statements for $\deltam^{(p)}$ under \EM[], then follow
        from Lemma~\ref{l:min_accretivity}
        as $\Delta^{(1)}$ is closed by Lemma~\ref{lem:EMp} and thus all 
        minimal operators are restrictions of $\Delta^{(1)}$.
        The statement on accretivity follows
        from Lemma~\ref{l:min_accretivity}.
    \end{proof}

We now put the conditions together to show how the minimal Laplacian may not be m-accretive.
We will need one more condition in order to make this happen. More specifically,
we will need the existence of a non-trivial harmonic function in $\ell^q$, i.e., a function
$h \in \dom(\Delta) \cap \ell^q(X,\mu)$ such that $\Delta h=0$.

For functions $f \in \ell^p(X,\mu)$ and $ g\in \ell^q(X,\mu)$ with $p^{-1}+q^{-1}=1$, we let
$$(f, g) = \sum_{x \in X} f(x)g(x)\mu(x)$$
which exists by H{\"o}lder's inequality.

\begin{theorem}\label{t:non_minimal_m-accretivity}
    Let $p \in [1,\infty]$ and let $q$ be the H{\"o}lder conjugate of $p$, i.e., $p^{-1}+q^{-1}=1$. 
    Assume that $G$ satisfies \SP[p] and \EM[p], and, if $p=\infty$, 
    then additionally $\mu(X)<\infty$. 
    If there exists a non-trivial $h \in \dom(\Delta) \cap \ell^q(X,\mu)$ 
    with $\Delta h=0$, then $\deltam^{(p)}$ is not m-accretive.
\end{theorem}
\begin{proof}
We divide the proof into two cases. For $p \in [1,\infty)$, we do not need
to assume the finiteness of measure. We handle this case first.

\noindent
\textit{Case $p \in [1,\infty)$}: Since \EM[p] holds, then 
by Lemma~\ref{lem:Cc_density_Omegap}, $C_c(X) \subseteq \Omega^p$, where $\Omega^p$ is 
a dense domain for which $\Delta^{(p)}_{|\Omega^p}$ is m-accretive by Theorem~\ref{thm:new2}.  
Since  $\Delta^{(p)}_{|\Omega^p}$ is m-accretive and therefore closed, we get that
\begin{equation*}
\deltam^{(p)} \subseteq \Delta^{(p)}_{|\Omega^p}.
\end{equation*}
Assume, toward a contradiction, that $\deltam^{(p)}$ is m-accretive. 
Since $\Delta^{(p)}_{|\Omega^p}$ is, in particular, accretive, then applying Lemma~\ref{l:inclusion}, 
it follows that $\deltam^{(p)} = \Delta^{(p)}_{|\Omega^p}$. 

Let now $h$ be as in the assumptions and let $x_0 \in X$ be such that $h(x_0) \neq 0$. Define 
$$g = \sgn(h(x_0))1_{x_0} \neq 0 $$
    so that $(g,h) = |h(x_0)|\mu(x_0) > 0$.

By the m-accretivity of $\Delta^{(p)}_{|\Omega^p}$, 
there exists $u_\lm \in \Omega^p \subseteq \dom(\Delta^{(p)})$ 
that satisfies $(\id + \lm \Delta )u_\lm =g$, equivalently, $u_\lm = g - \lm \Delta u_\lm$. We then have:

\begin{itemize}
    \item[(i)] By the equality  $\deltam^{(p)} = \Delta^{(p)}_{|\Omega^p}$ and by the definition 
    of $\deltam^{(p)}$ as the closure, there exists a sequence $(\vartheta_n)$  such that
    \begin{itemize}
       \item[(i.a)] $\vartheta_n \in C_c(X)$,
        \item[(i.b)] $\vartheta_n \to u_\lm$ in $\ell^p(X,\mu)$,
        \item[(i.c)] $\Delta\vartheta_n \to \Delta u_\lm$ in $\ell^p(X,\mu)$.
    \end{itemize}

    \item[(ii)] By the equality $\deltam^{(p)} = \Delta^{(p)}_{|\Omega^p}$ and 
    the proof of Theorem~\ref{thm:new2} (see Property~\eqref{eq:linear_1}) 
    there exists an exhaustion $(X_n)$ and a sequence $(\varphi_n)$  such that
    \begin{itemize}
       \item[(ii.a)] $\varphi_n \in C_c(X)$ with $\operatorname{supp}\varphi_n \subseteq X_n$,
        \item[(ii.b)] $\varphi_n \to u_\lm$ in $\ell^p(X,\mu)$,
        \item[(ii.c)] $S_{\lm,n}\varphi_n = g_n$, that is, $\Delta_n\varphi_n = \lm^{-1}(g_n - \varphi_n)$ where $g_n = \boldsymbol{\mathfrak{i}}_n\boldsymbol{\pi}_ng$ is the Dirichlet cutoff of $g$
        and $\Delta_n$ is defined in \eqref{eq:exhaustion_Lap}. 
    \end{itemize}
\end{itemize}

From (i),  since $h \in \dom(\Delta) \cap \ell^q(X,\mu)$ is harmonic, Green's formula, 
e.g., Proposition~1.5 in \cite{keller2021graphs}) implies
    $$
    (\Delta u_\lm, h) = \lim_{n \to \infty}(\Delta \vartheta_n, h)= \lim_{n \to \infty}(\vartheta_n, \Delta h)=0.
    $$
    This yields
    $$
    (u_\lm, h)=(g-\lm\Delta u_\lm, h) = (g,h)
    $$
    and thus 
    \begin{equation*}
        (g,h)=(u_\lm, h) \leq \|u_\lm\|_p \|h\|_q
    \end{equation*}
    or, equivalently,
    \begin{equation}\label{eq:first_bound}
        \frac{1}{\|u_\lm\|_p} \leq \frac{\|h\|_q}{(g,h)}.
    \end{equation}

Now, observe that $\Delta\varphi_n(x) = \Delta_n\varphi_n (x)$ for every $x\in X_n$ since
$\operatorname{supp} \varphi_n \subseteq X_n$
and, therefore,
\begin{equation*}
    ((\Delta - \Delta_n) \varphi_n, \varphi_n) = \sum_{x\notin X_n} (\Delta - \Delta_n)\varphi_n(x) \varphi_n(x)\mu(x) = 0 \qquad \mbox{by (ii.a)}.
  \end{equation*}
As a consequence, by Green's formula and the above identity,
\begin{equation}\label{eq:energy_identity}
\Q(\varphi_n) = (\Delta \varphi_n, \varphi_n) = (\Delta_n \varphi_n, \varphi_n).
\end{equation}
    From (ii), by using \SP[p] with $C_p>0$ as a $p$-Sobolev constant, and the identity~\eqref{eq:energy_identity}, we obtain
    \begin{align*}
        C_p \|u_\lm\|_p^2 = C_p \lim_{n \to \infty} \|\varphi_n\|_p^2 &\leq \limsup_{n \to \infty} \Q(\varphi_n) \\
        &= \limsup_{n \to \infty}( \Delta_n \varphi_n, \varphi_n)\\
        &= \limsup_{n \to \infty} \frac{1}{\lm} ( g_n -\varphi_n, \varphi_n )\\
        &\leq \limsup_{n \to \infty} \frac{1}{\lm}(g_n, \varphi_n)\\
        &\leq \limsup_{n \to \infty} \frac{1}{\lm} \|\varphi_n\|_p \|g_n\|_q = \frac{1}{\lm} \|u_\lm\|_p \|g\|_q
    \end{align*} 
 where the last limit follows by (ii.b) and the fact that $g\in C_c(X)$ 
 and so $g_n = \boldsymbol{\mathfrak{i}}_n\boldsymbol{\pi}_n g$ satisfies 
 $g_n \to g$ in all $\ell^r$ spaces.

 This yields, $\lm C_p \|u_\lm\|_p \leq \|g\|_q$ since $u_\lm$
 is non-trivial and combining with~\eqref{eq:first_bound} gives
    $$\lm C_p \leq \frac{\|g\|_q \|h\|_q}{(g,h)}.$$
    Since $\lm>0$ is arbitrary,
    this gives a contradiction and thus 
    completes the proof for $\deltam^{(p)}$ for $p \in [1, \infty)$ 
    as, for m-accretivity, $S^{(p)}_\lm$ should be surjective for all $\lm>0$.

\medskip
\noindent
\textit{Case $p = \infty$.} The arguments are mostly the same as in the 
previous case $p\in [1,\infty)$. Assume, toward a contradiction, 
that $\deltam^{(\infty)}$ is m-accretive and 
let $u_\lm \in \dom(\deltam^{(\infty)})$ satisfy $(\id + \lm \Delta)u_\lm =g$ 
with $g = \sgn(h(x_0))1_{x_0}\neq 0$.  
Fix $(\vartheta_n)$ such that $\vartheta_n \to u_\lm$ and 
$\Delta \vartheta_n \to \Delta u_\lm  =  \frac{1}{\lm}(g-u_\lm)$ in $\ell^\infty(X)$. 
Let us first highlight that \eqref{eq:first_bound} still holds.  
However, we cannot use the finite graph approximations from 
Theorem~\ref{thm:new2}, as we did for $p\in [1,\infty)$, since
Theorem~\ref{thm:new2} does not hold  for $p=\infty$. 

Note that thus far we have not used the assumption $\mu(X) < \infty$. 
Suppose now $\mu(X) < \infty$. Convergence in $\ell^\infty(X)$  then implies convergence 
in $\ell^1(X,\mu)$. Therefore,
$$
\lim_{n \to \infty}( \Delta \vartheta_n, \vartheta_n) = ( \Delta u_\lm, u_\lm ).
$$
By using now \SP[] with $C>0$, Green's formula and
    H{\"o}lder's inequality, we obtain
\begin{align*}
        C \|u_\lm\|_\infty^2 = C \lim_{n \to \infty} \|\vartheta_n\|_\infty^2 &\leq \lim_{n \to \infty} \Q(\vartheta_n) \\
        &= \lim_{n \to \infty}( \Delta \vartheta_n, \vartheta_n) = ( \Delta u_\lm, u_\lm )
        = \frac{1}{\lm} ( g-u_\lm, u_\lm ) \leq \frac{1}{\lm}(g, u_\lm) \leq \frac{1}{\lm} \|g\|_1 \|u_\lm\|_\infty.
    \end{align*} 
We now conclude as in the previous case by combining the above inequality with \eqref{eq:first_bound}, and getting a contradiction. 
\end{proof}

As for the additional assumption in the above result concerning the existence of 
a harmonic function in $\ell^q$ we note that this is the case whenever the killing term
satisfies some additional assumptions. For example, when $\kappa=0$ and $\mu(X)<\infty$, 
we can use
the constant function $h=1$ which is in all $\ell^q$ spaces if the graph has finite measure. 
We now highlight this.
\begin{corollary}\label{c:non_minimal_m-accretivity_killing}
    Let $p \in [1,\infty]$.
    If $G$ satisfies $\mu(X)<\infty$, $\kappa =0$, \SP[p], and \EM[p], 
    then $\deltam^{(p)}$ is not m-accretive.
\end{corollary}
\begin{proof}
    From the assumption that $\kappa =0$ it follows that $h=1$ is harmonic. Furthermore,
    by $\mu(X)<\infty$, it follows that $1 \in \ell^q(X,\mu)$ for all $q \in [1,\infty]$.
    Thus, the result follows
    by Theorem~\ref{t:non_minimal_m-accretivity}.
    \end{proof}    

    However, we note that some additional assumptions on $\kappa$ or on the graph 
    are required for the existence
    of a non-trivial harmonic function as the following examples
    show.
    \begin{example}[Existence of $h$]
    We give two examples to show that some additional assumptions 
    on $\kappa$ are needed to give the existence of a non-trivial harmonic function
    $h \in \ell^p(X,\mu)$. 
    \begin{itemize}
    \item[(1)]{[Single vertex graph].}
    Let $X=\{x\}$, $\mu(x)=1$, $w=0$ and $\kappa(x)=1$. Then, if $u \in C(X)$, then
$$
\Delta u = u \qquad \textup{ and } \qquad  Q(u)=u^2(x).
$$
In particular, this implies $\lambda_0(\Delta_D)=1>0$ so that \SP[2] holds.
Moreover, \EM[2] holds trivially and $\mu(X)<\infty$.
However,
$$
    \Delta_{\min}^{(2)}=\Delta^{(2)}=\id
$$
which is m-accretive.
In this example there is no non-trivial harmonic function, since
$\Delta h=h$, so $\Delta h=0$ implies $h=0$.

    \item[(2)]{[Birth-death chain].} 
    Let $G=(\N,w, \mu, \kappa)$ be a birth-death chain, i.e., $w(x,y)>0$
        if and only if $|x-y|=1$. If $G$ satisfies $\mu(X)<\infty$ and $\sum_r 1/w(r,r+1)<\infty$,
        then $G$ satisfies \SP[2] by Theorem~9.15 in
        \cite{keller2021graphs}. Furthermore, since $G$ is locally finite, $G$ satisfies 
        \EM[r] for all $r\in[1,\infty]$.
        On the other hand, if $h \in C(X)$ satisfies $\Delta h=0$, then
        $$h(n+1)-h(n)=\frac{1}{w(n,n+1)} \sum_{j=0}^n \kappa(j)h(j)$$
        for all $n \in \N_0$ by a recursion formula, see Lemma~9.16 in \cite{keller2021graphs}.
        Thus, for any $h \neq 0$ we can always choose $\kappa$ large enough 
        so that $h \not \in \ell^p(X,\mu)$
        for any $p \in [1,\infty]$.
    \end{itemize}
    \end{example}

As a consequence of the above, we obtain that the Sobolev and edge-measure conditions
imply the failure of m-accretivity for all minimal Laplacians below where the Sobolev and edge-measure
conditions hold. In particular,
for uniformly transient graphs of finite measure
which satisfy the edge-measure condition, all minimal Laplacians are not m-accretive for every $p \in [1,\infty]$.
\begin{corollary}\label{c:non_minimal_m-accretivity}
    If $G$ satisfies $\mu(X)<\infty$, $\kappa =0$, 
    \SP[r] and \EM[r] for some $r \in [1, \infty]$,
    then $\deltam^{(p)}$ is not m-accretive and thus $\deltam^{(p)} \neq \Dep$ for all $p \in [1,r]$.
    In particular, if $G$ satisfies $\mu(X)<\infty$, $\kappa =0$, 
    \SP[] and \EM[], then $\deltam^{(p)}$ is not m-accretive and thus  $\deltam^{(p)} \neq \Dep$ for all 
    $p \in [1, \infty]$.
\end{corollary}
\begin{proof}
    Since \SP[r] and \EM[r] hold for some $r \geq 1$ and $\mu(X)<\infty$, we get that
    \SP[p] and \EM[p] hold for all $p \leq r$ by Lemmas~\ref{l:pSob}~and~\ref{l:em_finitemeasure}.
    Now, the conclusion follows from Corollary~\ref{c:non_minimal_m-accretivity_killing}.
    That $\deltam^{(p)} \neq \Dep$ follows from Lemma~\ref{l:min_accretivity2}.
\end{proof}

\section{Essential self-adjointness, form uniqueness and accretivity of the Laplacian}\label{sec:Laplacian}
In this section we briefly discuss the $\ell^2$ case and give some connections
between accretivity and the well-studied properties of essential self-adjointness and form uniqueness. 
In particular, we show that accretivity of the maximal Laplacian on $\ell^2$
implies form uniqueness.
Finally, we show that if form uniqueness fails, then accretivity
fails for $\Dep$ for all $p \in [1,\infty]$.

We start with the relevant definitions. Recall that
the energy form $\Q$ is given by
$$\Q(u) = \frac{1}{2} \sum_{x, y \in X} w(x,y) (u(x)-u(y))^2+ \sum_{x \in X} \kappa(x)u^2(x)$$
for $u \in C(X)$ and $\D$ is the space of functions of finite energy:
$$\D= \{ u \in C(X) \mid \Q(u) < \infty \}.$$
We note that $\Q$ is lower semi-continuous, i.e., if $u_n(x) \to u(x)$ for 
every $x \in X$, then
$$\Q(u) \leq \liminf_{n \to \infty}\Q(u_n),$$
see Proposition~1.3 in \cite{keller2021graphs}.

There are two distinguished restrictions of the energy form: $\QN$, the \emph{form
with Neumann boundary conditions}, is defined
as the restriction of $\Q$ to
$$\dom(\QN) = \ell^2(X,\mu) \cap \D$$
while, $\QD$, the \emph{form with Dirichlet boundary conditions}, 
is defined as the restriction of $\Q$ to
$$\dom(\QD) = \overline{C_c(X)}^{\| \cdot \|_\Q}$$
where
$$\| \varphi \|_\Q^2 = \|\varphi\|_2^2 + \Q(\varphi)$$
for $\varphi \in C_c(X)$ and $\| \cdot \|_2$ is the norm on $\ell^2(X,\mu)$.
We say that a graph satisfies \emph{form uniqueness} if $\QD=\QN$.
We note that this is equivalent to \emph{Markov uniqueness}, i.e., that $\Delta$
has a unique Markov realization. More specifically, there exists a
unique operator acting
as $\Delta$ and whose form is a Dirichlet form with domain
containing $C_c(X)$, see \cite{Sch20, Sch20b, keller2021graphs}
for the equivalence of these two notions.

We denote the operators associated to $\QD$ and $\QN$
by $\Delta_{D}$ and $\Delta_{N}$, respectively, and note that
both act as restrictions of $\Delta$ on their domains, see Theorem~1.12 in \cite{keller2021graphs}.
As both operators
are positive and self-adjoint, they are m-accretive by general principles.

When \EM[2] holds, we 
say
that $\Delta_c$ is \emph{essentially self-adjoint} whenever $\Delta_{\min}^{(2)}$ is self-adjoint.
This is equivalent to $\Delta_c$ having a unique self-adjoint extension.
Essential self-adjointness implies form uniqueness, however, the reverse implication is not true.

We note that from Theorem~\ref{t:acc_inj}, $\Delta^{(2)}$ is accretive if and only if $S_\lm=\id +\lm \Delta$ is injective
on $\dom(\Delta^{(2)})$ if and only if $S_\lm$ is injective on $\dom(\Delta) \cap \ell^2(X,\mu)$. Furthermore,
by general theory, form uniqueness is equivalent to $S_\lm$ being injective on $\D \cap \ell^2(X,\mu)$ while essential
self-adjointness is
equivalent to $S_\lm$ being injective on $\dom(\Delta^{(2)})$.
This easily gives connections between m-accretivity and these properties as we now discuss.
We start with form uniqueness.

\begin{proposition}\label{p:MU}
If $\Delta^{(2)}$ is m-accretive, then $\QD=\QN$. If $\dom(\Delta^{(2)}) \subseteq \D$
and $\QD=\QN$, then $\Delta^{(2)}$ is m-accretive.
In this case, $\Delta^{(2)}=\Delta_D$.
\end{proposition}
\begin{proof}
    From general theory, $\QD=\QN$ if and only if $S_\lm=\id+\lm \Delta$ is injective
    on $\D \cap \ell^2(X,\mu)$, see Theorem~3.2
    in \cite{keller2021graphs}. Furthermore, $\D \subseteq \dom(\Delta)$, see Proposition~1.4 in \cite{keller2021graphs}.
    Combining this with Theorem~\ref{t:acc_inj}, which
    states that $\Delta^{(2)}$ is m-accretive if and only if $S_\lm$ is
    injective on $\dom(\Delta^{(2)})$, equivalently, on $\dom(\Delta) \cap \ell^2(X,\mu)$,
    gives the first two statements. If $\Delta^{(2)}$ is m-accretive, then
    $\Delta^{(2)}$ is self-adjoint by general theory \cite{kato2013perturbation}, 
thus, it follows that $\Delta^{(2)}=\Delta_D$ as both are self-adjoint
    restrictions of $\Delta$ and $\Delta_D \subseteq \Delta^{(2)}$. 
\end{proof}

We note that $\QD=\QN$ does not imply either $\dom(\Delta^{(2)}) \subseteq \D$ or that $\Delta^{(2)}$
is m-accretive as, for example,
taking a graph which satisfies form uniqueness but not essential self-adjointness shows.
Such a graph will have a non-trivial function $u \in \dom(\Delta^{(2)})$ such that $S_\lm u=0$, but $u \not \in \D$. 
See \cite{HKMW13, HKMRW} for some concrete examples.

Regarding the condition $\dom(\Delta^{(2)}) \subseteq \D$, we also have the following result which states
that this inclusion is implied by the finite balls assumption for an $\ell^2$-intrinsic metric. As this may be of independent interest, we include
a proof. See also \cite{HM15}. 

\begin{proposition}\label{p:energy}
If $G$ is a graph with an $\ell^2$-intrinsic metric satisfying $\FB$, then
$\dom(\Delta^{(2)}) \subseteq \D$.
\end{proposition}
\begin{proof}
Let $\varrho$ denote the $\ell^2$-intrinsic metric with $\FB$.
Let $u \in \dom(\Delta^{(2)})$ and let
$$\eta_R(x) = \left(1 - \frac{\varrho(x_0, x)}{R} \right)_+$$
for $x \in X$ and $R >0$.
By a direct calculation using the $\ell^2$-intrinsic condition and the triangle inequality as in
the proof of Lemma~\ref{l:cutoff}, we obtain
$$\frac{1}{\mu(x)} \sum_{y \in X} w(x,y) (\eta_R(x)-\eta_R(y))^2 \leq \frac{C}{R^2}$$
for every $x\in X$ and $R > 0$.
By the finite balls condition, $u\eta_R \in C_c(X)$ and Lebesgue's dominated convergence theorem implies
$u\eta_R \to u$ in $\ell^2(X,\mu)$. Furthermore, a Caccioppoli-type
inequality, see Lemma~12.2 in \cite{keller2021graphs}, implies
$$\Q(u\eta_R)\leq \langle \Delta u, u \eta_R^2\rangle + \frac{1}{2} \sum_{x,y \in X}w(x,y) u^2(x) (\eta_R(x)-\eta_R(y))^2
\leq\langle \Delta u, u \eta_R^2\rangle+ \frac{C}{2R^2} \|u\|^2. $$
Since the right-hand side is uniformly bounded in $R$, the lower semi-continuity of $\Q$ then gives $u \in \D$. This completes the proof.
\end{proof}
\begin{remark}[Domain inclusions]
We note that the preceding lemma actually shows that $\dom(\Delta^{(2)}) \subseteq \dom(Q_N) = \ell^2(X,\mu) \cap \D$.
Furthermore, as $\FB$ for an $\ell^2$-intrinsic metric implies form uniqueness, see \cite[Theorem~12.7]{keller2021graphs}
or combine Theorem~\ref{t:Lap_accretive} and Proposition~\ref{p:MU} above,
we get $\dom(\Delta^{(2)}) \subseteq \dom(Q_D)$.
\end{remark}

As for essential self-adjointness, we have the following result.
\begin{proposition}\label{p:ESA}
    Let $G$ satisfy \EM[2]. The following statements are equivalent:
    \begin{itemize}
        \item[\textup{(i)}] $\Delta_c$ is essentially self-adjoint.
        \item[\textup{(ii)}] $\Delta^{(2)}$ is m-accretive.
        \item[\textup{(iii)}] $\deltam^{(2)}$ is m-accretive.
        \item[\textup{(iv)}] $\deltam^{(2)}=\Delta^{(2)}$.
    \end{itemize}
\end{proposition}
\begin{proof}
    When $\Delta(C_c(X)) \subseteq \ell^2(X,\mu)$, i.e., when \EM[2] holds,
    $\Delta_c$ is essentially self-adjoint
    if and only if $S_\lm^{(2)}$ is injective by general theory, e.g., \cite[Theorem~3.6]{keller2021graphs}.
    The equivalence of (i) and (ii) now follows by Theorem~\ref{t:acc_inj}. 
    As essential self-adjointness is equivalent to the self-adjointness of $\deltam^{(2)}$
    by definition, the
    equivalence of (i) and (iii) follows as m-accretivity of $\deltam^{(2)}$
    is equivalent to self-adjointness of $\deltam^{(2)}$, see
    \cite[Problem V.3.32]{kato2013perturbation}.
    { As have already established the equivalence of (i), (ii), and (iii),
    the equivalence to (iv) now follows from Lemma~\ref{l:min_accretivity2}.}
\end{proof}

As a consequence, we note that finiteness of balls for an $\ell^2$-intrinsic metric
implies that the minimal and maximal Laplacians agree on $\ell^2$.
\begin{corollary}\label{c:balls_Lap}
  If $G$ satisfies \EM[2] and has an $\ell^2$-intrinsic metric that satisfies $\FB$, then
  $\Delta^{(2)}=\Delta^{(2)}_{\min}$. In particular, $\Delta^{(2)}_{\min}$ is m-accretive.
\end{corollary}
\begin{proof}
    This follows immediately from the fact that $\FB$ for an $\ell^2$-intrinsic
    metric implies essential self-adjointness, see \cite[Theorem~12.21]{keller2021graphs}, and Proposition~\ref{p:ESA}.
\end{proof}

We have seen that the accretivity of $\Delta^{(2)}$ implies form uniqueness.
On the other hand, 
we now show that the failure of form uniqueness implies that the maximal operator is not accretive
for all $p \in [1,\infty]$.
\begin{theorem}\label{t:counter}
If $\QD \neq \QN$,  
then $\Delta^{(p)}$ is not accretive for all $p \in [1,\infty]$.
In particular, $\Delta^{(p)} \neq \Delta^{(p)}_{|\Omega^{p}}$ for all $p \in [1,\infty)$. 
If additionally $G$ satisfies $\EMp$, then
$\Delta^{(p)} \neq \Delta^{(p)}_{\min}$ for all $p \in [1,\infty]$.
\end{theorem}
\begin{proof}
It follows by general theory that if $\QD \neq \QN$,
then
$S_\lambda= \id + \lambda \Delta$ is not injective on $\dom(\Delta^{(p)}) \cap \D$ for 
all $p \in [1,\infty]$ and $\lambda>0$
see, e.g., \cite[Theorem~3.2]{keller2021graphs}. Hence, $\Delta^{(p)}$ is not
accretive by definition for all $p \in [1,\infty]$. 
The ``in particular'' statement follows as $\Delta^{(p)}_{| \Omega^{p}}$
is always accretive, see Proposition~\ref{p:accretivityDeltap}.
Finally, 
by Lemma~\ref{l:min_accretivity}, if $\EMp$ holds, then $\Delta_{\min}^{(p)}$ is accretive for $p \in [1,\infty]$
and $\Delta^{(p)} \neq \Delta_{\min}^{(p)}$ follows.
\end{proof}

\begin{remark}\label{rem:LneqL|omega}
For an example where the above result gives non-accretive maximal operators,
we note that whenever $\mu(X) < \infty$, $\QD \neq \QN$ is
equivalent to transience as well as to
stochastic incompleteness. Thus, to construct examples where accretivity fails,
it suffices to take any transient graph, for example a regular
3-tree with standard edge weights, and put a finite measure on the graph. See also \cite{HKMRW} and \cite{HKMW13} 
for examples with $\mu(X)=\infty$ and where form uniqueness fails.
With respect to Theorem~\ref{t:non_minimal_m-accretivity}, we note that
whenever a graph is weakly spherically symmetric, satisfies $\mu(X)<\infty$,
and is not form unique, then $\lm_0(\Delta_D) > 0$, see \cite{KLW13} or \cite[Theorem~9.15]{keller2021graphs} and thus \SP[2] holds.
Thus, the example of a regular 3-tree with a finite measure 
and $\kappa=0$ also works for this purpose.
\end{remark}

\section{\texorpdfstring{Stochastic completeness at infinity and m-accretivity}{Stochastic completeness at infinity and m-accretivity}}\label{s:sc_and_acc}
In this final section we consider stochastic completeness at infinity.
This property concerns uniqueness of bounded solutions of the heat equation. Here, 
we discuss connections to m-accretivity of the Laplacians.

We note that $\ell^\infty(X) \subseteq \dom(\Delta)$ and 
now introduce the property of interest which is equivalent
to injectivity of the shifted formal Laplacian.
\begin{definition}[Stochastic completeness]\label{def:stochastic_completeness}
    A graph $G$ is called \emph{stochastically complete at infinity} if $S_\lm =\id+\lm \Delta$
    is injective on $\ell^\infty(X)$ for some (equivalently, all) $\lm>0$.
    If $\kappa=0$ and $S_\lm$ is injective on $\ell^\infty(X)$ for $\lm>0$, then the graph
    is called \emph{stochastically complete}.
\end{definition}

For the equivalence of this formulation to stochastic properties, see \cite[Theorem~1]{KL12} or
\cite[Theorems~7.16~and~7.18]{keller2021graphs}.

We first note that the definition of stochastic completeness at infinity
can also be rephrased in terms of $\Delta^{(\infty)}$.
\begin{lemma}\label{l:SC}
    A graph $G$ is stochastically complete at infinity if and only if
    $S_\lm^{(\infty)}=\id+\lm\Delta^{(\infty)}$ is injective on $\dom(\Delta^{(\infty)})$ for some (equivalently, all) $\lambda >0$.
\end{lemma}
\begin{proof}
    Since $\ell^\infty(X) \subseteq \dom(\Delta)$, the domain
    of the maximal Laplacian reads as
$$\dom(\Delta^{(\infty)})= \{u \in \ell^\infty(X) \mid \Delta u \in \ell^\infty(X)\}.$$
    In particular, if $u \in \ell^\infty(X)$ satisfies $S_\lm u =(\id+\lm\Delta)  u=0$,
    then $u \in \dom(\Delta^{(\infty)})$. Therefore, $\ker(S_\lambda)= \ker(S^{(\infty)}_\lambda)$. This gives the result.
\end{proof}

We recall that \EM[1] always holds and thus $\deltam^{(1)}$
always exists by Lemma~\ref{l:min_accretivity}.
Specializing the results  in \cite{milatovic2015maximal} to the case of the graph Laplacian establishes that, if the graph
is stochastically complete, then $\Delta_{\min}^{(1)}$ generates a strongly continuous contraction semigroup on $\ell^1(X,\mu)$. In particular, $\Delta_{\min}^{(1)}$ is m-accretive (\cite[Theorem 2.11]{milatovic2015maximal}). 
We now extend this result by showing that stochastic completeness at infinity is actually equivalent to the m-accretivity of $\Delta_{\min}^{(1)}$
as well as the m-accretivity of $\Delta^{(\infty)}$.
\begin{theorem}\label{thm:stoch_comp_maccr}
Let $G$ be a graph. The following statements are equivalent:
\begin{enumerate}[(i)]
    \item[\textup{(i)}] $G$ is stochastically complete at infinity.
     \item[\textup{(ii)}] $\Delta^{(\infty)}$ is m-accretive.
    \item[\textup{(iii)}] $\Delta^{(1)}_{\min}$ is m-accretive.
\end{enumerate}
\end{theorem}
\begin{proof}
(i)\,$\Longleftrightarrow$\,(ii): 
By Theorem~\ref{t:acc_inj}, the m-accretivity of $\Delta^{(\infty)}$ is equivalent to the injectivity of $S^{(\infty)}_\lambda=\id+\lambda\Delta^{(\infty)}$
on $\dom(\Delta^{(\infty)})$ 
for some (equivalently, all) $\lambda>0$. 
The conclusion now follows from Lemma~\ref{l:SC}.

(ii)\,$\Longleftrightarrow$\,(iii): 
This was established in the course of the proof of 
Theorem~\ref{t:acc_inj} for $p=\infty$. 
In particular, the closed range theorem combined with the adjoint identification $\bigl(\Delta_{\min}^{(1)}\bigr)^{\ast}=\Delta^{(\infty)}$ given in Lemma~\ref{l:infty_adjoint} yields, for every $\lambda>0$,
\[
S^{(1)}_{\min, \lambda} \text{ is surjective onto } \ell^1(X,\mu)
\quad\Longleftrightarrow\quad
S^{(\infty)}_\lambda \text{ is injective on } \dom(\Delta^{(\infty)})
\]
where $S^{(1)}_{\min, \lambda}=\id+\lambda\Delta_{\min}^{(1)}$. 

Since $\Delta_{\min}^{(1)}$ is accretive on $\ell^1(X,\mu)$ by Lemma~\ref{l:min_accretivity}, 
the surjectivity of $S^{(1)}_{\min, \lambda}$ for every $\lambda>0$ is equivalent to the m-accretivity of $\Delta_{\min}^{(1)}$. 
On the other hand, by Theorem~\ref{t:acc_inj}, the injectivity of $S^{(\infty)}_\lambda$ is equivalent to the m-accretivity of $\Delta^{(\infty)}$. This completes the proof.
\end{proof}

We recall that if the vertex degree is bounded, then all minimal and maximal Laplacians are m-accretive
for $p\in[1,\infty)$, see Theorem~\ref{t:Lap_accretive} and Proposition~\ref{p:bounded_Lap}. 
Furthermore, for $p=\infty$, we get that $\Delta^{(\infty)}$ is m-accretive
whenever $1/\Deg$ is not summable over paths, see Proposition~\ref{p:Lap_accretive2}.
In particular, the non-summability of $1/\Deg$ implies stochastic completeness at infinity
extending results found in \cite{KL10, Woj21}. 

\begin{corollary}
    If $G$ satisfies $\sum_{n} \frac{1}{\Deg(x_{n})}=\infty$ for every infinite path $(x_n)$, then
    $G$ is stochastically complete at infinity.
\end{corollary}
\begin{proof}
    This follows immediately by combining Proposition~\ref{p:Lap_accretive2} and 
    Theorem~\ref{thm:stoch_comp_maccr}.
\end{proof}

If a graph satisfies stochastic completeness at infinity, $\EM$ and $\IP$, then Theorem~2.15 in 
\cite{milatovic2015maximal} states that $\Delta_{\min}^{(1)}=\Delta^{(1)}$. 
We now improve this result by removing the $\EM$ condition.
\begin{corollary}\label{c:min_ell1}
    If $G$ is stochastically complete at infinity and satisfies $\IP$, then $\Delta^{(1)}=\Delta_{\min}^{(1)}$.
\end{corollary}
\begin{proof}
    {The infinite path measure condition $\IP$ implies that $\Delta^{(1)}$ is m-accretive by Proposition~\ref{p:Lap_accretive}
    while stochastic completeness at infinity implies that $\deltam^{(1)}$ is m-accretive by Theorem~\ref{thm:stoch_comp_maccr}. Since \EM[1] always holds,
    Lemma~\ref{l:min_accretivity2} gives $\Delta^{(1)}=\Delta_{\min}^{(1)}$.}
\end{proof}

We also get an immediate corollary in the case of measure uniformly bounded from below.

\begin{corollary}\label{c:sc_um}
If $G$ is stochastically complete at infinity 
and satisfies $\UM$, then $\Delta^{(p)}$ is m-accretive for every $p\in [1,\infty]$.
Furthermore, $\Delta^{(p)}=\deltam^{(p)}$ for every $p \in [1,\infty)$.
\end{corollary}
\begin{proof}

By $\UM$ we have that $\ell^{p}(X,\mu) \subseteq \ell^{\infty}(X)$ for all $p \in [1,\infty]$. 
In particular, $\dom(\Delta^{(p)})  \subseteq \dom(\Delta^{(\infty)}) $ and all operators
are restrictions of $\Delta$.
Let $u \in \dom(\Delta^{(p)})$ be such that $S^{(p)}_\lambda u = 0$ for $p \in [1,\infty]$. 
If $G$ is stochastically complete at infinity, then $S^{(\infty)}_\lambda$ is injective
by Lemma~\ref{l:SC}. Therefore, since $u \in \dom(\Delta^{(\infty)})$, we get $S^{(\infty)}_\lambda u = S^{(p)}_\lambda u = 0$ which implies $u=0$, i.e., $S^{(p)}_\lambda$ is injective and  then  $\Delta^{(p)}$ is m-accretive for all $p \in [1,\infty]$ by Theorem~\ref{t:acc_inj}.

It remains to show $\Delta^{(p)}=\deltam^{(p)}$ for $p \in [1,\infty)$. 
{ For $p=1$, the equality follows from Corollary~\ref{c:min_ell1}
since $\UM$ implies $\IP$. For $p\in(1,\infty)$, $\UM$ implies both $\IP$
and \EM[p], the latter by Corollary~\ref{c:emp}. Thus,
$\Delta^{(p)}=\deltam^{(p)}$ by Proposition~\ref{p:IP_Lap}.}
\end{proof}

\appendix

\section{\texorpdfstring{On the proof of Theorem~1 in \cite{bianchi2022generalized}}{On the proof of Theorem 1 in Bianchi et al.}}\label{sec:proof_Theo1}

The aim of this appendix is to revisit the original proof of \cite[Theorem 1]{bianchi2022generalized}, which, as pointed out in \cite{Ble24}, contains a gap. In Section~\ref{sec:m-accretivity_Lomega}, we presented a refined argument in Theorem~\ref{thm:new} that yields a strengthened version of \cite[Theorem 1]{bianchi2022generalized} and, along the way, closes this gap.

Adapted to the notation of this manuscript, the original statement
in question reads as follows.
\begin{namedtheorem}[Statement of Theorem 1 in \cite{bianchi2022generalized}]
Let $G$ be a graph. Then, there exists a dense subset $\Omega \subseteq \dom(\Deltaphi)$ such that $\Deltaphi_{|\Omega}$ is accretive.  Moreover, for every $\lambda >0$ and for every $g \in \ell^{1,\pm}(X,\mu)$ there exists a unique $u \in \ell^{1,\pm}(X,\mu)\cap\Omega$ such that
\begin{equation*}\label{eq:main_equation}
\mathcal{S}_\lambda u=g.
\end{equation*}
If  either  $G$ is locally finite, or $\UM$ holds, or $G$ satisfies $\B$ and $\Phi$ satisfies $\C$, then 
$\operatorname{id}+\lambda\Deltaphi$ restricted to $\Omega$ is also surjective.
In particular, $\Deltaphi_{|\Omega}$ is m-accretive.
Moreover, in all cases, the solution $u$ satisfies the
contractivity estimate  $$||u||_1\leq ||g||_1.$$
\end{namedtheorem}

The set $\Omega$ in the above statement coincides with that of Definition~\ref{def:Omega} in the present paper and is fixed by choosing any exhaustion $(X_n)$ of connected sets. 

For the reader's convenience and ease of comparison, we restate Theorem~\ref{thm:new} below.

\begin{namedtheorem}[Statement of Theorem~\ref{thm:new}]
Let $G$ be a  graph. Then, there exists a dense subset $\Omega\subseteq\dom(\Deltaphi)$ such that $\Deltaphi_{|\Omega}$ is m-accretive. In particular, for every $\lambda>0$ and every $g\in \ell^{1}(X,\mu)$, there exists a unique $u\in \Omega$ such that
\[
\mathcal{S}_\lambda u=g
\]
and the solution $u$ satisfies the contractivity estimate
\[
\|u\|_1\leq \|g\|_1.
\]
Moreover, if $g\geq 0$, then $u\geq 0$, and if $g\leq 0$, then $u\leq 0$.
\end{namedtheorem}

First, we observe that the additional hypotheses in \cite[Theorem 1]{bianchi2022generalized}, that is, local finiteness or $\UM$  or $\B$+$\C$, were introduced to justify the interchange of limit and summation
$$
\sum_{y \in X} w(x,y) \Phi u(y) = \sum_{y \in X} w(x,y) \lim_{k\to\infty} \Phi \varphi_{n_k}(y) = \lim_{k\to\infty} \sum_{y \in X} w(x,y) \Phi \varphi_{n_k}(y)
$$
in~\eqref{eq:extraHP} (see also~\cite[Equation (4.24)]{bianchi2022generalized}). However, this interchange comes for free from dominated convergence: as shown in~\eqref{eq:dominate_convergence_2},
\begin{equation*}
\sum_{y\in X} w(x,y)|\Phi \varphi_{n_k}(y)|
\le
\sum_{y\in X} w(x,y)\bigl(|\Phi u^-(y)| + |\Phi u^+(y)|\bigr)
< \infty
\end{equation*}
since both the auxiliary comparison solutions $u^+$ and  $u^-$ lie in the formal domain of $\Delta\Phi$. This allows us to drop any assumption on the graph $G$. As a side note, we observe that $\UM$  or $\B$+$\C$ imply $\Deltaphiom = \Deltaphi$, as proved in Theorem~\ref{t:lf} and Theorem~\ref{t:accretive_bd}, respectively.

Turning to the gap, in the signed  case $g \in \ell^1(X,\mu)$ of the proof of \cite[Theorem 1]{bianchi2022generalized}, the solution $u$ of $\mathcal{S}_\lambda u = g$ was obtained by extracting a subsequence $(\varphi_{n_k})$ from a sequence $(\varphi_n)$ supported on an initial exhaustion $(X_n)$. This guarantees $u \in \Omega_{(X_{n_k})}$ but not $u \in \Omega = \Omega_{(X_n)}$ as claimed since, in general, $\Omega_{(X_n)} \subseteq \Omega_{(X_{n_k})}$. 

The gap is resolved in the current paper by passing first to a suitable diagonal exhaustion $Y_k = X_{n_k}$, with $(X_{n_k})$ a subsequence of $(X_n)$, and then setting $\Omega = \Omega_{(Y_k)}$ as carried out in Step~3 of the proof of Theorem~\ref{thm:new}. The exhaustion $(Y_k)$, and hence the set $\Omega$, depend neither on the parameter $\lambda$ nor on the datum $g$. Over that specific exhaustion $(Y_k)$, it is then proved that for any fixed pair $\lambda >0$ 
and $g \in \ell^1(X,\mu)$, 
the constructed sequence $(\varphi_{n_k})$ converges to a solution 
$u\in \Omega$ of $\mathcal{S}_\lambda u = g$.

\section*{Acknowledgments}
We thank Marcel Schmidt and Uwe Blechschmidt for pointing out the gap in the proof
of Theorem~1 in \cite{bianchi2022generalized} and for helpful discussions.
We thank Bobo Hua for inspiring discussions while visiting Sun Yat-sen University.
We thank Ognjen Milatovic and an anonymous referee for helpful comments
on an earlier version of this paper which led to improvements in the results
and presentation.
We thank the Universit\`a dell'Insubria, the Graduate Center and York College of CUNY, and Sun Yat-sen University 
for hosting us while parts of this work were completed.
The research of D.~B.~is  supported by the Startup Fund of Sun Yat-sen University.
The research of M.~K.~is supported by the DFG.
A.~G.~S.~is a member of the GNAMPA-INdAM group {``Equazioni Differenziali e Sistemi Dinamici''}.
The research of R.~K.~W.~is partially funded by the Simons Foundation, in the form
of a Travel Support for Mathematicians gift, and by PSC-CUNY, in the form
of a Department Chair CUNY Research Foundation (RF) Account. 
\printbibliography

@article{Mawhin2013,
  author  = {Mawhin, Jean},
  title   = {Variations on {Poincar\'e--Miranda}'s Theorem},
  journal = {Advanced Nonlinear Studies},
  volume  = {13},
  number  = {1},
  pages   = {209--217},
  year    = {2013},
  doi     = {10.1515/ans-2013-0112}
}

@article{benilan1981continuous,
  title={The continuous dependence on $\varphi$ of solutions of $u_t - \Delta \varphi (u)= 0$},
  author={B{\'e}nilan, Philippe and Crandall, Michael G},
  journal={Indiana University Mathematics Journal},
  volume={30},
  number={2},
  pages={161--177},
  year={1981},
  publisher={JSTOR}
}

@misc{berchio2026fractional,
  author = {Berchio, Elvise and Santagati, Federico and Vallarino, Maria},
  title = {The fractional porous medium equation on graphs},
  year = {2026},
  eprint = {2606.23360},
  eprinttype = {arXiv},
  eprintclass = {math.AP},
  url = {https://arxiv.org/abs/2606.23360}
}

@article{berchio2026semilinear,
  author = {Berchio, Elvise and Bianchi, Davide and Setti, Alberto G. and Vallarino, Maria},
  title = {Semilinear diffusion equations on infinite graphs: the dissipative and Lipschitz cases},
  journal = {Journal of Differential Equations},
  volume = {476},
  pages = {114485},
  year = {2026},
  doi = {10.1016/j.jde.2026.114485}
}

@article{hanche2010kolmogorov,
  author = {Hanche-Olsen, Harald and Holden, Helge},
  title = {The Kolmogorov--Riesz compactness theorem},
  journal = {Expositiones Mathematicae},
  volume = {28},
  number = {4},
  pages = {385--394},
  year = {2010},
  doi = {10.1016/j.exmath.2010.03.001}
}

@article{BrezisStraussJFA1973,
  author = {Brezis, Ha{\"\i}m and Strauss, Walter A.},
  title = {Semi-linear second-order elliptic equations in {$L^1$}},
  journal = {Journal of the Mathematical Society of Japan},
  volume = {25},
  number = {4},
  pages = {565--590},
  year = {1973},
  doi = {10.2969/jmsj/02540565}
}

@article{CrandallLiggett1971,
  author = {Crandall, Michael G. and Liggett, Thomas M.},
  title = {Generation of semi-groups of nonlinear transformations on general {B}anach spaces},
  journal = {American Journal of Mathematics},
  volume = {93},
  number = {2},
  pages = {265--298},
  year = {1971},
  doi = {10.2307/2373376}
}

@article{ma2022porous,
  author = {Ma, Li},
  title = {Porous media equation on locally finite graphs},
  journal = {Archivum Mathematicum},
  volume = {58},
  number = {3},
  pages = {177--187},
  year = {2022},
  doi = {10.5817/AM2022-3-177}
}

@article{schmidt2025nonlinear,
  author = {Schmidt, Marcel and Zimmermann, Ian},
  title = {A nonlinear characterization of stochastic completeness of graphs},
  journal = {Mathematische Nachrichten},
  volume = {298},
  number = {3},
  pages = {925--943},
  year = {2025},
  doi = {10.1002/mana.202400436}
}

@article{HM15,
  author = {Hua, Bobo and Mugnolo, Delio},
  title = {Time regularity and long-time behavior of parabolic $p$-Laplace equations on infinite graphs},
  journal = {Journal of Differential Equations},
  shortjournal = {J. Differ. Equ.},
  volume = {259},
  pages = {6162--6190},
  year = {2015},
  doi = {10.1016/j.jde.2015.07.018}
}

@article{chill2024real,
  author = {Chill, Ralph and Sharma, Praveen and Srivastava, Sachi},
  title = {Real interpolation of functions with applications to accretive
                  operators on {B}anach spaces},
  journal = {Journal of Differential Equations},
  shortjournal = {J. Differential Equations},
  volume = {402},
  pages = {554--592},
  year = {2024},
  issn = {0022-0396},
  doi = {10.1016/j.jde.2024.05.024},
  mrclass = {46B70 (35K92 47H06 47H20 47J35)},
  mrnumber = {4749397}
}

@article{beurich2024interpolation,
  author = {Beurich, Johann and Sharma, Praveen},
  title = {Interpolation results for convergence of implicit {E}uler
                  schemes with accretive operators},
  journal = {NoDEA. Nonlinear Differential Equations and Applications},
  shortjournal = {NoDEA Nonlinear Differential Equations Appl.},
  volume = {31},
  number = {6},
  pages = {Paper No. 109, 16},
  year = {2024},
  issn = {1021-9722},
  doi = {10.1007/s00030-024-01003-9},
  mrclass = {46B70 (34G10 35K90 35R70 47H06 47H20 47J35)},
  mrnumber = {4807081},
  mrreviewer = {Constantin B\u{a}cu\c{t}\u{a}}
}

@article{nochetto2006nonlinear,
  author = {Nochetto, Ricardo H. and Savar\'{e}, Giuseppe},
  title = {Nonlinear evolution governed by accretive operators in
                  {B}anach spaces: error control and applications},
  journal = {Mathematical Models and Methods in Applied Sciences},
  shortjournal = {Math. Models Methods Appl. Sci.},
  volume = {16},
  number = {3},
  pages = {439--477},
  year = {2006},
  issn = {0218-2025},
  doi = {10.1142/S0218202506001224},
  mrclass = {34G25 (35K90 47H06 47J35 65J15 65M15)},
  mrnumber = {2238759},
  mrreviewer = {J\'{e}r\^{o}me Bastien}
}

@article{anne2020m,
  author = {Ann{\'e}, Colette and Balti, Marwa and Torki-Hamza, Nabila},
  title = {m-accretive Laplacian on a non symmetric graph},
  journal = {Indagationes Mathematicae},
  shortjournal = {Indag. Math.},
  volume = {31},
  number = {2},
  pages = {277--293},
  year = {2020},
  doi = {10.1016/j.indag.2020.01.005}
}

@book{barbu2010nonlinear,
  author = {Barbu, Viorel},
  title = {{Nonlinear Differential Equations of Monotone Types in Banach Spaces}},
  series = {Springer Monographs in Mathematics},
  publisher = {Springer, New York, NY},
  year = {2010},
  doi = {10.1007/978-1-4419-5542-5}
}

@unpublished{benilan1988evolution,
  author = {B{\'e}nilan, Philippe and Crandall, Michael G. and Pazy, Amnon},
  title = {Nonlinear evolution equations in Banach spaces},
  note = {Preprint},
  year = {1988}
}

@article{bianchi2022generalized,
  author = {Bianchi, Davide and Setti, Alberto Giulio and Wojciechowski, Rados{\l}aw K.},
  title = {The generalized porous medium equation on graphs:	existence and uniqueness of solutions with $\ell^1$ data},
  journal = {Calculus of Variations and Partial Differential Equations},
  shortjournal = {Calc. Var. Partial Differ. Equ.},
  volume = {61},
  number = {5},
  pages = {Paper No. 171, 42},
  year = {2022},
  doi = {10.1007/s00526-022-02249-w}
}

@book{Ble24,
  author = {Blechschmidt, Uwe},
  title = {Existenz und Eindeutigkeit von
    L{\"o}sungen zum Cauchy-Problem der
    verallgemeinerten
    Por{\"o}se-Medien-Gleichung auf $\ell^1$},
  publisher = {Friedrich-Schiller-Universit{\"a}t Jena},
  year = {2024},
  pages = {87},
  note = {Thesis (Bachelorarbeit)}
}

@article{punzo2025semilinear,
  author = {Punzo, Fabio and Sacco, Alessandro},
  title = {On a semilinear parabolic equation with time-dependent source
                  term on infinite graphs},
  journal = {Journal of Evolution Equations},
  shortjournal = {J. Evol. Equ.},
  volume = {26},
  number = {1},
  pages = {Paper No. 13, 16},
  year = {2026},
  issn = {1424-3199,1424-3202},
  doi = {10.1007/s00028-025-01159-6},
  mrclass = {35A01 (35A02 35B44 35K05 35K58)},
  mrnumber = {5017866}
}

@article{BG15,
  author = {Bonnefont, Michel and Gol\'{e}nia, Sylvain},
  title = {Essential spectrum and {W}eyl asymptotics for discrete
                  {L}aplacians},
  journal = {Annales de la Facult\'{e} des Sciences de Toulouse. Math\'{e}matiques. S\'{e}rie 6},
  shortjournal = {Ann. Fac. Sci. Toulouse Math. (6)},
  volume = {24},
  number = {3},
  pages = {563--624},
  year = {2015},
  issn = {0240-2963},
  doi = {10.5802/afst.1456},
  mrclass = {58J50 (47A10)},
  mrnumber = {3403733},
  mrreviewer = {De Tang Zhou}
}

@article{CTT11,
  author = {Colin de Verdi\`ere, Yves and Torki-Hamza, Nabila and Truc,
                  Fran\c{c}oise},
  title = {Essential self-adjointness for combinatorial {S}chr\"{o}dinger
                  operators {II}---metrically non complete graphs},
  journal = {Mathematical Physics, Analysis and Geometry},
  shortjournal = {Math. Phys. Anal. Geom.},
  volume = {14},
  number = {1},
  pages = {21--38},
  year = {2011},
  issn = {1385-0172},
  doi = {10.1007/s11040-010-9086-7},
  mrclass = {81Q35 (05C12 05C63 35R02 47B25)},
  mrnumber = {2782792},
  mrreviewer = {Pavel V. Exner}
}

@book{deimling1985nonlinear,
  author = {Deimling, Klaus},
  title = {Nonlinear Functional Analysis},
  publisher = {Springer-Verlag Berlin Heidelberg},
  year = {1985},
  doi = {10.1007/978-3-662-00547-7}
}

@book{EngelNagel2000,
  author = {Engel, Klaus-Jochen and Nagel, Rainer},
  title = {One-parameter semigroups for linear evolution equations},
  series = {Graduate Texts in Mathematics},
  volume = {194},
  publisher = {Springer-Verlag, New York},
  year = {2000},
  pages = {xxii+586},
  isbn = {0-387-98463-1},
  doi = {10.1007/b97696}
}

@book{Folland1999RealAnalysis,
  author = {Folland, Gerald B.},
  title = {Real Analysis: Modern Techniques and Their Applications},
  edition = {Second},
  publisher = {Wiley-Interscience},
  year = {1999}
}

@article{GHKLW15,
  author = {Georgakopoulos, Agelos and Haeseler, Sebastian and Keller,
                  Matthias and Lenz, Daniel and Wojciechowski, Rados{\l}aw K.},
  title = {Graphs of finite measure},
  journal = {Journal de Math\'{e}matiques Pures et Appliqu\'{e}es},
  shortjournal = {J. Math. Pures Appl. (9)},
  volume = {103},
  number = {5},
  pages = {1093--1131},
  year = {2015},
  issn = {0021-7824},
  doi = {10.1016/j.matpur.2014.10.006},
  mrclass = {31E05 (05C63 47B39)},
  mrnumber = {3333051},
  mrreviewer = {Peter I. Kogut}
}

@book{goebel1990topics,
  author = {Goebel, Kazimierz and Kirk, W. A.},
  title = {Topics in metric fixed point theory},
  series = {Cambridge Studies in Advanced Mathematics},
  volume = {28},
  publisher = {Cambridge University Press, Cambridge},
  year = {1990},
  pages = {viii+244},
  isbn = {0-521-38289-0},
  doi = {10.1017/CBO9780511526152},
  mrclass = {47H10 (47-02 47H06)},
  mrnumber = {1074005},
  mrreviewer = {M. M. Day}
}

@article{Gol14,
  author = {Gol\'{e}nia, Sylvain},
  title = {Hardy inequality and asymptotic eigenvalue distribution for
                  discrete {L}aplacians},
  journal = {Journal of Functional Analysis},
  shortjournal = {J. Funct. Anal.},
  volume = {266},
  number = {5},
  pages = {2662--2688},
  year = {2014},
  issn = {0022-1236},
  doi = {10.1016/j.jfa.2013.10.012},
  mrclass = {35R02 (35J05 35Q60 47A10)},
  mrnumber = {3158705},
  mrreviewer = {M. S. Agranovich}
}

@article{GLY16a,
  author = {Grigor'yan, Alexander and Lin, Yong and Yang, Yunyan},
  title = {Kazdan-{W}arner equation on graph},
  journal = {Calculus of Variations and Partial Differential Equations},
  shortjournal = {Calc. Var. Partial Differential Equations},
  volume = {55},
  number = {4},
  pages = {Art. 92, 13},
  year = {2016},
  issn = {0944-2669},
  doi = {10.1007/s00526-016-1042-3},
  mrclass = {34B45 (35A15 35R02 58E30)},
  mrnumber = {3523107},
  mrreviewer = {Enrico Serra}
}

@article{GLY16b,
  author = {Grigor'yan, Alexander and Lin, Yong and Yang, Yunyan},
  title = {Yamabe type equations on graphs},
  journal = {Journal of Differential Equations},
  shortjournal = {J. Differential Equations},
  volume = {261},
  number = {9},
  pages = {4924--4943},
  year = {2016},
  issn = {0022-0396},
  doi = {10.1016/j.jde.2016.07.011},
  mrclass = {35R03 (35J20 35J25 35J91 35J92 58E30)},
  mrnumber = {3542963},
  mrreviewer = {Anna Maria Candela}
}

@article{GMP,
  author = {Grillo, Gabriele and Meglioli, Giulia and Punzo, Fabio},
  title = {Blow-up and global existence for semilinear parabolic
                  equations on infinite graphs},
  journal = {Calculus of Variations and Partial Differential Equations},
  shortjournal = {Calc. Var. Partial Differential Equations},
  volume = {65},
  number = {4},
  pages = {Paper No. 114, 20},
  year = {2026},
  issn = {0944-2669,1432-0835},
  doi = {10.1007/s00526-026-03291-8},
  mrclass = {35A01 (35A02 35B44 35K05 35K58)},
  mrnumber = {5037980}
}

@article{HKLW12,
  author = {Haeseler, Sebastian and Keller, Matthias and Lenz, Daniel and
                  Wojciechowski, Rados{\l}aw},
  title = {Laplacians on infinite graphs: {D}irichlet and {N}eumann
                  boundary conditions},
  journal = {Journal of Spectral Theory},
  shortjournal = {J. Spectr. Theory},
  volume = {2},
  number = {4},
  pages = {397--432},
  year = {2012},
  issn = {1664-039X},
  doi = {10.4171/jst/35},
  mrclass = {47B25 (05C63 31C20)},
  mrnumber = {2947294},
  mrreviewer = {Eugen J. Ionascu}
}

@article{HK14,
  author = {Hua, Bobo and Keller, Matthias},
  title = {Harmonic functions of general graph {L}aplacians},
  journal = {Calculus of Variations and Partial Differential Equations},
  shortjournal = {Calc. Var. Partial Differential Equations},
  volume = {51},
  number = {1-2},
  pages = {343--362},
  year = {2014},
  issn = {0944-2669},
  doi = {10.1007/s00526-013-0677-6},
  mrclass = {31C20 (58E20)},
  mrnumber = {3247392},
  mrreviewer = {Wolfgang Woess}
}

@article{HKSW23,
  author = {Hua, Bobo and Keller, Matthias and Schwarz, Michael and Wirth,
                  Melchior},
  title = {Sobolev-type inequalities and eigenvalue growth on graphs with
                  finite measure},
  journal = {Proceedings of the American Mathematical Society},
  shortjournal = {Proc. Amer. Math. Soc.},
  volume = {151},
  number = {8},
  pages = {3401--3414},
  year = {2023},
  issn = {0002-9939,1088-6826},
  doi = {10.1090/proc/14361},
  mrclass = {47A75 (05C63 39A70 46E39)},
  mrnumber = {4591775},
  mrreviewer = {Jos\'{e}\ Francisco Alves de Oliveira}
}

@article{HX23,
  author = {Hua, Bobo and Xu, Wendi},
  title = {Existence of ground state solutions to some nonlinear
                  {S}chr\"{o}dinger equations on lattice graphs},
  journal = {Calculus of Variations and Partial Differential Equations},
  shortjournal = {Calc. Var. Partial Differential Equations},
  volume = {62},
  number = {4},
  pages = {Paper No. 127, 17},
  year = {2023},
  issn = {0944-2669},
  doi = {10.1007/s00526-023-02470-1},
  mrclass = {35Q55 (39A14 58E30)},
  mrnumber = {4568177}
}

@article{HKMW13,
  author = {Huang, Xueping and Keller, Matthias and Masamune, Jun and
                  Wojciechowski, Rados{\l}aw K.},
  title = {A note on self-adjoint extensions of the {L}aplacian on
                  weighted graphs},
  journal = {Journal of Functional Analysis},
  shortjournal = {J. Funct. Anal.},
  volume = {265},
  number = {8},
  pages = {1556--1578},
  year = {2013},
  issn = {0022-1236},
  doi = {10.1016/j.jfa.2013.06.004},
  mrclass = {47B37 (05C22)},
  mrnumber = {3079229},
  mrreviewer = {George Stacey Staples}
}

@article{HMW21,
  author = {Hua, Bobo and Masamune, Jun and Wojciechowski, Rados{\l}aw K.},
  title = {Essential self-adjointness and the {$L^2$}-{L}iouville
                  property},
  journal = {The Journal of Fourier Analysis and Applications},
  shortjournal = {J. Fourier Anal. Appl.},
  volume = {27},
  number = {2},
  pages = {Paper No. 26, 27},
  year = {2021},
  issn = {1069-5869},
  doi = {10.1007/s00041-021-09833-2},
  mrclass = {47B25 (05C50 58J05)},
  mrnumber = {4231682}
}

@article{IKMW25,
  author = {Inoue, Atsushi and Ku, Sean and Masamune, Jun and
                  Wojciechowski, Rados{\l}aw K.},
  title = {Essential {S}elf-{A}djointness of the {L}aplacian on
                  {W}eighted {G}raphs: {H}armonic {F}unctions, {S}tability,
                  {C}haracterizations and {C}apacity},
  journal = {Mathematical Physics, Analysis and Geometry},
  shortjournal = {Math. Phys. Anal. Geom.},
  volume = {28},
  number = {2},
  pages = {Paper No. 12},
  year = {2025},
  issn = {1385-0172},
  doi = {10.1007/s11040-025-09498-z},
  mrclass = {31C20 (47B39 60J27)},
  mrnumber = {4910968}
}

@article{Jor08,
  author = {Jorgensen, Palle E. T.},
  title = {Essential self-adjointness of the graph-{L}aplacian},
  journal = {Journal of Mathematical Physics},
  shortjournal = {J. Math. Phys.},
  volume = {49},
  number = {7},
  pages = {073510, 33},
  year = {2008},
  issn = {0022-2488},
  doi = {10.1063/1.2953684},
  mrclass = {47B39 (05C50 31C20 47B25 94C99)},
  mrnumber = {2432048},
  mrreviewer = {Alexey V. Borovskikh}
}

@article{JP11,
  author = {Jorgensen, Palle E. T. and Pearse, Erin P. J.},
  title = {Spectral reciprocity and matrix representations of unbounded
                  operators},
  journal = {Journal of Functional Analysis},
  shortjournal = {J. Funct. Anal.},
  volume = {261},
  number = {3},
  pages = {749--776},
  year = {2011},
  issn = {0022-1236},
  doi = {10.1016/j.jfa.2011.01.016},
  mrclass = {47A10 (05C50)},
  mrnumber = {2799579},
  mrreviewer = {Vivien G. Miller}
}

@book{kato2013perturbation,
  author = {Kato, Tosio},
  title = {Perturbation Theory for Linear Operators},
  series = {Classics in Mathematics},
  volume = {132},
  publisher = {Springer-Verlag, Berlin},
  year = {1995},
  doi = {10.1007/978-3-642-66282-9}
}

@incollection{Kel15,
  author = {Keller, Matthias},
  title = {Intrinsic metrics on graphs: a survey},
  booktitle = {Mathematical technology of networks},
  series = {Springer Proc. Math. Stat.},
  volume = {128},
  pages = {81--119},
  publisher = {Springer, Cham},
  year = {2015},
  doi = {10.1007/978-3-319-16619-3_7},
  mrclass = {31C20 (05C12 05C50 47B37 58J50)},
  mrnumber = {3375157}
}

@article{KL10,
  author = {Keller, M. and Lenz, D.},
  title = {Unbounded {L}aplacians on graphs: basic spectral properties
                  and the heat equation},
  journal = {Mathematical Modelling of Natural Phenomena},
  shortjournal = {Math. Model. Nat. Phenom.},
  volume = {5},
  number = {4},
  pages = {198--224},
  year = {2010},
  issn = {0973-5348,1760-6101},
  doi = {10.1051/mmnp/20105409},
  mrclass = {60J45 (60J27)},
  mrnumber = {2662456},
  mrreviewer = {Ping\ He}
}

@article{KL12,
  author = {Keller, Matthias and Lenz, Daniel},
  title = {Dirichlet forms and stochastic completeness of graphs and
                  subgraphs},
  journal = {Journal f\"{u}r die Reine und Angewandte Mathematik},
  shortjournal = {J. Reine Angew. Math.},
  volume = {666},
  pages = {189--223},
  year = {2012},
  issn = {0075-4102},
  doi = {10.1515/CRELLE.2011.122},
  mrclass = {60J45 (05C22 31C20)},
  mrnumber = {2920886},
  mrreviewer = {Fr\'{e}d\'{e}ric Math\'{e}us}
}

@article{KLSW17,
  author = {Keller, Matthias and Lenz, Daniel and Schmidt, Marcel and
                  Wojciechowski, Rados{\l}aw K.},
  title = {Note on uniformly transient graphs},
  journal = {Revista Matem\'{a}tica Iberoamericana},
  shortjournal = {Rev. Mat. Iberoam.},
  volume = {33},
  number = {3},
  pages = {831--860},
  year = {2017},
  issn = {0213-2230,2235-0616},
  doi = {10.4171/RMI/957},
  mrclass = {39A70 (58J50)},
  mrnumber = {3713033}
}

@article{KLW13,
  author = {Keller, Matthias and Lenz, Daniel and Wojciechowski, Rados{\l}aw K.},
  title = {Volume growth, spectrum and stochastic completeness of
                  infinite graphs},
  journal = {Mathematische Zeitschrift},
  shortjournal = {Math. Z.},
  volume = {274},
  number = {3-4},
  pages = {905--932},
  year = {2013},
  issn = {0025-5874,1432-1823},
  doi = {10.1007/s00209-012-1101-1},
  mrclass = {05C63 (05C50 39A12 58J35)},
  mrnumber = {3078252},
  mrreviewer = {Carlos\ M.\ da Fonseca}
}

@book{keller2021graphs,
  author = {Keller, Matthias and Lenz, Daniel and Wojciechowski, Rados{\l}aw K.},
  title = {Graphs and discrete {D}irichlet spaces},
  series = {Grundlehren der mathematischen Wissenschaften [Fundamental
                  Principles of Mathematical Sciences]},
  volume = {358},
  publisher = {Springer, Cham},
  year = {2021},
  pages = {xv+668},
  isbn = {978-3-030-81458-8},
  doi = {10.1007/978-3-030-81459-5},
  mrclass = {05-01 (05C63 35P05)},
  mrnumber = {4383783}
}

@article{KM24,
  author = {Keller, Matthias and M\"{u}nch, Florentin},
  title = {Gradient estimates, {B}akry-{E}mery {R}icci curvature and
                  ellipticity for unbounded graph {L}aplacians},
  journal = {Communications in Analysis and Geometry},
  shortjournal = {Comm. Anal. Geom.},
  volume = {32},
  number = {2},
  pages = {343--364},
  year = {2024},
  issn = {1019-8385},
  doi = {10.4310/cag.241015004626},
  mrclass = {53C21 (05C63 58J65)},
  mrnumber = {4828188}
}

@article{KMW25,
  author = {Keller, Matthias and M\"{u}nch, Florentin and Wojciechowski,
                  Rados{\l}aw K.},
  title = {Neumann semigroup, subgraph convergence, form uniqueness,
                  stochastic completeness and the {F}eller property},
  journal = {Journal of Geometric Analysis},
  shortjournal = {J. Geom. Anal.},
  volume = {35},
  number = {1},
  pages = {Paper No. 14, 23},
  year = {2025},
  issn = {1050-6926},
  doi = {10.1007/s12220-024-01838-9},
  mrclass = {58J35 (05C63 35K08 35R02 39A12 47B25)},
  mrnumber = {4821965}
}

@article{KS18,
  author = {Keller, Matthias and Schwarz, Michael},
  title = {The {K}azdan-{W}arner equation on canonically compactifiable
                  graphs},
  journal = {Calculus of Variations and Partial Differential Equations},
  shortjournal = {Calc. Var. Partial Differential Equations},
  volume = {57},
  number = {2},
  pages = {Paper No. 70, 18},
  year = {2018},
  issn = {0944-2669},
  doi = {10.1007/s00526-018-1329-7},
  mrclass = {34B45 (35A15 35J61 35R01 35R02)},
  mrnumber = {3776360},
  mrreviewer = {Wenxiong Chen}
}

@article{HKMRW,
  author = {Hernandez, Luis and Ku, Sean and Masamune, Jun and Romanelli, Genevieve and Wojciechowski, Rados{\l}aw K.},
  title = {Form uniqueness for graphs with weakly spherically symmetric
                  ends},
  journal = {Journal of Mathematical Analysis and Applications},
  shortjournal = {J. Math. Anal. Appl.},
  volume = {564},
  number = {1},
  pages = {Paper No. 130855},
  year = {2026},
  issn = {0022-247X,1096-0813},
  doi = {10.1016/j.jmaa.2026.130855},
  mrclass = {31C20 (05C63 31C25 47B25)},
  mrnumber = {5083880}
}

@article{LSZ23,
  author = {Lenz, Daniel and Schmidt, Marcel and Zimmermann, Ian},
  title = {Blow-up of nonnegative solutions of an abstract semilinear
                  heat equation with convex source},
  journal = {Calculus of Variations and Partial Differential Equations},
  shortjournal = {Calc. Var. Partial Differential Equations},
  volume = {62},
  number = {4},
  pages = {Paper No. 140, 19},
  year = {2023},
  issn = {0944-2669},
  doi = {10.1007/s00526-023-02482-x},
  mrclass = {35K05 (35K08 35K58 47D07)},
  mrnumber = {4581147}
}

@article{LW17,
  author = {Lin, Yong and Wu, Yiting},
  title = {The existence and nonexistence of global solutions for a
                  semilinear heat equation on graphs},
  journal = {Calculus of Variations and Partial Differential Equations},
  shortjournal = {Calc. Var. Partial Differential Equations},
  volume = {56},
  number = {4},
  pages = {Paper No. 102, 22},
  year = {2017},
  issn = {0944-2669},
  doi = {10.1007/s00526-017-1204-y},
  mrclass = {35R02 (35A01 35K91 58J35)},
  mrnumber = {3688855},
  mrreviewer = {Guilherme Mazanti}
}

@article{LY21,
  author = {Lin, Yong and Yang, Yunyan},
  title = {A heat flow for the mean field equation on a finite graph},
  journal = {Calculus of Variations and Partial Differential Equations},
  shortjournal = {Calc. Var. Partial Differential Equations},
  volume = {60},
  number = {6},
  pages = {Paper No. 206, 15},
  year = {2021},
  issn = {0944-2669},
  doi = {10.1007/s00526-021-02086-3},
  mrclass = {35R02 (34B45)},
  mrnumber = {4305428}
}

@incollection{Mas09,
  author = {Masamune, Jun},
  title = {A {L}iouville property and its application to the {L}aplacian
                  of an infinite graph},
  booktitle = {Spectral analysis in geometry and number theory},
  series = {Contemp. Math.},
  volume = {484},
  pages = {103--115},
  publisher = {Amer. Math. Soc., Providence, RI},
  year = {2009},
  doi = {10.1090/conm/484/09468},
  mrclass = {05C50 (05C63 39A70 58J05)},
  mrnumber = {1500141},
  mrreviewer = {J\'{o}zef Dodziuk}
}

@article{milatovic2015maximal,
  author = {Milatovic, Ognjen and Truc, Francoise},
  title = {Maximal Accretive Extensions of Schr{\"o}dinger Operators on Vector Bundles over Infinite Graphs},
  journal = {Integral Equations and Operator Theory},
  shortjournal = {Integral Equations Operator Theory},
  volume = {81},
  pages = {35--52},
  year = {2015},
  doi = {10.1007/s00020-014-2196-z}
}

@incollection{neidhardt2018operator,
  author = {Neidhardt, Hagen and Stephan, Artur and Zagrebnov, Valentin
                  A.},
  title = {Operator-norm convergence of the {T}rotter product formula on
                  {H}ilbert and {B}anach spaces: a short survey},
  booktitle = {Current research in nonlinear analysis},
  series = {Springer Optim. Appl.},
  volume = {135},
  pages = {229--247},
  publisher = {Springer, Cham},
  year = {2018},
  doi = {10.1007/978-3-319-89800-1_9},
  mrclass = {47D06 (47B25)},
  mrnumber = {3793268},
  mrreviewer = {Sabina Milella}
}

@incollection{Sch20,
  author = {Schmidt, Marcel},
  title = {On the existence and uniqueness of self-adjoint realizations
                  of discrete (magnetic) {S}chr\"{o}dinger operators},
  booktitle = {Analysis and geometry on graphs and manifolds},
  series = {London Math. Soc. Lecture Note Ser.},
  volume = {461},
  pages = {250--327},
  publisher = {Cambridge Univ. Press, Cambridge},
  year = {2020},
  doi = {10.1017/9781108615259.012},
  mrclass = {05C50},
  mrnumber = {4412977}
}

@article{Sch20b,
  author = {Schmidt, Marcel},
  title = {A note on reflected {D}irichlet forms},
  journal = {Potential Analysis},
  shortjournal = {Potential Anal.},
  volume = {52},
  number = {2},
  pages = {245--279},
  year = {2020},
  issn = {0926-2601},
  doi = {10.1007/s11118-018-9745-z},
  mrclass = {31C25 (60J46)},
  mrnumber = {4064320},
  mrreviewer = {Wenjie Sun}
}

@article{Tor10,
  author = {Torki-Hamza, Nabila},
  title = {Laplaciens de graphes infinis ({I}-graphes) m\'{e}triquement
                  complets},
  journal = {Confluentes Mathematici},
  shortjournal = {Confluentes Math.},
  volume = {2},
  number = {3},
  pages = {333--350},
  year = {2010},
  issn = {1793-7442},
  doi = {10.1142/S179374421000020X},
  mrclass = {05C63 (05C12 05C50 35R02 47B25)},
  mrnumber = {2740044}
}

@book{vazquez2007porous,
  author = {V{\'a}zquez, Juan Luis},
  title = {The Porous Medium Equation: Mathematical Theory},
  publisher = {Oxford University Press},
  year = {2007},
  isbn = {9780198569039},
  doi = {10.1093/acprof:oso/9780198569039.001.0001}
}

@article{Web10,
  author = {Weber, Andreas},
  title = {Analysis of the physical {L}aplacian and the heat flow on a
                  locally finite graph},
  journal = {Journal of Mathematical Analysis and Applications},
  shortjournal = {J. Math. Anal. Appl.},
  volume = {370},
  number = {1},
  pages = {146--158},
  year = {2010},
  issn = {0022-247X},
  doi = {10.1016/j.jmaa.2010.04.044},
  mrclass = {35Q79 (35A08 35B30 35B50 47B25)},
  mrnumber = {2651136},
  mrreviewer = {Song Jiang}
}

@book{Woj08,
  author = {Wojciechowski, Rados{\l}aw Krzysztof},
  title = {Stochastic completeness of graphs},
  publisher = {ProQuest LLC, Ann Arbor, MI},
  year = {2008},
  pages = {87},
  isbn = {978-0549-58579-4},
  mrclass = {Thesis},
  mrnumber = {2711706},
  note = {Thesis (Ph.D.)--City University of New York}
}

@article{Woj09,
  author = {Wojciechowski, Rados{\l}aw K.},
  title = {Heat kernel and essential spectrum of infinite graphs},
  journal = {Indiana University Mathematics Journal},
  shortjournal = {Indiana Univ. Math. J.},
  volume = {58},
  number = {3},
  pages = {1419--1442},
  year = {2009},
  issn = {0022-2518},
  doi = {10.1512/iumj.2009.58.3575},
  mrclass = {35K08 (05C50 35P05)},
  mrnumber = {2542093}
}

@article{Woj21,
  author = {Wojciechowski, Rados{\l}aw K.},
  title = {Stochastic completeness of graphs: bounded {L}aplacians,
                  intrinsic metrics, volume growth and curvature},
  journal = {The Journal of Fourier Analysis and Applications},
  shortjournal = {J. Fourier Anal. Appl.},
  volume = {27},
  number = {2},
  pages = {Paper No. 30, 45},
  year = {2021},
  issn = {1069-5869,1531-5851},
  doi = {10.1007/s00041-021-09821-6},
  mrclass = {58J35 (05C63 31C12 31C20 60J35)},
  mrnumber = {4240786}
}
\end{document}